\documentclass[11pt,reqno]{amsart}
\usepackage{amsmath,amsfonts,amssymb,amscd,amsthm,amsbsy,bbm, epsf,calc,  pdfsync, comment, caption}
\usepackage{color}
\usepackage[usenames,dvipsnames]{xcolor}
\usepackage{datetime}

\usepackage{thmtools, thm-restate}
\usepackage{thm-patch, thm-kv, thm-autoref}
\usepackage[colorlinks=true, urlcolor=blue, linkcolor=blue, citecolor=black]{hyperref}

\usepackage[pdftex]{graphicx}
\usepackage{wrapfig}
\usepackage{geometry}
\numberwithin{equation}{section}
\newcounter{hours}\newcounter{minutes}

\theoremstyle{plain}
\declaretheorem[title=Theorem, parent=section]{thm}
\declaretheorem[title=Lemma,sibling=thm]{lem}
\declaretheorem[title=Corollary,sibling=thm]{cor}
\declaretheorem[title=Proposition,sibling=thm]{prop}
\declaretheorem[title=Definition,sibling=thm]{DEF}
\declaretheorem[title=Remark,sibling=thm]{rem}
\declaretheorem[title=Question,sibling=thm]{question}

\newtheorem*{prop*}{Proposition}

\def\D{\mathcal D}
\def\I{\mathcal I}
\def\K{\mathcal K}
\def\L{\mathcal L}

\def\real{{\mathbb R}}

\def\Indicator{{\mathbbm{1}}}

\def\ep{\varepsilon}

\def\del{\delta}
\def\om{\omega}
\def\Om{\Omega}
\def\gam{\gamma}

\def\grad{\nabla}
\def\diam{\textnormal{diam}}
\def\hull{\textnormal{hull}}

\def\Tr{\textnormal{tr}}
\def\Id{\textnormal{Id}}

\def\union{\cup}
\def\Union{\bigcup}

\def\polhk#1{\setbox0=\hbox{#1}{\ooalign{\hidewidth
	    \lower1.5ex\hbox{`}\hidewidth\crcr\unhbox0}}}

\newcommand{\abs}[1]{\left| #1 \right|}

\newcommand{\norm}[1]{\lVert#1\rVert}

\begin{document}

\title{Min-max formulas for nonlocal elliptic operators}

\author{Nestor Guillen}
\author{Russell W. Schwab}

\address{Department of Mathematics\\
University of Massachusetts, Amherst\\
Amherst, MA  01003-9305}
\email{nguillen@math.umass.edu}

\address{Department of Mathematics\\
Michigan State University\\
619 Red Cedar Road \\
East Lansing, MI 48824}
\email{rschwab@math.msu.edu}

\begin{abstract}
  In this work, we give a characterization of Lipschitz operators on spaces of $C^2(M)$ functions (also $C^{1,1}$, $C^{1,\gam}$, $C^1$, $C^\gam$) that obey the global comparison property-- i.e. those that preserve the global ordering of input functions at any points where their graphs may touch, often called ``elliptic'' operators. Here $M$ is a complete Riemannian manifold. In particular, we show that all such operators can be written as a min-max over linear operators that are a combination of drift-diffusion and integro-differential parts.  In the \emph{linear} (and nonlocal) case, these operators had been characterized in the 1960's, and in the \emph{local, but nonlinear} case-- e.g. local Hamilton-Jacobi-Bellman operators-- this characterization has also been known and used since approximately since 1960's or 1970s.  Our main theorem contains both of these results as special cases. It also shows any nonlinear scalar elliptic equation can be represented as an Isaacs equation for an appropriate differential game. Our approach is to ``project'' the operator to one acting on functions on large finite graphs that approximate the manifold, use non-smooth analysis to derive a min-max formula on this finite dimensional level, and then pass to the limit in order to lift the formula to the original operator.  This is the Director's cut, and it contains extra details for our own sanity.
\end{abstract}

\date{\today, arXiv Ver 2 : Director's cut}
\thanks{The work of N. Guillen was partially supported by NSF DMS-1201413.  R. Schwab thanks Moritz Kassmann for introducing him to Courr\`ege's theorem.}

\maketitle
\markboth{N. Guillen, R. Schwab}{A min-max formula for nonlinear elliptic operators}


\section{Introduction and Background}\label{sec:Introduction}
\setcounter{equation}{0}

Consider a Lipschitz map $I:C^2_b(\mathbb{R}^d) \to C_b(\mathbb{R}^d)$ with the property that given any functions $u,v\in C^2_b(\mathbb{R}^d)$ and a fixed $x\in \mathbb{R}^d$ such that $u\leq v$ everywhere with $u(x)=v(x)$, then
\begin{align*}
  I(u,x)\leq I(v,x).	
\end{align*}	
Such a map is said to satisfy the \emph{global comparison property} (GCP).  Some of the most basic and frequently encountered maps with the GCP might be 
\begin{align*}
	I(u,x) = |\grad u(x)|,\ 
	I(u,x) = \Delta u(x),\ 
	\text{and}\ 
	I(u,x) = \max_{A^a\geq0}\left( \Tr(A^{a}D^2u(x)) \right)
\end{align*}
in the local case,
or 
\begin{align*}
	I(u,x) = \int_{\real^d}\left(u(x+h)-u(x)\right)K(x,h)dh,\ \text{with}\ K(x,h)\geq 0,
\end{align*}
in the nonlocal case.  For these and similar operators (e.g. general integro-differential or drift diffusion operators), it is straightforward to confirm the GCP because it follows immediately from their explicit formulas.

In this work, we prove a result in the reverse direction, i.e. we show that any (nonlinear) Lipschitz map $I$ with the GCP, plus minor and reasonable technical assumptions, has a representation as a min-max of L\'evy operators similar to those mentioned above, as presented in Theorem \ref{thm: Min-Max Riemannian Manifold}.  That is, $I$ can be written as
\begin{align}\label{eqn:intro min-max formula}
  I(u,x) = \min\limits_{a}\max\limits_{b} \;\{ f^{ab}(x) + L^{ab}(u,x)\},
\end{align}
where each $L^{ab}$ is an operator of L\'evy type, meaning that
\begin{align*}
L^{ab}(u,x) & = \Tr(A^{ab}(x)D^2u(x))+B^{ab}(x)\cdot \nabla u(x)+C^{ab}(x)u(x)\nonumber\\
   & \;\;\;\;+ \int_{\mathbb{R}^d\setminus\{0\}}u(x+y)-u(x)-\Indicator_{B_{r_0}(0)}(y)\nabla u(x)\cdot y \;\mu^{ab}_x(dy),	
\end{align*}	
where $f^{ab},A^{ab},B^{ab},C^{ab}\in L^\infty(\mathbb{R}^d)$ are Borel functions (with norms uniform in $ab$), $A^{ab}\geq 0$, and $\mu^{ab}_x$ are Borel measures on $\mathbb{R}^d\setminus\{0\}$ such that
\begin{align*}
  \sup_{ab}\sup\limits_x\int_{\mathbb{R}^d\setminus\{0\}}\min\{1,|y|^2\}\mu^{ab}_x(dy)<\infty.
\end{align*}
The setting of the main result is more general, and it covers operators $I:C^{2}_b(M)\to C_b(M)$ where $M$ is a complete Riemannian manifold, see Section \ref{subsec: Main Results} for a full description.  Such a min-max characterization for nonlocal, nonlinear operators has been relatively widely known as an open problem in the field of nonlocal equations for a few years, and min-max representations play a fundamental role in many results, which we mention in the Background and Existing Results, Sections \ref{sec:Background}, \ref{sec:ExamplesThatUseMinMaxLocal}, and \ref{sec:ExamplesAssumeMinMaxNonlocal}.

\textbf{Example (Dirichlet to Neumann Maps)} An important class of examples is given by the Dirichlet to Neumann maps for fully nonlinear elliptic equations. Consider, for instance, a bounded domain $\Omega$ with a $C^2$ boundary.  Under mild assumptions on $F$, the Dirichlet problem
\begin{align*}
  \left \{ \begin{array}{rll}	
    F(D^2U) & = 0 & \textnormal{ in } \Omega\\
    U & = u & \textnormal{ on } \partial \Omega
  \end{array}  \right.
\end{align*}
has a unique viscosity solution $U \in C^{1,\alpha}(\bar \Omega)$, whenever $u \in C^{2}(\partial \Omega)$ (for some $\alpha>0$ independent of $u$). This defines a map
\begin{align*}
  I:C^2(\partial \Omega) \to C(\partial \Omega)
\end{align*}
obtained by setting $I(u,x):= (\nabla U(x),n(x))$ (i.e. $\partial_n U$), where $n$ is the inner normal to $\partial \Omega$ at $x$. Using the comparison principle for $F$, it is straightforward to see that this map $I$ has the global comparison property, and boundary regularity theory for $U$ shows that the mapping is indeed Lipschitz.  In particular, our main result applies to the Dirichlet to Neumann map, even for nonlinear equations. In a forthcoming paper, the min-max formula and boundary estimates for elliptic equations are used to analyze these operators in detail.

\textbf{Example (Isaacs-Bellman equations)} Given linear operators $\{L^{ab}\}_{ab}$ each satisfying the global comparison property, one may consider equations of the form
\begin{align*}
  I(u,x)  = 0,\;\;\textnormal{where } I(u,x) := \min\limits_{a}\max\limits_{b} L^{ab}(u,x).	
\end{align*}	
These are known as Isaacs-Bellman equations, and they arise in stochastic control (e.g Bellman, \cite{Bell-1957} for first order equations), or zero sum games (e.g. Isaacs \cite{Isaacs-1965Book} or Elliott-Kalton \cite{ElliottKalton-1972Book}).  The original references dealt mainly with first order equations, but second order examples quickly followed; see e.g. \cite{FlemingRishel--1975Book}. It is easy to see that such an operator must satisfy the global comparison property, as it is preserved from $L^{ab}$ through the min-max. Our main result can be seen as the converse assertion: we show that every Lipschitz operator for which the global comparison property holds corresponds to an Isaacs-Bellman equation for an appropriate family of Markov processes.

\subsection{Statement of The Main Results}\label{subsec: Main Results}

\begin{DEF}\label{def: touching from below}
  Given a set $X$ and functions $u,v:X\to\mathbb{R}$, it is said that $u$ touches $v$ from below at $x_0 \in X$ if
  \begin{align*}
    u(x) & \leq  v(x),\;\;\forall\;x\in X,\\
    u(x_0) & = v(x_0). 		
  \end{align*}	
  If the inequality is reversed, it is said that $u$ touches $v$ from above at $x_0$.	 
\end{DEF}

\begin{DEF}\label{def: operators with maximum principle}
  Consider a set $X$ and let $\mathcal{F}\subset \mathbb{R}^X $ be a class of real valued functions defined over $X$. Given a (possibly nonlinear) operator
  \begin{align*}
    I: \mathcal{F} \subset \mathbb{R}^X \to \mathbb{R}^X, 	  
  \end{align*}	  	
  $I$ is said to satisfy the \textbf{global comparison property (GCP)} if whenever $u \in \mathcal{F}$ touches $v \in \mathcal{F}$ from below at $x_0$ we have the inequality
  \begin{equation*}
    I(u,x_0)\leq I(v,x_0).	
  \end{equation*} 
\end{DEF}

\begin{rem}\label{rem: GCP is convex and closed} It is clear that the set of maps having the global comparison property is convex and closed with respect to $(u,x)$-pointwise limits, i.e. for limits $I_n\to I$ in the sense that
	\begin{align*}
		\lim_{n\to 0} I_n(u,x) = I(u,x)\ \forall\ u\in\mathcal{F},\ \text{and}\ \forall x\in X.
	\end{align*} 
	
\end{rem}

Our goal is to prove a representation theorem for nonlinear operators with the GCP.  In order to include examples such as the nonlinear Dirichlet-to-Neumann mapping mentioned above, the main result necessarily deals the case that $X$ is a Riemannian manifold  (in that example, $X=\partial\Om$).

\begin{DEF}\label{def:LevyTypeFunctional}
  Let $(M,g)$ be a $d$-dimensional $C^3$ Riemannian manifold with injectivity radius $r_0>0$, let $\exp_x$ denote the exponential map based at $x \in M$, and fix some $x\in M$.  A linear functional, $L_x\in (C^\beta_b(M))^*$, is said to be a functional of L\'evy type based at $x\in M$ if $L_x(u)$ has the following form
\begin{align}\label{eqIntro:LevyTypeDef}
  L_x(u) & = \Tr(A \nabla^2u(x))+(B, \nabla u(x))_{g_x}+Cu(x)\nonumber\\
   & \;\;\;\;+ \int_{M\setminus\{x\}}u(y)-u(x)-\Indicator_{B_{r_0}(x)}(y)(\nabla u(x),\exp_{x}^{-1}(y))_{g_x} \;\mu(dy),	
\end{align}
where $A:(TM)_x\to (TM)_x$ is a linear self-adjoint map such that $A\geq 0$, $B\in (TM)_x$, $C\in\mathbb{R}$ and $\mu$ is a Borel measure in $M\setminus \{x\}$ such that
\begin{align*}
  \int_{M\setminus \{x\}}\min\{1,d(x,y)^2\} \;\mu(dy).	
\end{align*}	
\end{DEF}

\begin{rem}
  When $M$ is given by the Euclidean space $\mathbb{R}^d$, the exponential at $x$ mapping simply becomes $y\to x+y$, and so the last term in \eqref{eqIntro:LevyTypeDef} takes the more commonly seen form of
  \begin{align*}
    \int_{\mathbb{R}^d\setminus \{x\} } u(y)-u(x)- \left ( \nabla u(x), y-x\right )\Indicator_{B_{1}(x)}(y) \;\mu(dy),
  \end{align*}	
  where $\mu$ is a Borel measure in $\mathbb{R}^d\setminus \{x\}$ such that
  \begin{align*}
    \int_{\mathbb{R}^d\setminus\{0\}} \min\{1,|x-y|^2\}\;\mu(dy).	  
  \end{align*}	  
\end{rem}
We are now ready to state our main results.
\begin{thm}\label{thm: Min-Max Riemannian Manifold}
  Let $(M,g)$ be as in Definition \ref{def:LevyTypeFunctional}, and let $\beta \in [0,2]$.
  Let $C^\beta_b(M)$ be one of the Banach spaces $C^{0,\beta}_b(M)$  or $C^{0,1}(M)$ if $\beta\in(0,1)$; $C^1_b(M)$ if $\beta=1$;  $C^{1,\beta-1}_b(M)$ if $\beta\in(1,2)$; $C^{1,1}(M)$ or $C^2_b(M)$ if $\beta=2$.  Let $I$
  \begin{align*}
    I:C^\beta_b(M)\to C_b(M)	  
  \end{align*}	  
  be a Lipschitz map having the global comparison property, and that satisfies the additional assumption that there is a modulus, $\om$ with $\om(r)\to 0$ as $r\to\infty$, such that for all $r$ large enough,
  \begin{align}\label{eqIntro:AssumptionQuatifiedLocalUnifConv1}
	  \forall u,v\in C^\beta_b,\ \ \norm{I(u)-I(v)}_{L^\infty(B_r)}
	  \leq C\norm{u-v}_{C^\beta(\overline{B_{2r}})} + C\om(r)\norm{u-v}_{L^\infty(M)}.
  \end{align}
 Then, $I$ has the following min-max representation (proved in Section \ref{sec: Min-Max formula for M,g})
  \begin{align}\label{eqIntro:MinMax!}
    I(u,x) = \min_{v\in C^\beta_b}\max_{L_x\in \K_{Levy}(I)}\{ I(v,x)+L_x(u-v)\},
  \end{align}
  where $\K_{Levy}(I)$ is a collection of L\'evy type linear functionals on $C^\beta_b(M)$, as in Definition \ref{def:LevyTypeFunctional} and \eqref{eqIntro:LevyTypeDef}. Moreover, the norm of each $L_x$ is bounded by the Lipschitz norm of $I$.  The formula \eqref{eqIntro:MinMax!} holds for $u$ in different spaces, depending upon the domain of $I$.  The cases are for respectively the domain of $I$ and the type of $u$ for which \eqref{eqIntro:MinMax!} holds are: domain is $C^2_b$, $u\in C^2_b$; domain is $C^{1,1}$, $u\in C^2_b$; domain is $C^{1,\gam}_b$, $u\in C^{1,\gam+\ep}_b$ for any $0<\ep<1-\gam$; domain is $C^1_b$, $u\in C^{1,\ep}_b$ for any $0<\ep<1$; domain is $C^\gam_b$, $u\in C^{\gam+\ep}$ for any $0<\ep<1-\gam$.
\end{thm}	

\begin{prop}\label{prop:CasesOnBetaForMinMax}
	In the min-max formula of \eqref{eqIntro:MinMax!}, not only do the functionals $L_x$ have the L\'evy-type form of \eqref{eqIntro:LevyTypeDef}, but they also reduce to simpler cases on $\beta$ as follows:
	\begin{enumerate}
		\item if $\beta=2$ or $C^\beta=C^{1,1}(M)$, then all terms in \eqref{eqIntro:LevyTypeDef} may be present;
		\item if $\beta\in [0,2)$, excluding the case $C^{1,1}(M)$, but including the cases of $C^1$ and $C^{0,1}$, then $A^{ab}\equiv 0$ for all $x\in M$;
		\item if $\beta\in [0,1)$ excluding the case $C^{0,1}$, then both $A^{ab}\equiv 0$ and $B^{ab}\equiv 0$ for all $x\in M$. 
	\end{enumerate}
\end{prop}

A stronger version of the min-max holds if one imposes a further assumption on $I$, 
\begin{align}\label{eqIntro:Equicontinuity Assumption}
  \forall\; \mathcal{K}\subset\subset C^\beta_b(M),\;\;\textnormal{the family }  \left \{ x\to \frac{I(v+u,x)-I(v,x)}{\|u\|_{C^\beta_b(M)}} \right \}_{u,v\in \mathcal{K}}\;\textnormal{ is equicontinuous} 
\end{align}
This assumption is satisfied if one assumes that $I$ is a Lipschitz map from  $C^{\beta}_b$ to the H\"older space $C^\alpha_b$  (for any $\alpha>0$), or even a space $C^{\omega}_b$, where $\omega$ is some modulus of continuity. 

\begin{thm}\label{thm: Min-Max Riemannian Manifold With Operators}
  Suppose that, in addition to the assumptions of Theorem \ref{thm: Min-Max Riemannian Manifold}, the operator $I$ satisfies \eqref{eqIntro:Equicontinuity Assumption}. Then, there is a family $\mathcal{L}$ of linear operators from $C^\beta_b$ to $C_b$, such that
  \begin{align}
    I(u,x) = \min_{v\in C^\beta_b}\max_{L\in \mathcal{L}}\{ I(v,x)+L(u-v,x)\},
  \end{align}
  Furthermore, for each $x\in M$, the functional defined by $L(\cdot,x)$ belongs to the same class of functionals $\K_{Levy}(I)$ above.
\end{thm}

\begin{rem}
	In the case that $I$ is linear, Theorem \ref{thm: Min-Max Riemannian Manifold} was shown by Courr\'ege, for $M = \mathbb{R}^d$ {\cite[Theorem 1.5]{Courrege-1965formePrincipeMaximum}}, and by Bony-Courr\`ege-Priouret for an arbitrary $d$-dimensional manifold \cite{BonyCourregePriouret-1966FellerSemiGroupOnDifferentialVariety}.  In fact, those works showed the result holds simply when $L$ is a continuous linear operator from $C^2$ to $C$, endowed with the non-Banach space topology of local uniform convergence on compact sets.
\end{rem}

\begin{rem}
  If the operator $I$ in Theorem \ref{thm: Min-Max Riemannian Manifold} is \emph{convex}, then the min-max formula simplifies to a max formula, see Lemma \ref{lem: max formula convex operator}, and if $I$ is linear, then there is no min-max.
\end{rem}

\begin{rem}
	As suggested by the result of Theorem \ref{thm: Min-Max Riemannian Manifold}, the GCP imposes significant structure on $I$.  A good example of this, is that in fact $I$ must depend on the $C^\beta_b$-norm in a very particular way.  For example, one possible estimate that can be shown (not exactly the one we use, but illustrative enough) is for a fixed $x$,
	\begin{align*}
		\abs{I(u,x)-I(v,x)}\leq C(R)\norm{I}_{Lip(C^\beta_b,C_b)}\left( 
		 \norm{u-v}_{C^\beta(B_R(x))} + \norm{u-v}_{L^\infty(M)}
			 \right).
	\end{align*} 
	A similar type of splitting of the estimate on the right hand side between $C^\beta_b$ and $L^\infty$ turns out to be fundamental to our method, and we explain it in detail in Section \ref{sec:NicePropertiesIInPin}.  We note for the reader familiar with the integro-differential theory that if $I$ were already known to be of the L\'evy form \eqref{eqIntro:LevyTypeDef}, then this decomposition is immediate for $\beta=2$ (also for operators that are a min-max of \eqref{eqIntro:LevyTypeDef} with uniform bounds on the ingredients).
\end{rem}


\subsection{Background}\label{sec:Background} 

There are several precedents for this result. It was shown by Courr\`ege \cite{Courrege-1965formePrincipeMaximum} that a bounded linear operator $C^2(\real^d)\to C(\real^d)$ has the global comparison property if and only if it is of \emph{L\'evy type}, in \eqref{eqIntro:LevyTypeDef}, which was later extended to linear operators on functions in a manifold $M$ in work of Bony, Courr\`ege, and Priouret \cite{BonyCourregePriouret-1966FellerSemiGroupOnDifferentialVariety}. A related result by Hsu \cite{Hsu-1986ExcursionsReflectingBM} provides a representation for the Dirichlet to Neumann map for the Laplacian in a smooth domain $\Omega$, and this corresponds to studying the boundary process for a reflected Brownian motion.  After a time rescaling, the boundary process is a pure-jump L\'evy process on the boundary, and it's generator is of the form
\begin{align*}
  L(u,x) & = b(x)\cdot \nabla_{\tau} u(x)+\int_{\partial \Omega \setminus \{x\}} \left ( u(y)-u(x) -\Indicator_{B_1(x)}(y)\nabla_\tau u(x)\cdot (y-x)\right ) k(x,y)\;d\sigma(y), 
\end{align*}
where $\nabla_\tau$ denotes the tangential gradient, $b(x)$ is a tangent vector field to $\partial \Omega$, $\sigma$ is the surface measure, and $k$ is comparable to $\abs{x-y}^{-d-1}$ for $\abs{x-y}$ small. An interesting family of nonlocal operators on Riemannian manifolds are the fractional Paneitz operators, which are also conformally invariant; recently, such operators have been studied in relation to Dirichlet to Neumann maps by Chang and Gonzalez \cite{ChGo-2011FracLaplaceGeometry} and Case and Chang \cite{CaseChang2016fractionalGJMS}; these linear operators satisfy the GCP, under certain curvature conditions. A related (nonlinear) Dirichlet to Neumann operator arising in conformal geometry is the boundary operator for the fully nonlinear Yamabe problem on manifolds with boundary \cite{LiLi2006FullyYamabeBoundary}.

 If $I$ is not necessarily linear but happens to satisfy the stronger \emph{local} comparison principle,  there are min-max results by many authors, e.g. Evans \cite{Evans-1984MinMaxRepresentations}, Souganidis \cite{Souganidis-1985MaxMinRep}, Evans-Souganidis \cite{EvansSoug-84DiffGameRepresentation} and Katsoulakis \cite{Katsou-1995RepresentationDegParEq}. In this case, the operator takes the form,
\begin{align*}
  I(u,x) = F(x,u(x),\nabla u(x),D^2u(x)),	
\end{align*}	
where $F: \mathbb{R}^d\times\mathbb{R}  \times \mathbb{R}^d\times \textnormal{Sym}(\mathbb{R}^d) \to \mathbb{R}$ can be expressed as
\begin{align*}
  F(x,u,p,M) = \min\limits_{a}\max \limits_{b} \{ \Tr(A^{ab}(x)M) + B^{ab}(x)\cdot p + C^{ab}(x)u+f^{ab}(x) \}.
\end{align*}
This was extended to even include the possibility of weak solutions acting as a \emph{local} semi-group on $BUC(\real^d)$, related to image processing, in Alvarez-Guichard-Lions-Morel \cite{AlvarezLionsGuichardMorel-1993AxiomsImageProARMA}, and to weak solutions of sets satisfying an order preserving set flow by Barles-Souganidis in \cite{BarlesSouganidis-1998NewApproachFrontsARMA}.  In \cite{AlvarezLionsGuichardMorel-1993AxiomsImageProARMA} it was shown under quite general assumptions that certain nonlinear semigroups must be represented as the unique viscosity solution to a degenerate parabolic equation.   Recent work of Gilboa and Osher \cite{GiOs-2008NLImagePro} has explored the practical advantages of image processing algorithms that are not local.
Thus, the family of nonlinear \emph{local} elliptic operators has a simple description, and hence the representation of Lipschitz operators in the local case is more or less complete.  So far, very little has been said about operators that don't necessarily have the local comparison principle, but only the weaker version that is the GCP (i.e. operators containing a nonlocal part).



\subsection{Some examples of the advantage of a min-max}\label{sec:ExamplesThatUseMinMaxLocal}

Using an equation that involves a min-max of linear operators of course goes back to studying differential games, where the equation gives information about the value and strategies of the game.  However, here we briefly list some results where the flow of information is reversed: beginning with a nonlinear PDE, some results are more easily (or only) attainable after the solutions (sub or super solutions) are represented as value functions for certain differential games-- via the dynamic programming principle.  Some very early results on existence for solutions to nonlinear first order equations utilized the properties of the value function in a stochastic differential game and the vanishing viscosity method in Fleming \cite{Fleming-1969CauchyProbNonlinFirstOrder-UsesMinMax} and Friedman \cite{Friedman-1974TheCauchyProbFirstOrderEquations-IUMJ}.  Also, solving some similar nonlinear equations, the accretive operator method of Evans \cite{Evans-1980OnSolvingPDEAccretive}  utilized a convenient min-max structure.
More refined properties of Hamilton-Jacobi equations, such as ``blow-up'' limits appear in Evans-Ishii \cite{EvansIshii-1984DiffGamesNonlinearPDE-ManMath} and inequalities for directional derivatives of solutions in Lions-Souganidis \cite{LionsSouganidis-1985DiffGamesDerivOfSol}.    Applications to the structure of level sets, geometric motions, ``generalized'' characteristics, and finite domain/cone of dependence appear in Evans-Souganidis \cite{EvansSoug-84DiffGameRepresentation}.    Some constructions of finite difference schemes in e.g. Kuo-Trudinger \cite{KuoTru-2007NewMaxPrinciplesIUMJ} utilized the fact that second order uniformly elliptic equations are necessarily a min max of linear operators in order to choose appropriate stencil sizes; and a min-max was used by Krylov \cite{Krylov-2015RatesFinDiffIsaacs} to produce a rate of convergence for some approximation schemes.   The Lions-Papanicolaou-Varadhan preprint for homogenization of Hamilton-Jacobi equations \cite{LiPaVa-88unpublished} used the fact that any semigroup with the properties inherited by the homogenized limit must be a translation invariant Hamilton-Jacobi semigroup of viscosity solutions-- a result very close in spirit to the one we show for nonlocal equations (see \cite[Section 1.2]{LiPaVa-88unpublished} and work of Lions cited therein).  Katsoulakis \cite{Katsou-1995RepresentationDegParEq} used a min-max to leverage the value function of a stochastic differential game to show existence of a viscosity solution and its Lipschitz/H\"older regularity properties.  More recently, Kohn-Serfaty exploited a min-max structure to make a link between solutions of fully nonlinear second order parabolic equations and a class of deterministic two-player games in the papers \cite{KohnSerfaty-2006DeterministicGameMCM-CPAM} and \cite{KohnSerfaty-2010DeterministicGamesFullyNonlinearPDE-CPAM} (as opposed to the already known link with \emph{stochastic} differential games).   All of these results mentioned above are solely in the context of local equations.


\subsection{Nonlocal results that assume a min-max}\label{sec:ExamplesAssumeMinMaxNonlocal}
One of the reasons why there is such a strong link between nonlinear elliptic PDE and min-max formulas associated with differential games is that it turns out the property of being a unique viscosity solution of such an equation is more or less equivalent to satisfying a dynamic programming principle/equation.  Thus, it is natural that even though in the nonlocal setting, no min-max formula for general operators was known to exist, many results assume their operators to have a min-max structure.  Some of these examples are as follows.  Some uniqueness theorems for viscosity solutions (weak solutions) to somewhat general nonlinear and nonlocal equations assume the operator to have a min-max structure in both Jakobsen-Karlsen \cite{JakobsenKarlsen-2006maxpple} and Barles-Imbert \cite{BaIm-07}. Caffarelli-Silvestre \cite[Sections 3 and 4]{CaSi-09RegularityIntegroDiff} assume a min-max structure of their equations in proving some properties of viscosity solutions-- but the main result of the paper, \cite{CaSi-09RegularityIntegroDiff} does not make a min-max assumption.  Silvestre \cite{Silv-2011DifferentiabilityCriticalHJ} assumes the min-max in proving regularity results for critical nonlocal equations, where the nonlocal term is of order 1, the same as the drift.  One of the authors in \cite{Schw-10Per} and \cite{Schw-12StochCPDE} assumes the nonlocal operators to have a min-max so as to be able to set-up a corrector equation in homogenization for some nonlocal problems.  Furthermore, in \cite{Schw-10Per} and \cite{Schw-12StochCPDE} a homogenized limit equation is proved to exist, but it is only known as an abstract nonlinear nonlocal operator of a certain ellipticity class, and its precise structure is left as an unresolved question.  Also, in connection to the known results for local Hamilton-Jacobi equations, Koike-{\'S}wi{\polhk{e}}ch \cite{KoikeSwiech-2013RepFormulaIntegroPDE-IUMJ} showed that the value function for some stochastic differential games driven by L\'evy noise is indeed the unique viscosity solution of the related nonlocal Isaacs equation. Thus, Theorem \ref{thm: Min-Max Riemannian Manifold} in our current work can be seen as a sort of a posteriori justification for the existing min-max assumptions in the nonlocal literature.



\subsection{Notation}\label{subsec: Notation}
Here we collect a table of notation that is used throughout the work.

\begin{tabular}{lll} 
\allowdisplaybreaks
Notation  & Definition \\
\hline\\
$M$ & Complete Riemannian manifold\\
$d$ & dimension of $M$\\
$d(x,y)$ & Geodesic distance on $M$\\
$TM$, $(TM)_x$ & The tangent bundle to $M$ and the tangent space at $x \in M$\\
$\exp_x$ & The exponential map of the manifold $M$\\
$r_0$ & a lower bound for the injectivity radius of $M$\\
$Q,Q',\ldots$ & cubes in some tangent space $(TM)_x$\\
$Q^*$ & cube concentric with $Q$ whose common length is increased by a factor of $9/8$\\
$\nabla^2 u(x)$ & the Hessian of $u$ over $M$\\
$\nabla_{a}$, $\nabla_{ab}$ & components of covariant derivatives on $M$ w.r.t. a chart (e.g \cite{Heb-2000NonlinearAnalysisManifoldsBook})\\
$\nabla^1_n u(x)$ & a discrete gradient over the finite set $\tilde G_n$\\
$\nabla^2_n u(x)$ & a discrete Hessian over the finite set $\tilde G_n$\\
$C_b(M)$ & functions which are continuous and bounded in $M$, with the sup-norm\\
$C^2_b(M) $ & functions for which $\nabla^2u$ is continuous and bounded in $M$, with the sup-norm\\
$C^\beta_b(M)$ & Any of: $C^2_b(M)$, $C^{1,1}(M)$, $C^{1,\beta-1}_b(M)$ if $1\leq \beta<2$, $C^{0,1}(M)$, or $C^{0,\beta}_b(M)$ if $\beta<1$\\
$C^\beta_c(M)$& functions in $C^\beta_b(M)$ that have compact support \\
$X_n^\beta(M)$ & finite dimensional subspace of $C^\beta_b(M)$ given by a Whitney extension\\
$\L(X,Y)$ & space of linear operators\\
$\textnormal{hull}(E)$ & convex hull of the set $E$\\
$l(p,x;y)$ & a ``linear'' function with gradient $p$, centered at $x$ (Def \ref{def:LinearAndQuadPolynomials})\\
$q(D,x;y)$ & a ``quadratic'' function with Hessian $D$, centered at $x$ (Def \ref{def:LinearAndQuadPolynomials})\\
$p^\beta_{u,k}$ & a ``polynomial'' approximation to $u$ using $l$ and $q$ (eq \eqref{eqWhitney:DefOfPbetaU})\\
$\rho$ & a smooth approximation to $\min(t,1)$, can be fixed for the entire work (Def \ref{def:RhoTheSmoothApproxOfMin(1,t)})\\
$\eta^\del_x$, $\tilde\eta^\del_x$ & smooth approximations to $\Indicator_{B_{r_0}(x)}$ and $\Indicator_{\{x\}}$ (Def \ref{def:EtaAndTildeEtaDef}).
\end{tabular}

\subsection{Outline of the rest of the paper} In Section \ref{sec: Min-Max finite dimension} we prove a ``finite dimensional'' version of Theorem \ref{thm: Min-Max Riemannian Manifold} for operators acting on functions defined on a finite graph. In Section \ref{sec: finite dimensional approximations} and Section \ref{sec: Min-Max formula for M,g} we use finite dimensional approximations to extend the min-max formula to the case of a Riemannian manifold, proving Theorem \ref{thm: Min-Max Riemannian Manifold}. Finally, in Section \ref{sec: Further Questions} we mention several reasonable questions that could be addressed and which are directly related to our main result.


\section{The min-max formula in the finite dimensional case}\label{sec: Min-Max finite dimension}
\setcounter{equation}{0}

A cornerstone of our proof relies on the fact that in the finite dimensional setting, min-max representations for Lipschitz functions are known.  Later we will produce a finite dimensional approximation to the original operator, $I$, and we will then invoke the tools from the finite dimensional setting.  Here we collect the necessary theorems we need, and present them in a context that is consistent with our subsequent application.

Consider a finite set $G$, let $C(G) = \mathbb{R}^G$ denote the space of real valued functions defined on $G$. In this finite dimensional setting the characterization of linear maps satisfying the global comparison property is elementary.  Thus, the importance of this section is not to establish a new result for Lipschitz maps, but rather to present all of the results in a way that will match our needs for extending the min-max to the infinite dimensional case.

\begin{lem}\label{lem: Finite Dimensional Courrege} 
  Any bounded linear map $L:C(G)\to C(G)$ can be expressed as follows
  \begin{align*}
	  Lu(x) = c(x)u(x)+\sum \limits_{y\in G,\ y\not=x} (u(y)-u(x))K(x,y)\;\;\;\forall\;x\in G,
  \end{align*}	  
  where $K(x,y):G \times G\to \mathbb{R},\;c:G\to \mathbb{R}$.  If it happens that $L$ also satisfies the GCP, then $K(x,y)\geq 0$ for all $x,y \in G$.
\end{lem}

\begin{proof}
  Consider the ``canonical basis'' of $C(G)$, $\{e_x\}_{x\in G}$, where for each $x\in G$,
  \begin{align*}
    e_x(y) = \begin{cases}
	1 & \textnormal{ if } x=y,\\
	0 & \textnormal{ otherwise}.
    \end{cases}	  
  \end{align*}	
  This means that we can write $u$ as
  \begin{align*}
	  u(x) = \sum_{y\in G} u(y)e_y(x). 
  \end{align*}
	Then, for a generic $u \in C(G)$, we can use the linearity of $L$ to write
  \begin{align*}
    (Lu)(x) & = \sum \limits_{y\in G}u(y)(Le_y)(x),\\
	        & = u(x)(Le_x)(x)+\sum \limits_{y\neq x}u(y)(Le_y)(x).
  \end{align*}
  This can be rewritten as follows,
  \begin{align*}
    (Lu)(x) & = u(x)(Le_x)(x)+u(x)\left (\sum \limits_{y\neq x} (Le_y)(x) \right )-u(x)\left (\sum \limits_{y\neq x} (Le_y)(x) \right )	+ \sum \limits_{y\neq x}u(y)(Le_y)(x),\\
	        & = u(x)\left ( (Le_x)(x) + \sum \limits_{y\neq x} (Le_y)(x)  \right )+\sum \limits_{y\neq x} (u(y)-u(x))(Le_y)(x).
  \end{align*}	  
  Let us define then
  \begin{align*}
	K(x,y) & := (Le_y)(x),\;\;\forall\;x,y\in G,\\  
    c(x) & : =  (Le_x)(x) + \sum \limits_{y\neq x} (Le_y)(x),\;\;\forall\;x \in G.
  \end{align*}

  Now, suppose we are in the special case that $L$ has the GCP.
  Observe that $e_x(y)\geq 0$ with $e_x(y)=0$ whenever $x\neq y$, with this in mind, and recalling that $L$ satisfies the global comparison property, it is clear that
  \begin{align*}	  
    K(x,y) = (Le_x)(y) & \geq 0,\;\;\forall\;y \not= x,
  \end{align*}
  and the lemma is proved.
\end{proof}

If $I$ is nonlinear but Lipschitz, the above characterization can be extended as a min-max formula.  We will use some machinery from nonsmooth analysis (see Clarke's book \cite{Cla1990optimization}). In particular, we will be making extensive use of the generalized Jacobian and some of its properties.  Note we give it a slightly different name than the one used in \cite{Cla1990optimization}.

\begin{DEF}[\cite{Cla1990optimization} Def 2.6.1]\label{def:ClarkeDifferential}
	For $I: C(G)\to C(G)$, the \textbf{Clarke differential of $I$ at $u$}  is defined as the set 
	\begin{align*}
		\mathcal{D}I(u) = \hull\left\{\lim_{k\to\infty} DI(u_k)\ :\textnormal{ where } u_k\to u\ \text{and}\ DI(u_k)\ \text{exists for each } k \right\}.
	\end{align*}
	Here ``$\lim$'' is simply interpreted as the limit of a sequence of matrices (since this takes place in a finite dimensional vector space), and $DI(u_n)$ is the (Fr\`echet) derivative of $I$ at $u_n$.
	Given a set $E$ in a normed vector space, ``$\hull(E)$'' denotes the smallest closed convex containing it.
\end{DEF}

It will also be convenient to have notation for the collection of all differentials:
\begin{DEF}[Full differential of $I$]\label{def:FullDifferential}
	For $I:C(G)\to C(G)$, the \textbf{full differential of $I$} is the set
	\begin{align*}
		\D I = \hull\left(\Union_{u\in C(G)} \D I(u) \right).
	\end{align*}
\end{DEF}

The main result of this section is the observation (which is more or less well known) that Lipschitz maps have a min-max structure.  We record it here in a format that is useful to our subsequent approximations to $I$.
\begin{lem}\label{lem:FiniteDimensionalMinMax}
  Let $I:C(G)\to C(G)$ be a Lipschitz map. Then, for any $u \in C(G)$ and $x\in G$,
  
  \begin{align}\label{eqFinDim:FinDimMinMaxUsingDI}
	  I(u,x) = \min_{v\in C(G)}\max_{L\in\D I}\left\{ I(v,x) + L(u-v,x)   \right\},
  \end{align}
  where $\D I$ is as in Definition \ref{def:FullDifferential}.  This can equivalently be written as
  \begin{align}\label{eqFinDim:FinDimMinMaxUsingSummation}
    I(u,x) & = \min \limits_{a}\max_{b}	\left \{f^{a}(x)+ u(x)c^{ab}(x)+\sum\limits_{y\in G,\ y\not=x} (u(y)-u(x))K^{ab}(x,y) \right \}.
  \end{align}
 	  If $I$ happens to have the GCP, then it also holds that $K^{ab}(x,y)\geq0$.

\end{lem}

We first list some key properties of $I$ before we prove Lemma \ref{lem:FiniteDimensionalMinMax}.

\begin{prop}\label{prop: global comparison property is inherited by the Clarke Jacobian}
  The GCP is inherited under differentiation. Namely, if $I:C(G)\to C(G)$ is a Lipschitz mapping that has the GCP, then the same is true of any $L:C(G)\to C(G)$ in $\D I$. 
 
\end{prop}

\begin{proof}  
  Assume first that $I$ is differentiable at $u$ and let $L_u$ denote the derivative of $I$ at $u$. Then,
  \begin{align*}
    \frac{d}{dt}|_{t=0}\left ( I(u+t\phi,x)-I(u,x)\right ) = L_u(\phi,x),\;\;\forall\;\phi\in C(G),x\in G.	  
  \end{align*}	   
  If $\phi(x)\leq 0$ for all $x$ and $\phi(x_0)=0$ for some $x_0$, it follows that (for every $t>0$) $u+t\phi$ touches $u$ from below at $x_0$, therefore (since $I$ has the GCP)
  \begin{align*}
     I(u+t\phi,x_0) & \leq I(u,x_0),\;\;\forall\;t>0\\
	 \Rightarrow L_u(\phi,x_0) & 	\leq 0.	  
  \end{align*}	  
  It follows $L_u$ has the GCP. By definition, any $L \in \mathcal{D}I$ is a convex combination of limits of such $L_u$. Then, by Remark \ref{rem: GCP is convex and closed} we conclude that any $L \in \mathcal{D}I$ also has the GCP, and the proposition is proved.  
  
\end{proof}

The following result is a very useful fact of the Clarke differential, and it shows that the differential set enjoys the mean value property.

\begin{prop}\label{prop: mean value theorem for Clarke's subdifferential}
  Let $I:C(G) \to C(G)$ be a Lipschitz function. Then, for any $u,v\in C(G)$ there exists some $L \in \mathcal{D}F$ such that
  \begin{align*}
    I(u)-I(v) = L(u-v).	  
  \end{align*}	  
\end{prop}

\begin{proof}
See \cite[Chapter 2, Proposition 2.6.5]{Cla1990optimization} for the proof.

\end{proof}

With these previous results in hand, we can now prove the main Lemma of this section.

\begin{proof}[Proof of Lemma \ref{lem:FiniteDimensionalMinMax}]
  For any $v\in C(G)$, define an operator $K_v:C(G)\to C(G)$ as follows
  \begin{align*}
    K_v(u,x) = \max \limits_{L \in \mathcal{D}I } \left \{ I(v,x)+L(u-v,x) \right \}.  
  \end{align*}	  
  First, let us show that
  \begin{align}
    I(u,x) = \min\limits_{v \in C(G)} K_v(u,x).\label{eqNonlinearCourrege: I is equal to min of cones}	  
  \end{align}	  
  Since $K_u(u,x)=I(u,x)$ for every $u$ and $x$ it holds that $ I(u,x) \geq \min\limits_{v \in C(G)} K_v(u,x)$.
  
  Next, by Proposition \ref{prop: mean value theorem for Clarke's subdifferential} it follows that for any $u,v \in C(G)$ and any $x \in G$ there exists some $L\in \mathcal{D}I$ such that
  \begin{align*}
    I(u,x) = I(v,x) + L(u-v,x).
  \end{align*}
  In particular,
  \begin{align*}
    K_v(u,x)= \max\limits_{L \in \mathcal{D}I} \{ I(v,x)+L(u-v,x)\} \geq I(u,x),	  
  \end{align*}	  
  which proves \eqref{eqNonlinearCourrege: I is equal to min of cones} and hence \eqref{eqFinDim:FinDimMinMaxUsingDI}. We note that \eqref{eqFinDim:FinDimMinMaxUsingSummation} follows by applying Lemma \ref{lem: Finite Dimensional Courrege} to each of the operators $L\in\mathcal{D}I$. 
  
\end{proof}


\section{A Whitney Extension For $C^\beta_b(M)$}\label{sec: finite dimensional approximations}

In this section, we develop some tools necessary to build finite dimensional approximations to $I$.  This will involve taking a sequence of finite sets $G_n \subset M$ ``converging'' to $M$, all while constructing an embedding map $C(G_n) \mapsto C^\beta(M)$ to approximate $C^\beta(M)$ by a finite dimensional subspace. Because we are concerned with approximations that will not corrupt too badly the Lipschitz norm of $I$, we had a natural choice to use the Whitney extension.  If we were working in $M=\real^d$, then all of the results we would need are standard, and can be found e.g. in Stein's book \cite[Chapter 6]{Stei-71}.  Unfortunately, we could find no references for these theorems for the Whitney extension on $M\not=\real^d$, and so for completeness, we provide the details here.  We emphasize that nearly all of the theorems and proofs in the section are adaptations that mirror those of Stein's book \cite{Stei-71}, but are modified for the additional technical difficulties arising due to the Riemannian nature of $M$. A key fact is how the extension operator preserves regularity (Theorem \ref{thm:EnTnContOperatorInCbetaNorm}). Along the way, we will also prove a few important lemmas: one regarding the behavior of the extension operators as $n\to\infty$ (Lemma \ref{lem:EnTnEstToUInC3} ), and a  ``corrector lemma'' that says the extensions are in general order preserving up to a small error (Lemma \ref{lem: extension operators preserve order up to small error} ).

For all of this section, $(M,g)$ is a $d$-dimensional complete Riemannian manifold with injectivity radius bounded below by a constant $r_0>0$.  We remind the reader that the choice to work on $(M,g)$ rather than $\real^d$ is not just for mathematical generality-- rather, since we intend to apply the min-max theory to the Dirichlet-to-Neumann operators of fully nonlinear equations, we must understand those operators acting on functions on $M=\partial\Om$.

\subsection{Finite approximations to $M$, coverings, and partitions of unity.} The following basic lemma will be needed.  It simply states that on a (uniform) small neighborhood of $0 \in (TM)_x$, the map $\exp_x$ is nearly an isometry.

\begin{lem}\label{lem: injectivity radius}
  Let $M$ be a complete $d$-dimensional manifold with injectivity radius $r_0>0$ and bounded curvature. Then for any $\varepsilon \in (0,1)$ there exists a $\delta>0$ such that for any $w\in M$ we have
  \begin{align*}	
    (1+\varepsilon)^{-1}| \exp_{w}^{-1}(x)-\exp_{w}^{-1}(y)|_{g_{w}} \leq d(x,y) \leq (1+\varepsilon)| \exp_{w}^{-1}(x)-\exp_{w}^{-1}(y)|_{g_{w}}
  \end{align*}	   
  for every $x,y \in B_{4\delta\sqrt{d}}(w)$.   
\end{lem}

\begin{rem}
	We note that the operation $\exp_w^{-1}(x)-\exp_w^{-1}(y)$ reduces simply to $x-y$ when $M$ happens to be Euclidean space. The same can be said of $\exp_{y}^{-1}(x)$ --which will also appear later in a expression that involves $x-y$ in the case $M$ is flat.
\end{rem}

\begin{proof}[Proof of Lemma \ref{lem: injectivity radius}]
	This is just, for example, the result in Lee's book \cite[Prop 5.11]{Lee-1997RiemannianManifoldsBook} restated in our setting.  We leave the proof to \cite{Lee-1997RiemannianManifoldsBook}.
\end{proof}

  \begin{center}
    \includegraphics[scale=0.15]{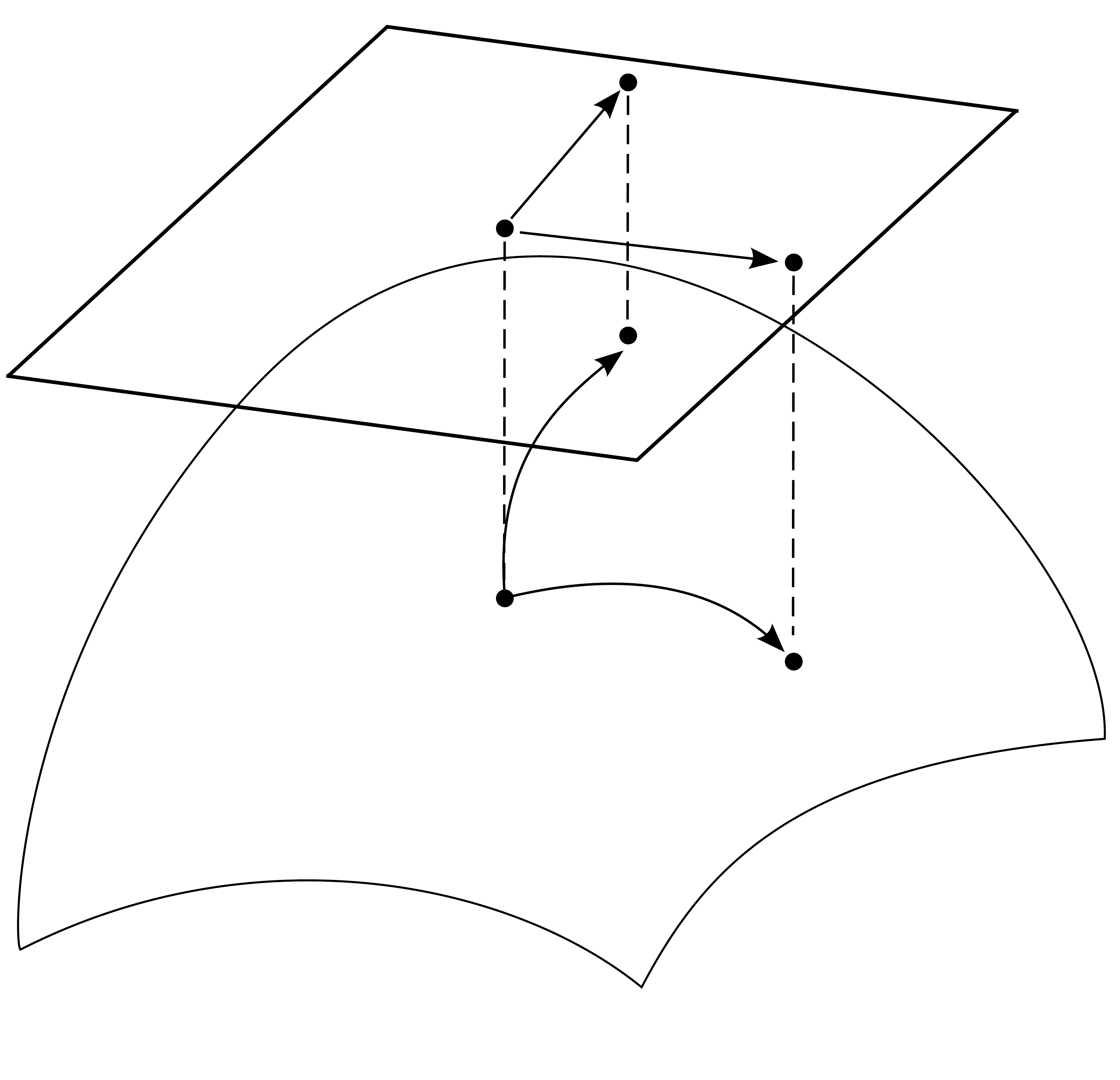}
  \captionof{figure}{The Exponential Map}\label{fig:ExpMap}   
  \end{center}
  
The above lemma says that we can control the amount by which the exponential map fails to be an isometry from $(TM)_w$ to $M$ by restricting to a small enough neighborhood of the origin in $(TM)_w$. We fill fix a ``distortion'' factor $\varepsilon$, and cover $M$ with sufficiently small balls where the above holds. We record this observation as a remark.

\begin{rem}\label{rem: good balls cover M} Choose $\delta \in (0,1)$ sufficiently small so that conclusion of Lemma \ref{lem: injectivity radius} holds with $\varepsilon=1/100$.  We fix an auxiliary sequence of points $\{w_i\}_i$ having the property 
  \begin{align}\label{eqMMonM:CoveringCenterPpoints}
    M = \bigcup\limits_{i} B_{\delta}(w_i).
  \end{align}
  Moreover, we select these points making sure the covering has the following property: there is a number $N_0>0$ such that any $x\in M$ lies in at most $N_0$ of the balls $\{B_{4\delta\sqrt{d}}(w_i)\}_i$.
\end{rem}

  From here on, we shall fix an infinite sequence of finite subsets $M$ which, informally speaking, approximate the entire manifold (let us emphasize these points are different from the centers of the cover in Remark \ref{rem: good balls cover M}). 
  
  It will be useful to construct a sequence of discrete, but not necessarily finite, approximations to $M$ (which will contain the finite ones). This sequence shall be denoted $\{\tilde G_n\}_n$, and it is assumed to have the following properties:
  \begin{enumerate}
    \item The sequence is monotone increasing, $\tilde G_n \subset \tilde G_{n+1},\;\;\forall\;n\in\mathbb{N}$.	
    \item For every $n$, we have
    \begin{align}\label{eqFinDimApp: tilde hn def}
      \tilde h_n := \sup \limits_{x\in M}d(x,\tilde G_n),\;\;\sup \limits_{n}\tilde h_n\leq \delta/500 ,\;\;\lim\limits_{n} \tilde h_n = 0.
    \end{align}
    \item There exist a constant $\lambda>0$ independent of $n$, such that	
    \begin{align}\label{eqFinDimApp: lambda def}	
      \inf \limits_{\substack{x,y \in \tilde G_n \\  x\neq y}}d(x,y) \geq \lambda \tilde h_n.	
    \end{align}	
	
  \end{enumerate}
  \begin{rem}
    The existence of such a sequence of sets is not too difficulty to verify. For the sake of brevity, we only sketch its construction: take an orthogonal grid at each of the points $w_i$, and push them down via the respective exponential map, throw away points as needed. 
  \end{rem}
  \begin{rem}
    The fact that $\tilde h_n$ is much smaller than $\delta$ is used at several points in the proof. In particular, the explicit factor of $500$ in \eqref{eqFinDimApp: tilde hn def} is chosen to guarantee there are sufficiently many points of $\tilde G_n$ in any ball of radius $\delta$, a fact that is not used until the Appendix (Proposition \ref{prop:Appendix Admissible Directions}), where we prove several important facts about the discretization of the gradient and the Hessian.
  \end{rem}

  Then, the sequence of finite sets $\{G_n\}_n$ is constructed as follows: we fix an auxiliary point $x_* \in M$ and let
  \begin{align}\label{eqFinDimApp: Mn def}
    M_n := B_{2^n}(x_*),
  \end{align}
  and define
  \begin{align}\label{eqFinDimApp: Gn def}
    G_n :=  \tilde G_n \cap M_{n+1}.
  \end{align}
  
  It is not surprising that the sequence $\{G_n\}_n$ has similar properties as $\{\tilde G_n\}$. As these properties will be used successively throughout the paper, we record them all in a single proposition. 
  \begin{prop}\label{prop: Properties of G_n}
    The following properties are satisfied by $\{G_n\}_n$
    \begin{enumerate}
      \item If $M$ is compact, then $G_n=\tilde G_n$ for all large enough $n$.		
      \item For every $n$ we have $G_n\subset G_{n+1}$.
      \item Each $G_n$ is finite.
      \item We have, with $M_n$ as defined in \eqref{eqFinDimApp: Mn def}, that $h_n = \tilde h_n$, in particular
      \begin{align}\label{eqFinDimApp: hn def}
        h_n & := \sup \limits_{x\in M_n}d(x,G_n), \textnormal{ satisfies } \sup \limits_{n} h_n \leq \delta/500,\;\;\lim\limits_{n} h_n = 0.
      \end{align}
      \item Let $h_n$ be as in \eqref{eqFinDimApp: hn def} and $\lambda$ as in \eqref{eqFinDimApp: lambda def}, then for all sufficiently large $n$ we have
      \begin{align}\label{eqFinDimApp: hn regularity property}
        \inf \limits_{\substack{x,y\in G_n \\ x\neq y}} d(x,y) \geq \lambda h_n.
      \end{align} 
    \end{enumerate}		
  \end{prop}
  
  \begin{proof}
    Properties (1) and (2) are obvious.  Next, from \eqref{eqFinDimApp: lambda def} it follows in particular that $\tilde G_n$ has no accumulation points, and thus Property (3) follows from the fact that $M_n$ is bounded.
	
    By the assumptions on $\tilde G_n$, for any $x\in M_n$ there is some $\hat x \in \tilde G_n$ such that
    \begin{align*}	 
      d(x,\hat x)\leq \tilde h_n.	
    \end{align*}	
    Since $d(x,x_*)\leq 2^n$, it follows that $d(\hat x,x^*)\leq 2^n+h_n\leq 2^{n+1}$ since $\tilde h_n\leq 1$ for all $n$ by \eqref{eqFinDimApp: tilde hn def}. This means that $\hat x\in B_{2^{n+1}}(x_*) = M_{n+1}$, and that $\hat x\in G_n$. This shows that   		
    \begin{align*}
      \tilde h_n = \sup \limits_{x \in M}d(x,\tilde G_n) = \sup \limits_{x\in M_n}d(x,G_n) = h_n.
    \end{align*}			
    and Property (4) is proved.  On the other hand, we have the trivial inequality 
    \begin{align*}
      \inf \limits_{\substack{x,y\in G_n \\ x\neq y}} d(x,y)  \geq \inf \limits_{\substack{x,y\in \tilde G_n \\ x\neq y}} d(x,y),
    \end{align*}		
    then \eqref{eqFinDimApp: lambda def} says this last term is at least $\lambda \tilde h_n$, which is equal to $\lambda h_n$, which proves Property (5).
  \end{proof}

  \begin{rem} If $M=\mathbb{R}^d$, for each $n\in\mathbb{N}$, we consider the Cartesian grid
    \begin{align*}
      \tilde G_n & := ( 2^{-2-n})\mathbb{Z}^d.
    \end{align*}	  
    It is straightforward to see that $\{\tilde G_n\}_n$ has all the desired properties.
  \end{rem}
  
  \begin{rem} Although the finite sets $G_n$ will be the ones actually used in the proof of the main theorem, that will not happen until Section \ref{sec: Min-Max formula for M,g}, for the rest of this section, we will be mostly concerned with $\tilde G_n$.
  \end{rem}
  
  We now start the construction.  For each $n$ we shall construct open covers $\{P_{n,k}\}_{k\in\mathbb{N}}$ and $\{P_{n,k}^*\}_{k\in\mathbb{N}}$ of  $M\setminus \tilde G_n$, comprised of subsets of $M \setminus \tilde G_n$ (that is, the sets $P_{n,k}$ and $P_{n,k}^*$ will be disjoint from $\tilde G_n$). The sets in these covers will obtained by applying the exponential map to families of cubes lying in the tangent spaces $\{TM_{w_i}\}_i$. The cubes themselves are chosen following the classical Whitney cube decomposition, see \cite[Chp 6, Thm 1]{Stei-71}.
  
  \begin{lem}\label{lem:CubesOnM}
    For every $n$ there exists two families of open sets $\{P_{n,k}\}_k,\{P_{n,k}*\}_k$ such that
    \begin{enumerate}
      \item For every $k$, there is some $w_{i_k}$ --$\{w_i\}$ being the points fixed in Remark \ref{rem: good balls cover M}-- such that
        \begin{align*}	  
          P_{n,k} & = \exp_{w_{i_k}}( Q_{n,k}),\;\;\;P_{n,k}^*= \exp_{w_{i_k}}( Q_{n,k}^*),			  
        \end{align*}  
      where $Q_{n,k}$ is a cube in $(TM)_{w_{i_k}}$, and $Q_{n,k}^*$ its concentric cube with length increased by a factor of $\tfrac{9}{8}$.
      \item For every $k$, we have 	  	  
        \begin{align*}
          \tfrac{1}{5}d(P_{n,k},\tilde G_n)\leq \textnormal{diam}(P_{n,k}) \leq 5d(P_{n,k},\tilde G_n),\\
          \tfrac{1}{5}d(P_{n,k}^*,\tilde G_n)\leq \textnormal{diam}(P_{n,k}^*) \leq 7d(P_{n,k}^*,\tilde G_n).		  
        \end{align*}
      \item There is a universal $N>0$, which in particular, is independent of $n$, such that if
	  \begin{align}
            K_x:= \{k \mid x\in P_{n,k}^*\},\;\;x\in M\setminus \tilde G_n,\label{eqFinDim:Set Of Covering Cubes}
	  \end{align}
	then
        \begin{align}
          \# \{ k \mid x\in P_{n,k}^*\} \leq N \;\;\forall\;x\in M\setminus \tilde G_n. \label{eqFinDim:Size Set Of Covering Cubes}
        \end{align}
      \item The sets $\{P_{n,k}\}_k$ cover the complement of $\tilde G_n$,
        \begin{align*}
          \bigcup \limits_{k} P_{n,k} = M \setminus \tilde G_n.
        \end{align*}
    \end{enumerate}	
  \end{lem}

  \begin{center}
    \includegraphics[scale=0.35]{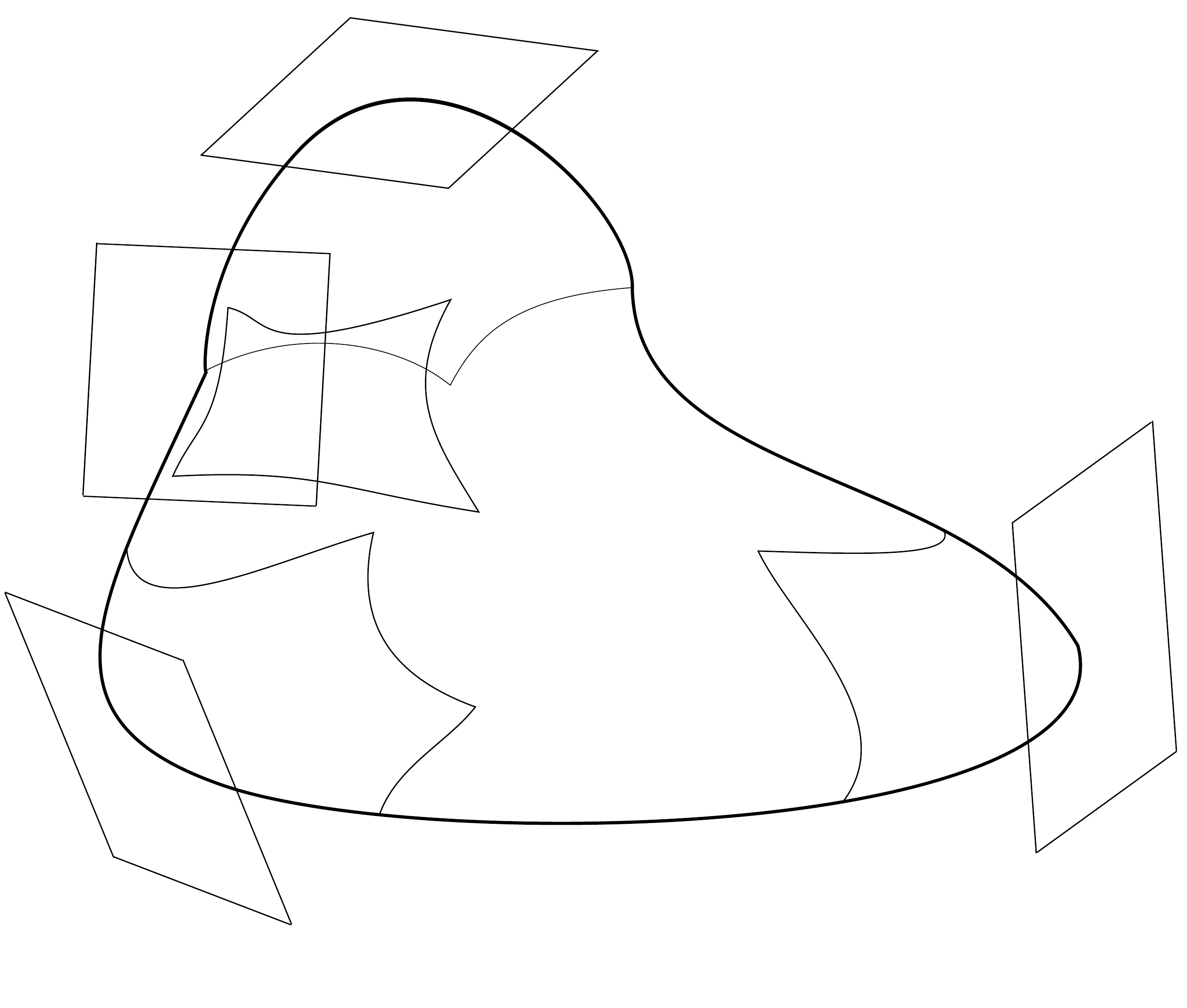}
  \captionof{figure}{Example cubes, $Q_{n,k}$, in $(TM)_x$ projected to $M$, as $P_{n,k}$}\label{fig:TangentSquares}
  \end{center}

  \begin{proof}   
    Let $\delta$ be the constant from Remark \ref{rem: good balls cover M}.  In what follows, we will lift a portion of $\tilde G_n$ to the vector space $TM_{w_i}$, for some nearby $w_i$. Then, we apply the Whitney cube decomposition to the resulting set \cite[Chp 6]{Stei-71}, producing cubes in $TM_{w_i}$ that will have the desired properties. These cubes are then mapped to $M$ via $exp_{w_i}$.
		
    \noindent For each $w_i$, we select 
    \begin{align*}
      [e_{i,1},\ldots,e_{i,d}],\;\textnormal{ an orthonormal basis of } (TM)_{w_i},		
    \end{align*}			
    the purpose of these bases is to allow us to set a rectangular grid in each of the tangent spaces. Which particular basis we choose each $w_i$ will be immaterial. For each $n\in \mathbb{N}$ and $w_i$, define
    \begin{align*}	   
      F_{n,i} := \exp_{w_i}^{-1}(\tilde G_n \cap B_{3\delta\sqrt{d}}(w_i)),\;\;\;\Omega_{n,i} := B_{\delta}(0)\setminus F_{n,i}.  	   
    \end{align*}
    For each $i$ we construct a family of cubes in $(TM)_{w_i}$, denoted by $\mathcal{Q}_{n,i}$. The family is obtained by applying Whitney's cube decomposition in $(TM)_{w_i}$. Ultimately, this cube decomposition will be pushed down to $M \setminus \tilde G_n$ via the exponential map at $w_i$ (see Figure \ref{fig:StackOfCubeDecomposition}).

    Let us go over the cube decomposition. As we are working on a manifold, it will be convenient to consider cubes inside a small enough cube centered at the origin of $(TM)_{w_i}$. Keeping this in mind --and recalling that $\delta$ was chosen in Remark \ref{rem: good balls cover M}-- we let $m_0 \in \mathbb{N}$ be the universal constant determined by
    \begin{align*}
      2\delta \leq 2^{-m_0} < 4\delta.	
    \end{align*}	
    In other words, $m_0$ is the largest number such that $B_\delta(0)\subset (TM)_{w_{i}}$ is contained inside the cube centered at $0$ with common side length equal to $2^{-m_0}$, that is
    \begin{align*}
      Q_{2^{-m_0-1}}(0) = \{ q \in (TM)_{w_i} : |(q,e_{i,l})_{g_x}| \leq 2^{-m_0-1},\;l=1,\ldots,d\}. 		
    \end{align*}			
    Then, considering only those cubes obtained by repeatedly bisecting the sides of $Q_{2^{-m_0-1}}(0)$, we define $\hat {\mathcal{Q}}_{n,i}$ to be the subfamily formed by those cubes $Q$ for which we also have
    \begin{align*}		
      Q \cap \{  q \in (TM)_{w_i} \mid 2\diam(Q) \leq d(q,F_{n,i}) \leq 4\diam(Q) \} =\emptyset.
    \end{align*}	
    Then, let us say that a cube $Q$ in $\hat{\mathcal{Q}}_{n,i}$ is maximal if there is no other cube $Q'$ in the family such that $Q'\subset Q$.  The family $\mathcal{Q}_{n,i}$ is then defined to be the subfamily of maximal cubes of $\hat {\mathcal{Q}}_{n,i}$.
	
    \begin{center}
      \includegraphics[scale=0.75]{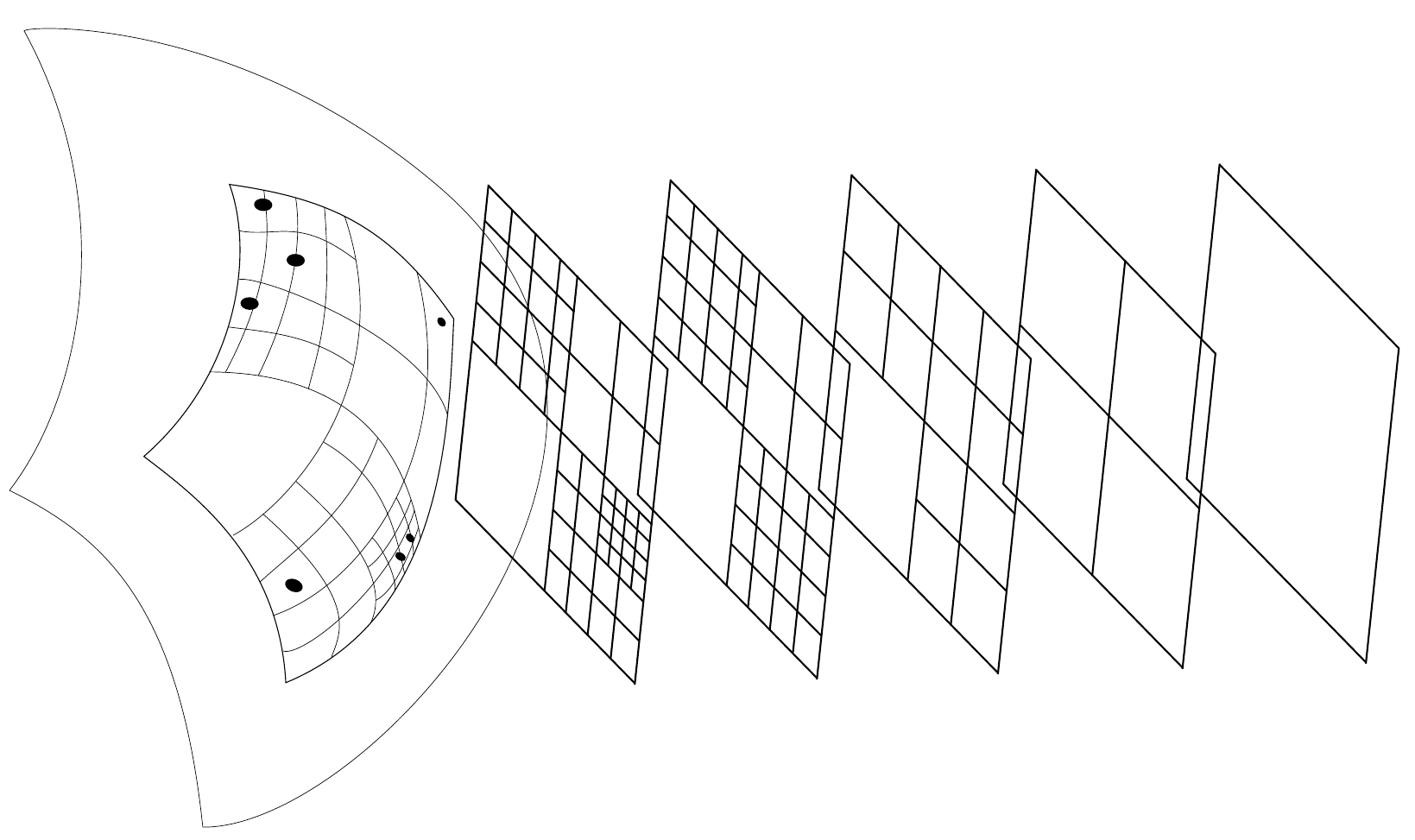}
    \captionof{figure}{Decomposing Cubes In $(TM)_x$}\label{fig:StackOfCubeDecomposition}  
    \end{center}
		
    The family $\mathcal{Q}_{n,i}$ has the following properties
    \begin{enumerate}
      \item Any two distinct elements of $\mathcal{Q}_{n,i}$ have disjoint interiors.
      \item Every $q\in \Omega_{n,i}$ lies in the interior of a cube belonging to $\mathcal{Q}_{n,i}$.
      \item If $Q \in \mathcal{Q}_{n,i}$, then the common side length of $Q$ is no larger than $2^{-m_0}\leq 4\delta$. In particular, $Q$ lies inside $B_{2\delta \sqrt{d}}$, and $Q^*$ lies inside $B_{3\delta \sqrt{d}}(0)$.
      \item There is a number $N_1$, independent of $n$ and $i$, such that any $q \in \Omega_{n,i}$ lies in at most $N_1$ of the sets $\{Q^*\}_{Q\in \mathcal{Q}_{n,i}}$.
      \item The cubes in the family have a diameter comparable to their distance to $F_{n,i}$. Concretely, 
      \begin{align}
        \diam(Q) \leq d(Q,F_{n,i}) \leq 4\diam(Q), \;\;\forall\;Q\in\mathcal{Q}_{n,i}.\label{eqFinDim:tangent space cubes diameter vs distance to Fni}
      \end{align}		  	  	  	   
    \end{enumerate}	
    We omit the verification of these properties, as it is standard. We refer the interested reader to \cite[Chap. 6, Sec. 1]{Stei-71} for details. 
	
      Let us immediately note that bounds akin to \eqref{eqFinDim:tangent space cubes diameter vs distance to Fni} extend to the respective ``stretched'' cubes $Q^*$. Indeed, fix some $Q \in \mathcal{Q}_{n,i}$. From $Q\subset Q^*$ we have $d(Q^*,F_{n,i}) \leq d(Q,F_{n,i})$, while from $d(Q,(Q^*)^c) = (1/8)\diam(Q)$ we have $d(Q^*,F_{n,i})\geq d(Q,F_{n,i})-(1/8)\diam(Q)$. From these observations and \eqref{eqFinDim:tangent space cubes diameter vs distance to Fni} it follows that
      \begin{align}
        \tfrac{7}{9}\diam(Q^*)= \tfrac{7}{8}\diam(Q)\leq d(Q^*,F_{n,i}) \leq 4\diam(Q^*), \;\;\forall\;Q\in\mathcal{Q}_{n,i}.\label{eqFinDim:tangent space larger cubes diameter vs distance to Fni}						
      \end{align}		  	  	  	   
    Having the families $\mathcal{Q}_{n,i}$ (for each $i$ for which $F_{n,i}\not=\emptyset$), let us combine them into a single one, which will also be countable. Let $\{Q_{n,k}\}_k$ denote an enumeration of the elements of this larger family. Each $Q_{n,k}$ is a cube belonging to some tangent space $(TM)_{w_{i_k}}$ for some $w_{i_k}$. 
		
    Let $q_{n,k}$ denote the the center of $Q_{n,k}$, and $l_{n,k}$ its common side length. Then, we define
    \begin{align*}		
      P_{n,k} := \exp_{w_i}(Q_{n,k}),\;\;P_{n,k}^*:=\exp_{w_i}(Q_{n,k}^*),\;\;y_{n,k} =\exp_{w_{i_k}}(q_{n,k}),
    \end{align*}	
    This produces a family of sets for which Property (1) holds. Let us verify these families satisfy the other three Properties. Let us prove Property (2). Fix $P_{n,k} = \exp_{w_{i_k}}(Q_{n,k})$. Then,
    \begin{align*}	
      d(P_{n,k},\tilde G_n) & \leq d(P_{n,k},\tilde G_n \cap B_{3\delta \sqrt{d}}(w_{i_k}))\\
	    & \leq \tfrac{101}{100}d(Q_{n,k},F_{n,i})\\
	    & \leq 4\tfrac{101}{100}\diam(Q_{n,k}) \leq 4(\tfrac{101}{100})^2\diam(P_{n,k}) \leq 5\diam(P_{n,k}).
    \end{align*}	
    The exact same argument yields
    \begin{align*}	
      d(P_{n,k}^*,\tilde G_n) & \leq 5\diam(P_{n,k}^*)
    \end{align*}	
    This yields one side of the bounds in Property (2). Next, note that  $Q_{n,k} \subset B_{2\delta \sqrt{d}}(0)$, which means that $\diam(Q_{n,k}) \leq 4\delta\sqrt{d}$ and
    \begin{align*}
      d(P_{n,k},\tilde G_n \setminus B_{3\delta \sqrt{d}}(w_{i_k})) & \geq \tfrac{100}{101}d(Q_{n,k},\partial B_{3\delta \sqrt{d}}(0))\\
	  & \geq \tfrac{100}{101}\delta \sqrt{d}\\
	  & \geq \tfrac{1}{4}\tfrac{100}{101}\diam(Q_{n,k})\geq \tfrac{1}{4}(\tfrac{100}{101})^2\diam(P_{n,k}) \geq \tfrac{1}{5}\diam(P_{n,k}).	  
    \end{align*}		
    At the same time,
    \begin{align*}	
      d(P_{n,k},\tilde G_n \cap B_{3\delta \sqrt{d}}(w_{i_k})) \geq \tfrac{100}{101}d(Q_{n,k},F_{n,i_k}) & \geq \tfrac{100}{101}\diam(Q_{n,k})\\
	  & \geq (\tfrac{100}{101})^2\diam(P_{n,k})\\
	  & \geq \tfrac{1}{5}\diam(P_{n,k}).
    \end{align*}
    Therefore,
    \begin{align*}	
      d(P_{n,k},\tilde G_n) \geq \tfrac{1}{5}\diam(P_{n,k}).
    \end{align*}	
    With the same argument, one can check that
    \begin{align*}	
      d(P_{n,k}^*,\tilde G_n) \geq \tfrac{7}{36}(\tfrac{100}{101})^2\diam(P_{n,k}^*) \geq \tfrac{1}{7}\diam(P_{n,k}^*) ,
    \end{align*}			
    and Property (2) is proved.  Next, recall the sequence $\{w_i\}$ is such that given $x \in M$, then
    \begin{align*}	
      \#\{ i \mid\;x\in B_{3\delta\sqrt{d}}(w_i)\}\leq N_0. 
    \end{align*}
       It follows that each $x$ lies in at most $N_0$ of the sets $\{\Omega_{n,i}\}_i$, and Property (3) follows immediately by taking $N:=N_0N_1$. Finally, from \eqref{eqFinDim:tangent space larger cubes diameter vs distance to Fni}, we have
       \begin{align*}	
         P_{n,k}\subset P_{n,k}^* \subset M \setminus \tilde G_n\;\;\forall\;n,k.	
       \end{align*}	
       Furthermore, since the balls $\{B_{\delta}(w_i)\}_i$ cover $M$, and each $B_{\delta}(w_i)$ is covered by $\{P_{n,k}\}$, we have
       \begin{align*}	
         \bigcup \limits_{k} P_{n,k} \supset \bigcup\limits_{i} \{ B_{\delta}(w_i)\setminus \tilde G_n \}  = M \setminus \tilde G_n.			
       \end{align*}
       Thus we obtain Property (4), and the lemma is proved.
  \end{proof}
  
  From this point onward, the sets $\tilde G_n,G_n$, and the associated family of open sets $\{P_{n,k}\}_k$ and $\{P_{n,k}^*\}_k$ will be fixed. For every $k$,  by the ``center'' of $P_{n,k}$ we will mean the point $y_{n,k} = \exp_{w_{i_k}}(q_{n,k})$. Furthermore, $\hat y_{n,k}$ will denote a point in $\tilde G_n$ which realizes the distance from $y_{n,k}$ to $\tilde G_n$.  Let us record these definitions for further reference:
  \begin{align}\label{eqWhitney:DefOfYnkHat}
    y_{n,k} := \exp_{w_{i_k}}(q_{n,k}),\ \text{and}\ \hat y_{n,k}\in \tilde G_n\ \text{such that}\     d(y_{n,k},\hat y_{n,k}) = d(y_{n,k}, \tilde G_n).  
  \end{align}

  The following elementary fact will be used repeatedly in this section, we record it as a remark.
  \begin{rem}\label{rem: Whitney local diameter of balls}
    Let $x \in P_{n,k}^*$. Then we have the inequalities
    \begin{align*}	
      \tfrac{1}{7}\diam(P_{n,k}^*)\leq d(x,\tilde G_n) \leq 6\diam(P_{n,k}^*).
    \end{align*}	
    Let us prove this. By the triangle inequality $d(x,\tilde G_n)\leq d(P_{n,k}^*,\tilde G_n)+\diam(P_{n,k}^*)$. Then, (2) from Lemma \ref{lem:CubesOnM} says that
    \begin{align*}
      d(x,\tilde G_n) & \leq 5\diam(P_{n,k}^*)+\diam(P_{n,k}^*) \leq 6\diam(P_{n,k}^*).	 
    \end{align*}
    On the other hand, since $d(P_{n,k}^*,\tilde G_n)$ is just the infimum of $d(\cdot,\tilde G_n)$ over $P_{n,k}^*$,
    \begin{align*}
      d(x,\tilde G_n) & \geq d(P_{n,k}^*,\tilde G_n) \geq \tfrac{1}{7}\diam(P_{n,k}^*),
    \end{align*}
    the second inequality being again thanks to (2) from Lemma \ref{lem:CubesOnM}.	
  \end{rem}
  
  Continuing in parallel with the classical approach to the extension problem \cite[Chapter 6]{Stei-71}, we construct a partition of unity for $M \setminus \tilde G_n$ associated to the family $\{P_{n,k}\}_{n,k}$.  Since we work on a Riemannian manifold, we will need to compute covariant derivatives for scalar functions, up to third order  (since the highest regularity we will be concerned with is $C^{2,\alpha}$, this will suffice for all our purposes). For a review of the definition of $\nabla^i\phi$ and its basic properties, see the end of Section 1.1 in \cite[Chapter 1]{Heb-2000NonlinearAnalysisManifoldsBook}.
  \begin{lem}[Partition of unity]\label{lem:PartitionOfUnityGn}
    For every $n$, there is a family of smooth functions $\{\phi_{n,k}\}_k$ such that
    \begin{enumerate}
     \item $\sum \limits_{k} \phi_{n,k}(x) = 1$ for all $x \in M\setminus \tilde G_n$.
     \item $0\leq \phi_{n,k}\leq 1$ in $M\setminus \tilde G_n$ and $\phi_{n,k}\equiv 0$ outside $P_{n,k}^*$. 
      \item There is a constant $C$ such that for every $x\in M\setminus \tilde G_n$, every $n,k$ and $i=1,2,3$ we have 
      \begin{align*}
        |\nabla^i \phi_{n,k}(x)|_{g_x} \leq \frac{C}{(\diam(P_{n,k}^*))^i}.	  
      \end{align*}		  	  
    \end{enumerate}
  \end{lem}
  
  \begin{proof}
    As is standard for a construction of a partition of unity, we will begin with auxiliary functions, $\tilde \phi_{n,k}$ that are basically smooth bumps localized at the centers of the sets $P_{n,k}$, and then we normalize their sum to obtain the desired family, $\{\phi_{n,k}\}$.
    
    Let us fix an auxiliary function $\psi_0:\mathbb{R}^d\to\mathbb{R}$ with the following properties
    \begin{align*}	
      \psi_0 \in C^{\infty}(\mathbb{R}^d),\;\;\psi_0 \equiv 1 \textnormal{ in } [-1,1]^d,\;\;\psi_0 \equiv 0 \textnormal{ outside } [-\tfrac{9}{8},\tfrac{9}{8}]^d.	
    \end{align*}	
    Using the basis $[e_{i,1},\ldots,e_{i,d}]$ for each $i$, we can ``push''  the above function to smooth functions
    \begin{align*}	
      \psi_{i}:(TM)_{w_i} \to \mathbb{R}.		
    \end{align*}		
    Then, for each $k$ we let $\tilde \phi_{n,k}$ be defined by
    \begin{align*} 
      \tilde \phi_{n,k}(x) = \left \{ \begin{array}{ll}
	  \psi_{i_k} \left ( (l_{n,k}/2)^{-1}(\exp_{w_{i_k}}^{-1}(x)-q_{n,k})\right ) & \textnormal{ inside } P_{n,k}^*\\
	  0 & \textnormal{ outside } P_{n,k}^*.
	  \end{array}\right.
    \end{align*}
    Here $l_{n,k} = (\sqrt{d})^{-1}\diam(Q_{n,k})$ is the common length of the sides of $Q_{n,k}$. Since $P_{n,k}^*$ lies uniformly in a normal neighborhood, and $\psi_0 \in C^\infty$, it follows for each $k$ that $\tilde \phi_{n,k}$ is a smooth function. Moreover, using the definition of $\tilde \phi_{n,k}$ above it is straightforward to check that
    \begin{align*}		
      \tilde \phi_{n,k}(x) \equiv	1 \textnormal{ in } P_{n,k},\;\;\tilde \phi_{n,k}(x)\equiv 0 \textnormal{ outside } P_{n,k}^*.
    \end{align*}	
    Furthermore, from the chain rule it follows easily there is a universal $C$ such that for $i=1,2,3$,
    \begin{align*}	
      \sup \limits_{x\in M}|\nabla^i \tilde \phi_{n,k}(x)|g_{x} \leq \frac{C}{(l_{n,k})^i} \leq \frac{C}{(\diam(P_{n,k}^*))^i}. 
    \end{align*}
    Next, we consider the function
    \begin{align*}
      \phi_n(x) & := \sum\limits_{k}\tilde \phi_{n,k}(x).
    \end{align*}
    Note that at most $N$ of the sets $P_{n,k}^*$ contain $x$ (Lemma \ref{lem:CubesOnM}), and therefore, only at most $N$ of the functions $\tilde \phi_{n,k}$ are non-zero. Thus the sum defining $\phi_n$ is locally the sum of at most $N$ non-zero smooth functions.  In particular, we may differentiate to obtain
    \begin{align*}	
      \nabla^{i} \phi_n(x)  = \sum \limits_{k}\nabla^{i} \tilde \phi_{n,k}(x),\;\;i=1,2,3.
    \end{align*}
    Let us estimate the derivatives of $\phi_n(x)$, for each $i=1,2,3$ we have
    \begin{align*}
      |\nabla^i \phi_n(x)|_{g_x} & \leq \sum \limits_{k}|\nabla \tilde \phi_{n,k}(x)|_{g_x} \leq \sum \limits_{k\in K_x}\frac{C}{(\diam(P_{n,k}^*))^i}
    \end{align*}
    Then, by Remark \ref{rem: Whitney local diameter of balls}
    \begin{align*}	
      k\in K_x \Rightarrow x \in P_{n,k}^* \Rightarrow d(x,\tilde G_n) \leq 6\diam(P_{n,k}^*).	
    \end{align*} 
    Using again that $\#K_x\leq N$, we conclude for $i=1,2,3,$ that		
    \begin{align*}
      |\nabla^i \phi_n(x)|_{g_x} & \leq  N\frac{C}{(d(x,\tilde G_n))^i}\leq \frac{C}{(d(x,\tilde G_n))^i}.
    \end{align*}	
    On the other hand,  for every $x$ there is at least one $k$ such that $x \in P_{n,k}$, thus
    \begin{align*}
      1\leq \phi_n(x) \leq N\;\;\forall\;x \in M\setminus \tilde G_n.
    \end{align*}	
    We may now define the actual family of functions $\{\phi_{n,k}\}_{n,k}$. For each $k$, let
    \begin{align*} 
      \phi_{n,k}(x) := \frac{\tilde \phi_{n,k}(x)}{\phi_n(x)}.
    \end{align*}
    It is simple to check that this family of functions has Property (1). Indeed, using that the sum is locally finite, we have
    \begin{align*}	
      \sum \limits_{k}\phi_{k}(x)= \sum \limits_{k} \frac{\tilde \phi_{n,k}(x)}{\phi_n(x)} = \frac{1}{\phi_n(x)}\sum \limits_{k} \tilde \phi_{n,k}(x) = \frac{\phi_n(x)}{\phi_n(x)} = 1.
    \end{align*}	
    On the other hand, Property (2) follows as each $\tilde \phi_{n,k}$ is non-negative and supported in $P_{n,k}^*$ from the definition of $\tilde \phi_{n,k}$. As for Property (3), we compute
    \begin{align*}
      \nabla \phi_{n,k}(x) & = \frac{\nabla \tilde \phi_{n,k}(x)}{\phi_n(x)}-\frac{\tilde \phi_{n,k}(x)}{\phi_n(x)^{2}} \nabla \phi_n(x)\\
      \nabla^2 \phi_{n,k}(x) & = \frac{\nabla^2\tilde \phi_{n,k}(x)}{\phi_n(x)}-\frac{\tilde \phi_{n,k}(x)}{\phi_n(x)}\nabla^2\phi_n(x).
    \end{align*}
    Combining the estimates for the derivatives of $\tilde \phi_{n,k}$ and $\phi_n$ yields, for each $x\in M\setminus \tilde G_n$,
    \begin{align*}
      |\nabla \phi_{n,k}(x)|_{g_x} & \leq |\nabla \tilde \phi_{n,k}(x)|_{g_x}+|\nabla \phi_n(x)|_{g_x} \leq  \frac{C}{\diam(P_{n,k}^*)} \;,\\
      |\nabla^2\phi_{n,k}(x)|_{g_x} & \leq |\nabla^2\tilde \phi_{n,k}(x)|_{g_x}+|\nabla^2\phi_n(x)|_{g_x} \leq \frac{C}{(\diam(P_{n,k}^*))^2}\;.
    \end{align*}
    Where we have used Remark \ref{rem: Whitney local diameter of balls} (once again) to obtain the second bound in each case. The respective bound for $\nabla^3\phi_{n,k}(x)$ follows similarly, and we omit the details. 
  	
  \end{proof}
  
  \begin{rem}
    As stated in Lemma \ref{lem:PartitionOfUnityGn}, we have $\sum_k \phi_{n,k}\equiv 1$ on $M \setminus \tilde G_n$. After repeatedly differentiating this identity we obtain another identity that will be of use later on,
    \begin{align}\label{eqFinDim:Partition of Unity Differentiation Identity}
      \sum \limits_{k} \nabla^{i} \phi_{n,k}(x) \equiv 0 \textnormal{ in } M\setminus \tilde G_n,\;\; i=1,2,3.    
    \end{align}
  
  \end{rem}
  
  \subsection{Local interpolators} We have constructed a ``cube'' covering of $M \setminus \tilde G_n$ (since $P_{n,k}$ are only cubes when seen in the right exponential chart), and a corresponding partition of unity. Next, we need to fix a choice for ``local'' interpolating functions.  Specifically, we need to define what will take the place of the local linear and quadratic functions in the usual Whitney extensions.
  
  Recall that $\delta \in (0,1)$ was chosen in Remark \ref{rem: good balls cover M} so the exponential map was roughly an isometry in balls of radius $4\delta \sqrt{d}$. In particular this means that for $y \in M$, the inverse exponential map $\exp_{y}^{-1}$ is a well defined, uniformly smooth map from $B_{4\delta \sqrt{d}}(y)$ to a neighborhood of zero in $(TM)_y$. This smooth map defines local charts on $M$ having useful properties (they are normal systems of coordinates), and using such charts we shall introduce (locally defined) functions that will play the role of  ``linear'' and ``quadratic'' functions near a given point $y \in M$.
  
  \begin{DEF}\label{def:LinearAndQuadPolynomials}
    Given $y \in M$ and a vector $p \in (TM)_{y}$, define $l(p,y;\cdot):B_{4\delta\sqrt{d}}(y)\to\mathbb{R}$ by
    \begin{align*}
      &l(p,y;x) := (\exp_{y}^{-1}(x),p)_{g_{y}},\;\;\forall\;x\in B_{4\delta\sqrt{d}}(y).
    \end{align*}
    Given a self-adjoint linear transformation $D\in \L((TM)_{y})$, define $q(D,y;\cdot): B_{4\delta\sqrt{d}}(y)\to\mathbb{R}$ by 
    \begin{align*}
      & q(D,y;x) := \tfrac{1}{2}(D\exp_{y}^{-1}(x),\exp_{y}^{-1}(x))_{g_{y}},\;\;\forall\;x\in B_{4\delta\sqrt{d}}(y).
    \end{align*}	 

  \end{DEF}
  
  \begin{rem}\label{rem:FinDimApp l and q derivatives in a chart}
    An equivalent formulation of the above is the following. In $B_{4\delta \sqrt{d}}(y)$ one obtains coordinate functions $\xi^1,\ldots,\xi^d$ by choosing an orthonormal basis $\{e_i\}$ at $(TM)_y$ and setting
    \begin{align*}	
      \xi^i(x):= (e_i,(\exp_{y})^{-1}(x))_{g_y}.	
    \end{align*}
    Then, the functions $l$ and $q$ seen in these coordinates are simply linear and quadratic polynomials,
    \begin{align*}	
      l(p,y;x) = p_i\xi^i,\;\;\; q(D,y;x) = \tfrac{1}{2}D_{ij}\xi^i\xi^j.  	
    \end{align*}	
    Where $p_i$ and $D_{ij}$ are the components of $p$ and $D$ in the basis $\{e_i\}$.
	
    Moreover, these coordinates are \emph{normal}, meaning that the Christoffel symbols vanish at the origin of the system of coordinates $\xi^1=\ldots=\xi^d=0$, that is, at the point corresponding to $y$ itself. In particular, it follows that 	
    \begin{align*}
      \nabla l(p,y;y)= p, & \quad \nabla^2 l(p,y;y) = 0,\\
      \nabla q(D,y;y)= 0, & \quad \nabla^2 q(D,y;y) = D.
    \end{align*}			
    Which confirms the idea that $l$ and $q$ play the role of linear and quadratic functions near a point.	
  \end{rem}
  
  The next remark explains an important technical fact. Namely, for each $k$ the set $P_{n,k}^*$ lies in a sufficiently small neighborhood of $\hat y_{n,k}$ so that, given $p$ or $D$, the functions $l(p,\hat y_{n,k};\cdot)$ and $q(D,\hat y_{n,k};\cdot)$ are well defined and smooth in $P_{n,k}^*$.
  \begin{rem}\label{rem:local interpolators are well defined}
    Let $Q_{n,k}^*$, $q_{n,k}$, be as in the proof of Lemma \ref{lem:CubesOnM}, and let $y_{n,k}$, $\hat y_{n,k}$ be as introduced in \eqref{eqWhitney:DefOfYnkHat}.  As noted in the proof of Lemma \ref{lem:CubesOnM}, the common side length of each of the cubes $Q_{n,k}$ is at most $4\delta$. Since $q_{n,k}$ is the center of $Q_{n,k}$, it follows that $Q_{n,k}^*$ lies inside the ball of radius $(\tfrac{9}{8})2\delta \sqrt{d}$ centered at $q_{n,k}$. In this case, Remark \ref{rem: good balls cover M} says that
    \begin{align*}	
       P_{n,k}^* \subset B_{(\tfrac{101}{100})(\tfrac{9}{4})\delta \sqrt{d}}(y_{n,k}).
    \end{align*}	
    At the same time, $d(\hat y_{n,k},y_{n,k}) = d(y_{n,k},\tilde G_n)\leq \tilde h_n$, and $\tilde h_n \leq \delta/500$ by \eqref{eqFinDimApp: tilde hn def}. Then, the triangle inequality yields $d(x,\hat y_{n,k}) \leq d(x,y_{n,k})+d(\hat y_{n,k},y_{n,k})\leq \diam(P_{n,k}^*)+\tilde h_n \leq (\tfrac{101}{100})(\tfrac{9}{4})\delta\sqrt{d} + \tfrac{1}{500}\delta$, for $x\in P_{n,k}^*$. This shows that,
    \begin{align*}	
      P_{n,k}^* \subset B_{3\delta \sqrt{d}}(\hat y_{n,k}).
    \end{align*}	
    In light of the discussion at the beginning of this section, we know that $\exp_{\hat y_{n,k}}^{-1}$ is well defined and smooth in the larger ball $B_{4\delta \sqrt{d}}(\hat y_{n,k})$. Therefore, we conclude that given $p$ or $D$ the functions $l(p,\hat y_{n,k};\cdot)$ and $q(D,\hat y_{n,k};\cdot)$ are well defined functions in $P_{n,k}^*$ which are also smooth.

  \end{rem}

  We refer the reader to the Appendix (Definition \ref{def:Appendix Discrete Gradient}, \ref{def:Appendix Discrete Hessian}) for the definition of the discrete gradient and discrete Hessian,
  \begin{align*}
    \nabla^1_n u(x) \in (TM)_x,\;\;\;\nabla^2_n u(x) \in \L((TM)_x)
  \end{align*}
  defined for every $x \in \tilde G_n$. With these, we introduce the local interpolation operators $p^\beta_{u,k}(x)$. These are real valued functions defined as follows, recall $\hat y_{n,k}$ from \eqref{eqWhitney:DefOfYnkHat}, then
  \begin{align*}
    p^\beta_{u,k}: P_{n,k}^* \to \mathbb{R},
  \end{align*}	  
  is defined as follows
  \begin{align}
      p^\beta_{(u,k)}(x) & := 
	  \begin{cases} 
        u(\hat y_{n,k}) & \textnormal{ if } \beta \in(0,1)\\	 
        u(\hat y_{n,k})+l(\nabla^1_n u (\hat y_{n,k}),\hat y_{n,k};x) & \textnormal{ if } \beta \in [1,2)\\
        u(\hat y_{n,k})+l(\nabla^1_n u(\hat y_{n,k}), \hat y_{n,k};x)+q(\nabla^2_n u(\hat y_{n,k}), \hat y_{n,k};x) & \textnormal{ if } \beta = 2.
      \end{cases}\label{eqWhitney:DefOfPbetaU}     
  \end{align}
  Thus, $p^\beta_{(u,k)}$ yields respectively a constant/first order/second order approximation to $u$ in $P_{n,k}^*$.
  
  Using the chain rule, and the smoothness of $\exp_{\hat y_{n,k}}^{-1}$  in $P_{n,k}^*$ (as explained in Remark \ref{rem:local interpolators are well defined}), we have the following proposition.
  \begin{prop}\label{prop: local interpolation operators}
    The following estimate holds with a constant independent of $n$:
    \begin{align*}
      \|p^\beta_{(u,k)}\|_{C^{\beta}(P_{n,k}^*)} \leq \left \{ \begin{array}{ll}
        C |u(\hat y_{n,k})|  & \textnormal{ if } \beta \in (0,1) \\	  
        C \left ( | \nabla^1_n u(\hat y_{n,k})|_{g_{\hat y_{n,k}}} + |u(\hat y_{n,k})|  \right ) & \textnormal{ if } \beta \in [1,2)\\
	C \left ( | \nabla^2_n u(\hat y_{n,k})|_{g_{\hat y_{n,k}}}+ | \nabla^1_n u(\hat y_{n,k})|_{g_{\hat y_{n,k}}} + |u(\hat y_{n,k})| \right ) & \textnormal{ if } \beta=2.\\
      \end{array}\right.
    \end{align*}		
  \end{prop}

  \subsection{The Whitney extension} With the partition of unity $\{\phi_{n,k}\}_k$ and the local interpolation operators at hand, we are ready to introduce the Whitney extension operators $E^{\beta}_n$.
  \begin{DEF}\label{def:WhitneyExt}
  For each $n$, we define
  \begin{enumerate}
    \item The restriction operator $\tilde T_n: C_b^{\beta}(M)\to C(\tilde G_n)$, defined in the usual manner
    \begin{align*}
      \tilde T_n(u,x) := u(x)\;\;\forall\;x\in \tilde G_n.
    \end{align*}
    \item The extension operator of order $\beta$, $\tilde E_n^{\beta}: C(\tilde G_n) \to C^{\beta}_b(M)$ defined by
    \begin{align*}
      \tilde E_n^{\beta}(u,x) & := \begin{cases} 
	u(x) & \textnormal{ if } x \in \tilde G_n\\
	\sum \limits_{k} p^\beta_{(u,k)}(x)\phi_{n,k}(x) & \textnormal{ if } x\not\in \tilde G_n.
      \end{cases}	
    \end{align*}	  
    \item The ``projection'' map $\tilde \pi_n^{\beta} : C^{\beta}_b(M) \mapsto C^{\beta}_b(M)$ defined by
    \begin{align*}
      \tilde \pi_n^{\beta} := \tilde E_n^{\beta}\circ \tilde T_n.	
    \end{align*}
  \end{enumerate}	
  \end{DEF}

\begin{rem} The fact that $\tilde E_n^\beta$ and $\tilde \pi_n^\beta$ map to $C^\beta_b$ is not at all trivial, and it will be proved below in Theorem \ref{thm:EnTnContOperatorInCbetaNorm}.
\end{rem}

\begin{rem}\label{rem:extension preserves compact support}
On the other hand, it is not difficult to see that if $u \in C(\tilde G_n)$ vanishes in $\tilde G_n\setminus G_n$ then $\tilde E_n^\beta(u)$ vanishes outside $M_{n+2}$ and in particular has compact support. Indeed, by recalling \eqref{eqFinDimApp: Mn def}, \eqref{eqFinDimApp: Gn def}, and the definition of $P_{n,k}$, one can show in this case that for $x\not\in M_{n+2}$ and $k\in K_x$ one has that $u(\hat y_{n,k}),\nabla^1_nu(\hat y_{n,k}),$ and $\nabla^2_nu(\hat y_{n,k})$ all vanish, and thus $\tilde E_n^\beta(u,x)=0$ for $x \not \in M_n$.

Accordingly, if $u \in C^{\beta}_b(M)$ is a function with compact support, then for $n$ large enough $\tilde T_n\circ u$ vanishes in $\tilde G_n \setminus G_n$,  and it follows $\tilde \pi_n^\beta(u)$ is compactly supported inside $M_n$.

\end{rem}
  
Our immediate goal is controlling the regularity of $\tilde \pi_n^\beta u$ in terms of $u$. For the sake of notation, we shall write for the rest of this section
\begin{align}\label{eqWhitney:FisPinU}
  f(x) := \tilde \pi_n^\beta(u,x).
\end{align}

The following propositions, leading to Theorem \ref{thm:EnTnContOperatorInCbetaNorm}, intend show that the maps $\tilde \pi_n^\beta$ are well behaved with respect to the $C^\beta_b$ norm in a manner which is independent of the sets $\tilde G_n$.  The validity of these bounds in a manner that does not depend on the set $\tilde G_n$ is a crucial feature of the Whitney extension.  

Among these propositions, we highlight two. First, we have Proposition \ref{prop:BoundaryEstimate}, which says $\tilde \pi_n^\beta u(x)$ (and its respective derivatives) approach $u(x)$ as $x$ approaches $\tilde G_n$. Meanwhile, Proposition \ref{prop:InteriorEstimate} states that away from $\tilde G_n$ the functions $\tilde \pi_n^\beta u(x)$ have the correct regularity. Once again, we remind the reader that these estimates are standard for the Whitney extension when $M=\mathbb{R}^d$, and refer to \cite[Chap. 6, Section 2]{Stei-71}). Here we review their straightforward adaptation to more general $M$ for the sake of completeness.

\begin{prop}\label{prop:BoundaryEstimate}
  Let $x\in M\setminus \tilde G_n$ and $u\in C^\beta_b(M)$. There is a universal constant $C$ such that, if $\beta \in [0,3)$ and $f(x) := \tilde \pi_n^\beta(u,x)$, we have 
  \begin{align*}
    |f(x)-f(\hat x)| \leq C\|u\|_{C^\beta(M)}d(x,\tilde G_n)^{\min\{1,\beta\}}
  \end{align*}	    
  Furthermore,
  \begin{align*}
    |\nabla_a f(x)-\nabla_a f(\hat x)| \leq C\|u\|_{C^\beta(M)}d(x,\tilde G_n)^{\min\{1,\beta-1\}},\;\;\textnormal{ if } \beta \geq 1,\\
    |\nabla_{ab}^2 f(x)-\nabla_{ab}^2 f(\hat x)| \leq C\|u\|_{C^\beta(M)}d(x,\tilde G_n)^{\min\{1,\beta-2\}},\;\;\textnormal{ if } \beta \geq 2.	
  \end{align*}	
  Here, $\grad_a$ and $\grad^2_{ab}$ are respectively the components of the first and second covariant derivatives of $f$ with respect to an orthonormal frame.
\end{prop}

\begin{proof}
    For the sake of explaining the key ideas of the proof without getting distracted with technicalities, we postpone the proof of the higher derivatives estimates to Section \ref{subsec:proofs beta bigger than one}. 	

  Let $x \in M\setminus \tilde G_n$ and $\hat x \in \tilde G_n$ such that $d(x,\tilde G_n)=d(x,\hat x)$. Recalling that $f = u$ on $\tilde G_n$, and using the first property of $\{\phi_{n,k}\}_k$ from Lemma \ref{lem:PartitionOfUnityGn}, we have that $f(\hat x)-f(x)$ is equal to 	
    \begin{align*}		
      & u(\hat x)- \sum \limits_{k} u(\hat y_{n,k})\phi_{n,k}(x) \ \textnormal{ if } \beta \in [0,1),\\	
      & u(\hat x)- \sum \limits_{k} (u(\hat y_{n,k})+l(\nabla_n^1u(\hat y_{n,k}),\hat y_{n,k};x ))\phi_{n,k}(x) \ \textnormal{ if } \beta \in [1,2),\\
      & u(\hat x)- \sum \limits_{k} (u(\hat y_{n,k})+l(\nabla_n^1u(\hat y_{n,k}),\hat y_{n,k};x )+q(\nabla_n^2u(\hat y_{n,k}),\hat y_{n,k};x ) )\phi_{n,k}(x) \ \textnormal{ if } \beta \in [2,3).	  
    \end{align*}	  
   Let us consider each case individually. If $\beta \in [0,1)$, the identity $\sum \limits_k \phi_{n,k}(x)=1$ allows us to write 
   \begin{align*}	  
      f(\hat x)-f(x) & = \sum \limits_{k} (u(\hat x)-u(\hat y_{n,k}))\phi_{n,k}(x)\\
  	  & = \sum \limits_{k\in K_x} (u(\hat x)-u(\hat y_{n,k}))\phi_{n,k}(x). 
    \end{align*}
    The set $K_x$ being the one defined in \eqref{eqFinDim:Set Of Covering Cubes}. Then, the triangle inequality and $|\phi_{n,k}|\leq 1$ yields
    \begin{align*}
      |f(\hat x)-f(x)| \leq \|u\|_{C^\beta}\sum \limits_{k\in K_x} d(\hat x,\hat y_{n,k})^{\beta}.
    \end{align*}
    The triangle inequality says that
    \begin{align*}	 
      d(\hat x,\hat y_{n,k}) & \leq d(\hat x,x)+d(x,y_{n,k})+d(y_{n,k},\hat y_{n,k}),
    \end{align*}
    where, according to \eqref{eqWhitney:DefOfYnkHat}, $d(y_{n,k},\hat y_{n,k}) = d(y_{n,k}, \tilde G_n)$. In this case, we see that $d(y_{n,k},\hat y_{n,k})\leq d(x,\tilde G_n)+d(x,y_{n,k})$, and we conclude that
    \begin{align*}		  
      d(\hat x,\hat y_{n,k})  & \leq 2d(x,\tilde G_n)+2\diam(P_{n,k}^*).					
    \end{align*}  
    Then, using Remark \ref{rem: Whitney local diameter of balls}, it follows that
    \begin{align}	 
      d(\hat x,\hat y_{n,k}) & \leq 16d(x,\tilde G_n)	\;\;\forall\;k\in K_x.\label{eqFinDim:hat x and hat Ynk distance is bounded by distance to tilde Gn}				
    \end{align}  	
    Furthermore, recall \eqref{eqFinDim:Size Set Of Covering Cubes} which says that $\#K_x\leq N$. All of this leads to the estimate
    \begin{align*}
      |f(\hat x)-f(x)| \leq C\|u\|_{C^\beta} d(x,\tilde G_n)^{\beta}.
    \end{align*}
    Instead, if $\beta \in [1,2)$, we have	
   \begin{align*}	  
      f(\hat x)-f(x) & = \sum \limits_{k} (u(\hat x)-u(\hat y_{n,k}))\phi_{n,k}(x) + \sum \limits_{k}l(\nabla_n^1u(\hat y_{n,k}),\hat y_{n,k};x ))\phi_{n,k}(x)\\
  	  & = \sum \limits_{k\in K_x} (u(\hat x)-u(\hat y_{n,k}))\phi_{n,k}(x)+ \sum \limits_{k\in K_x}l(\nabla_n^1u(\hat y_{n,k}),\hat y_{n,k};x ))\phi_{n,k}(x).
    \end{align*}	
   Just as before, using the triangle inequality and the fact that $|\phi_n,k|\leq 1$, it follows that 	
   \begin{align*}	  
      |f(\hat x)-f(x)| & \leq \sum \limits_{k\in K_x} |u(\hat x)-u(\hat y_{n,k})|+ \sum \limits_{k\in K_x}|l(\nabla_n^1u(\hat y_{n,k}),\hat y_{n,k};x )|\\
	  & \leq \|u\|_{C^1}\sum \limits_{k\in K_x} d(\hat x,\hat y_{n,k})+\sum \limits_{k\in K_x} |\nabla_n^1 u (\hat y_{n,k})|_{g_x}d(\hat y_{n,k},x).
    \end{align*}	
   Proposition \ref{prop: discrete derivatives are controlled by continuous derivatives} in the Appendix guarantees that $|\nabla_n^1 u (\hat y_{n,k})|_{g_x} \leq C\|u\|_{C^1}$, for some universal $C$. Using this bound, the fact that $\|u\|_{C^1}\leq \|u\|_{C^\beta}$ for $\beta\geq 1$, and the last inequality above, it follows that
   \begin{align*}	  
      |f(\hat x)-f(x)| & \leq C\|u\|_{C^\beta}\sum \limits_{k\in K_x} d(\hat x,\hat y_{n,k}).
   \end{align*}	  		
   From this point one argues exactly as done for $\beta \in [0,1)$ to conclude that
   \begin{align*}	  
      |f(\hat x)-f(x)| & \leq C\|u\|_{C^\beta}d(x,\tilde G_n).
   \end{align*}	  		   
   The proof for $\beta \in [2,3)$ is entirely analogous, and we leave the details to the reader. This proves the first estimate.
   	
    As mentioned above, we refer to Section \ref{subsec:proofs beta bigger than one} for the proofs for $\nabla_a f$ and $\nabla_{ab}^2 f$.
	
\end{proof}

We delay the technical proof of the following auxiliary proposition until the Appendix \ref{sec:ProofOfRegAtDistToBoundary}.

\begin{prop}\label{prop:RegularityOfEnTnDistToGn}
  Let $x  \in M\setminus \tilde G_n$ and $u\in C^\beta$. There is a universal constant $C$ such that the following bounds hold. First, if $0\leq\beta<1$, 
  \begin{align*}	
    |\nabla (\tilde E_n^\beta\circ \tilde T_n)u(x)|\leq C\|u\|_{C^\beta}d(x,\tilde G_n)^{\beta-1}.
  \end{align*}	
  If $1\leq\beta<2$, we have
  \begin{align*}	
    |\nabla^2(\tilde E_n^\beta\circ \tilde T_n)u(x)|\leq C\|u\|_{C^\beta}d(x,\tilde G_n)^{\beta-2}.
  \end{align*}	
  Finally, if $2\leq\beta<3$, we have
  \begin{align*}	
    |\nabla^3(\tilde E_n^\beta\circ \tilde T_n)u(x)|\leq C\|u\|_{C^\beta}d(x,\tilde G_n)^{\beta-3}.
  \end{align*}	
 
\end{prop}

Using Proposition \ref{prop:RegularityOfEnTnDistToGn} it is easy to show that $\tilde \pi_n^\beta(u,x)$ is regular away from $\tilde G_n$.
\begin{prop}\label{prop:InteriorEstimate}
  Let $x_0$ and $r$ be fixed such that $B_{4r}(x_0) \subset M \setminus  \tilde G_n$. Then, given $u\in C^\beta_b(M)$, for $f=\tilde \pi_n^\beta u$ we have the estimate
  \begin{align*}
    [\nabla^i f]_{C^{\beta-i} (B_r(x_0))} \leq C\|u\|_{C^\beta(M)},\;\;\textnormal{ for } \beta \in [i,i+1),\;\;i=0,1,2.
  \end{align*}	    
\end{prop}

\begin{proof}
	Here, we only prove the statement for $\beta\in(0,1)$, and we defer the remaining two cases until later, in Section \ref{subsec:proofs beta bigger than one}.

  \textbf{The case $\beta \in (0,1)$.} Let $x_1,x_2 \in B_r(x_0)$, and $x(t):[0,L]\to M$ a minimal geodesic between them, parametrized with respect to arc length, so $L = d(x_1,x_2)$. Then, by the triangle inequality
  \begin{align*}
    d(x(t),x_0) & \leq d(x(t),x_1)+d(x_1,x_0) \\
	  & \leq d(x_1,x_2)+d(x_1,x_0)\\
	  & \leq d(x_1,x_0)+d(x_2,x_0)+d(x_1,x_0) \leq 3r.	  
  \end{align*}	  
  In particular, it follows that $d(x(t),\tilde G_n) \geq r$ for all $t\in [0,L]$. Then,
  \begin{align*}
    |f(x_1)-f(x_2)| & = \int_0^L \frac{d}{dt}f(x(t))\;dt\\
	  & = \int_0^L(\nabla f(x(t)),\dot x(t))\;dt \leq C\|u\|_{C^\beta(M)}r^{\beta-1}L,
  \end{align*}
  the last inequality being thanks to Proposition \ref{prop:RegularityOfEnTnDistToGn} and the fact that $d(x(t),\tilde G_n)\geq r$ for all $t$. Since $d(x_1,x_2)\leq 2r$ and $\beta-1<0$, we conclude that
  \begin{align*}
    |f(x_1)-f(x_2)| \leq C\|u\|_{C^\beta(M)}r^{\beta-1}d(x_1,x_2) \leq C\|u\|_{C^\beta(M)}d(x_1,x_2)^{\beta}.
  \end{align*}
  The remaining cases (those with $\beta\geq 1$) are proved in Section \ref{subsec:proofs beta bigger than one}.
\end{proof}

For readers with a background in elliptic PDE, and in particular, those not familiar with the Whitney extension, it may be useful to make a na\"ive but possibly illustrative analogy with the derivation of global regularity estimates for solutions of elliptic equations.  Proposition \ref{prop:InteriorEstimate} is a kind of interior estimate, where in order to bound the solution in a ball, one needs the ``equation'' (here, being the extension) to take place in a bigger ball. Likewise, Proposition \ref{prop:BoundaryEstimate} is analogous to estimates at the boundary. In this sense, $\tilde G_n$ is the kind of boundary and $u$ provides the boundary values. Furthermore, the way these two estimates are ``glued'' in the next proof bears a great resemblance to the proof of global regularity estimates for elliptic equations from interior and boundary estimates.

With the previous two estimates in hand, we are ready to prove that $\tilde \pi_n^\beta$ is a bounded map from $C^\beta_b$ to $C^\beta_b$.
\begin{thm}\label{thm:EnTnContOperatorInCbetaNorm}
  If $u\in C^\beta_b(M)$, then $\tilde \pi_n^\beta \in C^\beta_b(M)$ and, for some universal $C$,
  \begin{align*}
    \|\tilde \pi_n^\beta u\|_{C^\beta(M)}\leq C\|u\|_{C^\beta(M)}.		
  \end{align*}		
\end{thm}

\begin{proof}
  As before we write $f=\tilde \pi_n^\beta u$.  Let us first show
  \begin{align*}	
    \|f\|_{L^\infty} \leq C\|u\|_{C^\beta},	
  \end{align*}
  for all $\beta \in [0,2]$. If $x \in M\setminus \tilde G_n$, then
  \begin{align*}
    f(x) = \sum \limits_{k} p^\beta_{(u,k)}(x)\phi_{n,k}(x).
  \end{align*}
  Proposition \ref{prop: local interpolation operators} implies that
  \begin{align*}
    \sup \limits_{x\in P_{n,k}^*}|p^\beta_{(u,k)}(x)| \leq C\|u\|_{C^\beta}.
  \end{align*}
  Then,
  \begin{align*}
    \sup \limits_{x\in M}|f(x)|\leq C\|u\|_{C^\beta}.
  \end{align*}
  Let us now prove $f(x)$ has the right regularity. The argument is separated in cases depending on $\beta$, in each case the proof will consist in ``gluing'' the interior and boundary estimates proved for $\tilde \pi_n^\beta$ in Propositions \ref{prop:InteriorEstimate} and \ref{prop:BoundaryEstimate}.
  
  \textbf{The case $\beta \in [0,1)$.}  Let $x_1,x_2\in M \setminus \tilde G_n$. If $4d(x_1,x_2) < \max\{d(x_1,\tilde G_n),d(x_2,\tilde G_n)\}$, then we can apply Proposition \ref{prop:InteriorEstimate} and conclude that
  \begin{align*}
    |f(x_1)-f(x_2)| & \leq C\|u\|_{C^\beta}d(x_1,x_2)^{\beta}.
  \end{align*}
  Consider on the other hand the case $4d(x_1,x_2)\geq \max\{d(x_1,\tilde G_n),d(x_2,\tilde G_n)\}$, then, for $\hat x_i\in \tilde G_n$ such that $d(x_i,\hat x_i)=d(x_i,\tilde G_n)$ we have
  \begin{align*}
    |f(x_1)-f(x_2)| & \leq |f(x_1)-f(\hat x_1)|+|f(\hat x_1)-f(\hat x_2)|\\
	  & \;\;\;\;+ |f(\hat x_2)-f(x_2)|.
  \end{align*}	  
  Applying Proposition \ref{prop:BoundaryEstimate} to the first and third terms, and recalling that $f(\hat x_i)=u(\hat x_i)$,
  \begin{align*}
    |f(x_1)-f(x_2)| & \leq C\|u\|_{C^\beta}(d(x_1,\hat x_1)^\beta+d(x_2,\hat x_2)^\beta)+\|u\|_{C^\beta}d(\hat x_1,\hat x_2)^\beta.
  \end{align*}	  
  Given that in this case we have $d(x_1,\hat x_1)+d(x_2,\hat x_2)\leq 8d(x_1,x_2)$, we can use the triangle inequality to conclude that $d(\hat x_1,\hat x_2) \leq 10d(x_1,x_2)$, therefore
  \begin{align*}
    |f(x_1)-f(x_2)| & \leq C\|u\|_{C^\beta}d(x_1,x_2)^\beta.
  \end{align*}
  Combining the above estimates we obtain the desired bound for $\beta \in [0,1)$. 
  
  \textbf{The case $\beta\in [1,2)$.} Let us show first that if $u \in C^1_b$, then $f\in C^1_b$, and $\nabla f(x)=\nabla_n^1u(x)$ for every $x\in \tilde G_n$. In order to do this, we shall show that $\nabla_a f(x)$ is continuous in $x$ for every index $a$. Note that 
  \begin{align*}
      \nabla_{a}f(x) & = \sum \limits_{k} \nabla_a \left ( l( \nabla^1_n u (\hat y_{n,k}),\hat y_{n,k};x)\right )\phi_{n,k}(x)\\
  	& \;\;\;\;+\sum \limits_{k} \left ( u(\hat y_{n,k})+l( \nabla^1_n u (\hat y_{n,k}),\hat y_{n,k};x)\right )\nabla_a \left ( \phi_{n,k}(x) \right ).
  \end{align*}
  Recall that any point $x_0\in M\setminus \tilde G_n$ has a neighborhood where at most $N$ of the terms in the above sums are non-zero. Since each term is continuous in $x$, it follows that $\nabla_af(x)$ is continuous in $M\setminus \tilde G_n$. It remains to show the continuity for a point $x_0 \in \tilde G_n$. Let us also recall Remark \ref{rem: Whitney local diameter of balls}, which says that for any $x\in M\setminus \tilde G_n$ we have
  \begin{align*}
    \diam(P_{n,k}^*)\leq 7d(x,\tilde G_n),\;\;\forall\;k\in K_x,
  \end{align*}
  where $K_x$ was defined in \eqref{eqFinDim:Set Of Covering Cubes}. Then, since $d(x_0,\tilde G_n \setminus \{x_0\})>0$, it follows that if $x$ is sufficiently close to $x_0\in \tilde G_n$, then $x\in P_{n,k}^*$ implies that there is a unique closest point in $\tilde G_n$ to $y_{n,k}$, $x_0$ itself. In other words (recall $\hat y_{n,k}$ was defined in \eqref{eqWhitney:DefOfYnkHat}),
  \begin{align*}
    \hat y_{n,k} = x_0\;\;\forall\;x\in K_x.
  \end{align*}
  This means that if $x$ is sufficiently close to $x_0$, $\nabla_{a}f(x)$ has the form
  \begin{align*}
      \nabla_{a}f(x) & = \sum \limits_{k} \nabla_a \left ( l( \nabla^1_n u (x_0),x_0;x)\right )\phi_{n,k}(x)\\
  	& \;\;\;\;+\sum \limits_{k} \left ( u(x_0)+l( \nabla^1_n u (x_0),x_0;x)\right )\nabla_a \left ( \phi_{n,k}(x) \right )\\
        &  = \nabla_a \left ( l( \nabla^1_n u (x_0),x_0;x)\right ).  
  \end{align*}
  Where we used that $\phi_{n,k}$ is a partition of unity: (1) in Lemma \ref{lem:PartitionOfUnityGn} and the identity \eqref{eqFinDim:Partition of Unity Differentiation Identity} to obtain the last identity. From the last inequality we see that as $x\to x_0$ we have $\nabla_a f(x) \to \nabla_af(x_0) = \nabla_n^1u(x_0)$ and thus $u \in C^1_b$.
  
  Next, let us show the H\"older bound for $\beta \in (1,2)$. Let $x_1,x_2\in M_n$. If $4d(x_1,x_2) < \max\{d(x_1,\tilde G_n),d(x_2,\tilde G_n)\}$, then we can apply Proposition \ref{prop:InteriorEstimate} and conclude that
  \begin{align*}
    |\nabla_af(x_1)-\nabla_af(x_2)| \leq C\|u\|_{C^\beta}d(x_1,x_2)^{\beta-1}.
  \end{align*}
  If instead we have $4d(x_1,x_2)\geq \max \{d(x_1,\tilde G_n),d(x_2,\tilde G_n)\}$, then the triangle inequality yields
  \begin{align*}
    |\nabla_af(x_1) -\nabla_af(x_2)| \leq |\nabla_af(x_1) -\nabla_af(\hat x_1)|+|\nabla_af(\hat x_1) -\nabla_af(\hat x_2)|+|\nabla_af(\hat x_2) -\nabla_af(x_2)|.
  \end{align*}
  Therefore, by Proposition \ref{prop:BoundaryEstimate}
  \begin{align*}
    |\nabla_af(x_1) -\nabla_af(x_2)| \leq C\|u\|_{C^\beta}d(x_1,\hat x_1)^{\beta-1}+|\nabla_af(\hat x_1) -\nabla_af(\hat x_2)|.
  \end{align*}
  On the other hand, since in a neighborhood of $\hat x_i$ we have $\grad \tilde \pi_n^\beta u(x)=\nabla^1_n u(\hat x_i)$ (see Def \ref{def:WhitneyExt} and \ref{def:Appendix Discrete Gradient}), the first half of Proposition \ref{prop:InterpolatedDerivativesAreRegular} says that
  \begin{align*}
    |\nabla_af(\hat x_1) -\nabla_af(\hat x_2)|\leq C\|u\|_{C^\beta}d(\hat x_1,\hat x_2)^{\beta-1}.
  \end{align*}
  Hence, it follows that
  \begin{align*}
    |\nabla_af(x_1) -\nabla_af(x_2)| \leq C\|u\|_{C^\beta}d(x_1,x_2)^{\beta-1},\;\;\forall\;x_1,x_2\in M_{n}.
  \end{align*}
  
  \textbf{The case $\beta \in [2,3)$.} An argument entirely analogous to that used for $\beta \in [1,2)$ shows that if $u\in C^2_b$, then $f\in C^2_b$, with $\nabla^i f(x) = \nabla^i_n u(x)$ for $i=1,2$ and every $x\in \tilde G_n$.
  
  Then, let us prove that $\nabla^2_{ab}f(x)$ are H\"older continuous for $\beta>2$. As in the previous cases, suppose first that $4d(x_1,x_2)$ is no larger than $\max\{d(x_1,\tilde G_n),d(x_2,\tilde G_n)\}$. In this case, Proposition \ref{prop:InteriorEstimate} yields
  \begin{align*}
    |\nabla_{ab}^2f(x_1)-\nabla_{ab}^2f(x_2)|\leq C\|u\|_{C^\beta}d(x_1,x_2)^{\beta-2}.
  \end{align*}
  Consider now the case where $4d(x_1,x_2) \geq \max\{d(x_1,\tilde G_n),d(x_2,\tilde G_n)\}$. We shall argue in a parallel manner to the case $\beta \in [1,2)$. First off, we have
  \begin{align*}
    |\nabla_{ab}^2f(x_1)-\nabla_{ab}^2f(x_2)| & \leq |\nabla_{ab}^2f(x_1)-\nabla_{ab}^2f(\hat x_1)| +|\nabla_{ab}^2f(\hat x_1)-\nabla_{ab}^2f(\hat x_2)|\\
      & \;\;\;\;+|\nabla_{ab}^2f(\hat x_2)-\nabla_{ab}^2f(x_2)|.   	 
  \end{align*}	  
  Next, since in a neighborhood of $\hat x_i$ we have $\grad^2 \tilde \pi_n^\beta u(x)=\nabla^2_n u(\hat x_i)$ (see Def \ref{def:WhitneyExt} and \ref{def:Appendix Discrete Gradient}), the second part of Proposition \ref{prop:InterpolatedDerivativesAreRegular} says that
  \begin{align*}
    |\nabla_{ab}^2f(\hat x_1)-\nabla_{ab}^2f(\hat x_2)|\leq C\|u\|_{C^\beta}d(\hat x_1,\hat x_2)^{\beta-2}.
  \end{align*}	     
  Since $d(\hat x_1,\hat x_2) \leq 10d(x_1,x_2)$ and $4d(x_1,x_2)$ is larger than $d(x_1,\tilde G_n)$ and $d(x_2,\tilde G_n)$,
  \begin{align*}
    |\nabla_{ab}^2f(x_1)-\nabla_{ab}^2f(x_2)|\leq C\|u\|_{C^\beta}d(x_1,x_2)^{\beta-2},\;\;\forall\;x_1,x_2\in M_{n}.	  
  \end{align*}	   
  This concludes the proof of the theorem.
\end{proof}

The operators $\tilde \pi_n^\beta$ also enjoy the useful property of having a finite range of dependence, a property that will play a role in some arguments of Section \ref{sec: Min-Max formula for M,g}.

\begin{lem}\label{lem:EnTnRangeOfDependence}
  Assume that $\beta\in [0,3)$.  There is a universal constant $C$, such that if $K,K'\subset M$ are open sets such that $d(K', M\setminus K)\geq r+10^3h_n$, then
  \begin{align*}
    \norm{\tilde \pi_n^\beta u -\tilde \pi_n^\beta v}_{C^\beta(K)}\leq C(1+r^{-\beta}) \norm{u-v}_{C^\beta(K')},\;\;\forall\;u,v\in C^\beta_b(M).
  \end{align*}
\end{lem}

\begin{proof}
  Given $K$ and $K'$ with $d(K',M\setminus K)\geq r+10^3h_n$. It will be convenient to introduce an ``intermediate'' set, 
  \begin{align*}
    \tilde K:= \{ x\in M \mid d(x,K)\leq 400\tilde h_n \}.
  \end{align*}
  In other words, $\tilde K$ is the closure of the $400 \tilde h_n$-neighborhood of $K$. Thanks to the triangle inequality and the assumption on $K$ and $K'$ we have $d(\tilde K, M\setminus K')\geq r$.
  
  \noindent Next, we construct a function $\eta = \eta_{\tilde K,K'}$ such that 
  \begin{align*}
    0\leq \eta\leq 1,\;\;\eta \equiv 1 \textnormal{ in } \tilde K,\;\;\eta \equiv 0 \textnormal{ in } M\setminus K'.
  \end{align*}
  It is not difficult to see that $\eta$ can be chosen so that (for some universal $C$)
  \begin{align*}
    \|\eta\|_{C^\beta(M)} \leq C\left (1+\frac{1}{d(K',M\setminus K'')^\beta} \right ) \leq C\left ( 1+r^{-\beta} \right ).
  \end{align*}	 
  In fact, this can be done using the Whitney decomposition itself, see the ``regularized distance'' construction in \cite[Chapter 6, Section 2.1]{Stei-71} . 
  
  Let $u,v\in C^\beta_b$. As proved later in Proposition \ref{prop:Appendix Estimate on Locality of the Extension} (see Appendix), if $x\in M$ and $w \in C^\beta_b$ is identically zero in $B_{400 \tilde h_n}(x)$, then
  \begin{align*}
    \tilde \pi_n^\beta w \equiv 0 \textnormal{ in } B_{100\tilde h_n}(x).  
  \end{align*}
  We apply this to the function $w = (u-v)-\eta (u-v) \in C^\beta_b$, and to any $x\in K$, making use of the fact that $w$ vanishes in $\tilde K$.  It follows that for every $x$ in a small neighborhood of $K$ we have
  \begin{align*}
    \tilde \pi_n^\beta (u,x) = \tilde \pi_n^\beta(\eta u,x).	  
  \end{align*}	  
  In this case, it is clear that $\|\tilde \pi_n^\beta u\|_{C^\beta(K)}\leq \|\tilde \pi_n^\beta (\eta u)\|_{C^\beta(M)}$. Then, by Theorem \ref{thm:EnTnContOperatorInCbetaNorm}, 
  \begin{align*}
    \|\tilde \pi_n^\beta u\|_{C^\beta(K)} \leq C\|\eta u\|_{C^\beta(M)}.
  \end{align*}
  Using that $\eta \equiv 0$ in $M\setminus K'$, the Leibniz rule, and the bound on $\|\eta\|_{C^\beta_b}$, we conclude that
  \begin{align*}
    \|\tilde \pi_n^\beta u\|_{C^\beta(K)} \leq C(1+r^{-\beta})\|u\|_{C^\beta(K')}.
  \end{align*}
 
\end{proof}

\begin{lem}\label{lem:EnTnEstToUInC3}
  There is universal constant $C$ such that for any $u \in C^{3}_b(M)$ we have
  \begin{align*}
    \|\tilde \pi_n^\beta u-u\|_{C^{\beta}(M)} & \leq Ch_n^\gamma \|u\|_{C^3(M)}.
  \end{align*}		
  Here, $\gamma= i-\beta$ if $\beta \in [i-1,i)$, for $i=1,2,3$.
\end{lem}

\begin{rem}\label{rem:C3 norm assumption can be weakened}
  It will be evident from the proof, that the $C^3_b$ norm on the right hand side can be weakened when $\beta<2$. Here we simply state the lemma with $C^3_b$ for the sake of brevity.
\end{rem}

\begin{proof}
  Recall we are writing $f$ for $\tilde \pi_n^\beta u$, and that given $x\in M$ we write $\hat x$ for an element of $\tilde G_n$ for which $d(x,\hat x)=d(x,\tilde G_n)$. 
  
  We begin by estimating $\|f-u\|_{L^\infty(M)}$. Since $f\equiv u$ in $\tilde G_n$, we have
  \begin{align*}
    |f(x)-u(x)| & \leq |f(x)-f(\hat x)|+|u(\hat x)-u(x)|\\
      & \leq \|f\|_{C^1} d(x,\hat x)+\|u\|_{C^1}d(x,\hat x)\\
      & \leq C\|u\|_{C^3}d(x,\hat x).	
  \end{align*}
  Then, regardless of $\beta$ we have,
  \begin{align*}
    \sup \limits_{x \in M}|f(x)-u(x)| & \leq \sup\limits_{x\in M} C\|u\|_{C^3}d(x,\tilde G_n) \leq C\|u\|_{C^3}\tilde h_n.
  \end{align*}   
  Note this already shows $\|f-u\|_{L^\infty(M)}$ goes to zero with a rate determined by $\tilde h_n$.  To bound $\|f-u\|_{C^\beta(M)}$, it remains to control the H\"older seminorm of either $u$, $\nabla u$, or $\nabla^2  u$, depending on the range where $\beta$ lies. Let us treat the case $\beta \in (0,1)$ first, which means we must estimate $[u]_{C^\beta}$. We defer the proof of the remaining two cases ($\beta \in [1,2)$ and $\beta \in [2,3$) until later,  in Section \ref{subsec:proofs beta bigger than one}.

  \textbf{The case $\beta \in [0,1)$.} Let $x_1,x_2\in K$. We shall bound
  \begin{align*}
    \frac{|f(x_1)-u(x_1)-(f(x_2)-u(x_2))|}{d(x_1,x_2)^\beta}.
  \end{align*}
  In what follows, it will be useful to fix $\hat x_i \in \tilde G_n$ such that $d(x_i,\hat x_i) = d(x_i,\tilde G_n)$ for $i=1,2$. First, suppose that $d(x_1,x_2) \leq \max \{d(x_1,\hat x_1),d(x_2,\hat x_2) \}$, then
  \begin{align*}
    |f(x_1)-u(x_1)-(f(x_2)-u(x_2))| \leq \|f-u\|_{C^1(M)}d(x_1,x_2),
  \end{align*}
  \begin{align*}
    \frac{|f(x_1)-u(x_1)-(f(x_2)-u(x_2))|}{d(x_1,x_2)^{\beta}} & \leq \frac{\|f-u\|_{C^1}d(x_1,x_2)}{d(x_1,x_2)^{\beta}}.
  \end{align*}	
  Since $\|f-u\|_{C^1} \leq \|f\|_{C^1}+\|u\|_{C^1}$,  Theorem \ref{thm:EnTnContOperatorInCbetaNorm} yields $\|f-u\|_{C^1} \leq C\|u\|_{C^1}$.

  Using that $\beta<1$, we have $d(x_1,x_2)^{1-\beta} \leq \max\{d(x_1,\hat x_1)^{1-\beta},d(x_2,\hat x_2)^{1-\beta} \} \leq \tilde h_n^{1-\beta}$. Then, for this case we have	 
  \begin{align*}	
    \frac{|f(x_1)-u(x_1)-(f(x_2)-u(x_2))|}{d(x_1,x_2)} & \leq C\|u\|_{C^1(M)}\tilde h_n^{1-\beta}.
  \end{align*}
  Second, let us consider the case where $d(x_1,x_2)>\max\{ d(x_1,\hat x_1),d(x_2,\hat x_2)\}$. Then, we proceed by writing
  \begin{align*}
    |f(x_1)-u(x_1)-(f(x_2)-u(x_2))| & \leq |f(x_1)-u(x_1)|+|f(x_2)-u(x_2)|.
  \end{align*}
  Next, due to $f=u$ in $\tilde G_n$, for $i=1,2$ we have
  \begin{align*}
    & |f(x_i)-u(x_i)| = |f(x_i)-u(x_i)-(f(\hat x_i)-u(\hat x_i))| \leq (\|f\|_{C^1}+\|u\|_{C^1})d(x_i,\hat x_i),
  \end{align*}
  and since $\|f\|_{C^1}\leq C\|u\|_{C^1}$ (Theorem \ref{thm:EnTnContOperatorInCbetaNorm}), we conclude that
  \begin{align*}
    & |f(x_i)-u(x_i)|  \leq C\|u\|_{C^1} d(x_i,\hat x_i),\;\;i=1,2.
  \end{align*}
  The assumption $d(x_1,x_2)>\max\{d(x_1,\hat x_1),d(x_2,\hat x_2)\}$ yields that $d(x_i,\hat x_i) \leq d(x_i,\hat x_i)^{1-\beta}d(x_1,x_2)^\beta$ and furthermore $d(x_i,\hat x_i) \leq \tilde h_n^{1-\beta}d(x_1,x_2)^\beta$ both for $i=1,2$ (this uses again that $\beta<1$, since it means that $t\to t^{1-\beta}$ is nondecreasing). Therefore,
  \begin{align*}
    & |f(x_i)-u(x_i)| \leq C\|u\|_{C^1}d(x_i,\hat x_i)^{1-\beta}d(x_1,x_2)^\beta \leq C\|u\|_{C^1}\tilde h_n^{1-\beta}d(x_1,x_2)^\beta.
  \end{align*}
  Then, in this case we also conclude that
  \begin{align*}
    \frac{|f(x_1)-u(x_1)-(f(x_2)-u(x_2))|}{d(x_1,x_2)^{\beta}} & \leq C\|u\|_{C^1}\tilde h_n^{1-\beta}.
  \end{align*}
  Combining the estimates for either case, we conclude that
  \begin{align*}
    [f]_{C^\beta(M)} = \sup \limits_{x_1\neq x_2}\frac{|f(x_1)-u(x_1)-(f(x_2)-u(x_2))|}{d(x_1,x_2)^{\beta}} \leq C\|u\|_{C^1}\tilde h_n^{1-\beta}.
  \end{align*}
  Now, since $\tilde h_n\leq 1$ always, we have $\tilde h_n \leq \tilde h_n^{1-\beta}$ for all $n$, therefore, we have proved that
  \begin{align*}
    \|f-u\|_{C^\beta(M)}  & := \|f-u\|_{L^\infty(M)}+[f]_{C^\beta(M)}\\
	& \leq C\|u\|_{C^1(M)}\tilde h_n+C\|u\|_{C^1(M)}\tilde h_n^{1-\beta},\\
	& \leq C\|u\|_{C^1(M)}\tilde h_n^{1-\beta}.
  \end{align*}
  For the proofs for $\beta \geq 1$, see Section \ref{subsec:proofs beta bigger than one}.

\end{proof}


\subsection{The Whitney extension is almost order preserving}\label{sec:WhitneyAlmostOrderPreserving}

When $\beta \in [0,1)$ it turns out that $\tilde E_n^{\beta}$ preserves the ordering of functions.

\begin{rem}\label{rem: interpolation preserves order if beta less than one}
  Suppose $\beta <1$. If $u,v\in C(\tilde G_n)$ and $u(x)\leq v(x)$ $\forall\; x\in \tilde G_n$, then
  \begin{align*}	
    \tilde E_n^{\beta}(u,x)\leq \tilde E_n^{\beta}(v,x)\;\;\forall\;x\in M.
  \end{align*}
  Indeed, take $u\leq v$ in $\tilde G_n$ and $x \in M\setminus \tilde G_n$. Then, from the definition of $E^{\beta}_n$ when $\beta<1$, we have
  \begin{align*}
    u(x) & = \sum \limits_{k} u(\hat y_{n,k})\phi_{n,k}(x)\\
      & \leq \sum \limits_{k} v(\hat y_{n,k})\phi_{n,k}(x) = v(x),		   	   
  \end{align*}	   
  where we used that  $\phi_{n,k}\geq 0$ and $u(\hat y_{n,k})\leq v(\hat y_{n,k})$ for every $k$.	
\end{rem}
  It is unclear --or rather unlikely-- that the operators continue to be order preserving for $\beta$ larger than $1$. However, when considering the extension among functions in $\tilde G_n$ that are sufficiently regular (in the sense that they are the restriction of smooth functions) then $E^\beta_n$ preserves the ordering up to a small correcting function whose $C^\beta$ norm vanishes as $n$ goes to infinity. It is worthwhile to point out to a recent preprint of Fefferman, Israel, and Luli \cite{FeffermanIsraelLuli-2016}, where a closely related question, the interpolation of functions with a positivity constraint, is studied.

 The next proposition -which is chiefly needed for Lemma \ref{lem: extension operators preserve order up to small error} below- quantifies the intuitive fact that if $u \in C^3_b$ vanishes at a point $x_0\in \tilde G_n$, and $u\geq 0$ everywhere in $\tilde G_n$, then the gradient and the negative eigenvalues of the Hessian of $\tilde \pi_n^\beta u$ at $x_0$ must be small when $n$ is large.
 
We will need a cutoff function in the next few proofs.  We fix one, and call it $\phi_0$ such that
\begin{align}\label{eqWhitney:Phi0Cutoff}
	\phi_0:\real\to\real,\ 0\leq \phi_0\leq1,\ \phi_0\ \text{smooth},\ \phi_0\equiv 1\ \text{in}\ [-1,1],\ \text{and}\ \phi_0\equiv0\ \text{in}\ \real\setminus[2,2].
\end{align}
  
\begin{prop}\label{prop:local minima in Gn almost gives zero gradient nonnegative Hessian}
  Let $w\in C^3_b(M)$ be nonnegative in $\tilde G_n$ and such that $w(x_0)=0$ at some $x_0 \in \tilde G_n$. Then, with some universal $C$ we have
  \begin{align*}
    |\nabla \tilde \pi_n^\beta w(x_0)|_{g_{x_0}}\leq C\|w\|_{C^3}h_n \textnormal{ if } \beta\geq 1.\\
    |(\nabla^2 \tilde \pi_n^\beta w(x_0))_-|_{g_{x_0}} \leq C\|w\|_{C^3}h_n \textnormal{ if } \beta\geq 2.	
  \end{align*}	  
  Here, we recall that $\tilde h_n$ is as defined in \eqref{eqFinDimApp: tilde hn def}, and note that for a given matrix $D$, $D_-$ denotes its negative part.
\end{prop}

\begin{proof}[Proof of Proposition \ref{prop:local minima in Gn almost gives zero gradient nonnegative Hessian}]

  According to 2) in Proposition \ref{prop:Appendix Cbeta implies bound on Taylor remainder} if $d(x,x_0) \leq 4\delta\sqrt{d}$ (recall $\delta$ was defined in Remark \ref{rem: good balls cover M}), then
  \begin{align*}
    |w(x)-w(x_0)-l(\nabla \tilde \pi_n^\beta w(x_0),x_0;x)|\leq \|u\|_{C^2}d(x,x_0)^2. 
  \end{align*}	
  Then, using that $w(x_0)=0$ and $w(x)\geq 0$ for all $x\in \tilde G_n$, we conclude that
  \begin{align}\label{eqFinDimApp: positivity first order}
    0\leq l(\nabla \tilde \pi_n^\beta w(x_0),x_0;x)+\|\tilde \pi_n^\beta w\|_{C^2}d(x,x_0)^2,\;\;\forall\;x\in \tilde G_n \cap B_{4\delta \sqrt{d}}(x_0).
  \end{align}	
  Then, using that $d(x_0,\tilde G_n)\leq \tilde h_n$, it is not hard to see there is some  $x_1 \in \tilde G_n$ with $d(x_1,x_0)\leq 4\delta\sqrt{d}$ such that (for some universal $C$),
  \begin{align*}
    l(\nabla \tilde \pi_n^\beta w(x_0),x_0;x_1)\leq -C^{-1}|\nabla \tilde \pi_n^\beta w(x_0)|_{g_{x_0}}|\exp_{x_0}^{-1}(x_1)|_{g_{x_0}}.
  \end{align*}
  Then, using \eqref{eqFinDimApp: positivity first order} with this $x_1$, and, Theorem \ref{thm:EnTnContOperatorInCbetaNorm}, we see that
  \begin{align*}
    |\nabla \tilde \pi_n^\beta w(x_0)|_{g_{x_0}} \leq C\|\tilde \pi_n^\beta w\|_{C^2}d(x,x_0)\leq C\|w\|_{C^2}h_n,
  \end{align*}
  proving the first estimate. 
  
  Next, we prove the second estimate. Let $\beta\geq 2$. Assume first that $\nabla \tilde \pi_n^\beta w(x_0)=0$. Then, we may use Proposition \ref{prop:Appendix Cbeta implies bound on Taylor remainder} as before to obtain,
  \begin{align}\label{eqFinDimApp: positivity second order}
    0\leq q(\nabla^2 \tilde \pi_n^\beta w(x_0),x_0;x)+\|\tilde \pi_n^\beta w\|_{C^3}d(x,x_0)^3,\;\;\;\forall\;x\in \tilde G_n \cap B_{4\delta \sqrt{d}}(x_0).
  \end{align}
  Then, as in the previous case, one can see there is some $x_1$ with $d(x_1,x_0)\leq  4\delta \sqrt{d}$ such that
  \begin{align*}
    q(\nabla^2 \tilde \pi_n^\beta w(x_0),x_0;x_1) & \leq -C^{-1}|(\nabla^2 \tilde \pi_n^\beta w(x_0))_-|_{g_x}|\exp_{x_0}^{-1}(x_1)|_{g_x}\\
	& = -C^{-1}|(\nabla^2 \tilde \pi_n^\beta w(x_0))_-|_{g_x}d(x_1,x_0)^2.
  \end{align*}
  Using \eqref{eqFinDimApp: positivity second order} with this $x_1$, we conclude that
  \begin{align*}
    |(\nabla^2 \tilde \pi_n^\beta w(x_0))_-|_{g_x}\leq C\|\tilde \pi_n^\beta w\|_{C^3}d(x_1,x_0) \leq C\|w\|_{C^3}\tilde h_n.
  \end{align*}
  If $\nabla \tilde \pi_n^\beta w(x_0) \neq 0$, we apply the above argument to the function
  \begin{align*}
    \tilde w = \tilde \pi_n^\beta(w,x)-l( \nabla \tilde \pi_n^\beta(w)(x_0),x_0;x), \textnormal{ defined in } B_{4\delta \sqrt{d}}(x_0).
  \end{align*}
  As explained at the end of Remark \ref{rem:FinDimApp l and q derivatives in a chart}, we always have $\nabla^2l( \nabla \tilde \pi_n^\beta(w)(x_0),x_0;x_0)=0$, thus the Hessian at $x_0$ is not perturbed by this change. Moreover, it is clear that $\|\tilde w\|_{C^\beta} \leq C\|w\|_{C^\beta}$, with a universal $C$, and the proof follows.
\end{proof}

Using Proposition \ref{prop:local minima in Gn almost gives zero gradient nonnegative Hessian}, we now show the existence of a kind of ``corrector'' to the Whitney extension, in the sense that $\tilde \pi_n^\beta w$ plus this corrector is non-negative in $M$ whenever $w$ is non-negative in $\tilde G_n$, the corrector having a $C^\beta$ norm which vanishes as $n$ goes to infinity.
\begin{lem}\label{lem: extension operators preserve order up to small error}
  Fix $\beta \in [0,3)$. Let $w \in C_b^3(M)$ and suppose that there is some $x_0 \in \tilde G_n$ such that
  \begin{align*} 
    w(x) & \geq 0, \;\forall\; x\in \tilde G_n,\\
    w(x_0) & = 0, \;x_0\in \tilde G_n.	  	   
  \end{align*}	  
  Then, there is a function $R_{\beta,n,w,x_0}$ with $R_{\beta,n,w,x_0}(x_0)=0$ and such that
  \begin{align*}
    & (\tilde \pi_n^\beta w)(x)+R_{\beta,n,w,x_0}(x)\geq 0\ \  \forall\; x\in M,\\
    & \|R_{\beta,n,w,x_0}\|_{C^{\beta}(M)} \leq Ch_n^{\gamma}\|w\|_{C^3(M)}.  
  \end{align*}	  
  Here, $\gamma := i-\beta$ for $\beta \in [i-1,i)$, $i=2,3$, while $\tilde h_n$ is as in \eqref{eqFinDimApp: tilde hn def}. 
\end{lem}

\begin{proof}
  As in previous proofs, for any $x \in M$, let $\hat x \in \tilde G_n$ be a point such that $d(x,\tilde G_n)=d(x,\hat x)$. Further, by Remark \ref{rem: interpolation preserves order if beta less than one} the lemma is trivial with $R_{\beta,n,w,x_0}\equiv 0$ in the case $\beta \in [0,1)$.
  
  \textbf{The case $\beta \in [1,2)$}. First, we must take care of the first order part of $w$ near $x_0$, by writing
  \begin{align*}
    \tilde w_n(x) = \tilde \pi_n^\beta w(x)-l(\nabla \tilde \pi_n^\beta w(x_0),x_0;x)\phi_{0}(d(x,x_0)^2)
  \end{align*}
  Where $\phi_{0}(t)$ is a smooth function which is identically equal to $1$ for $t\leq (\delta/4)^2$ and vanishes for $t>(\delta/2)^2$. Let us gather a few properties of $\tilde w_n$. First, thanks to Proposition \ref{prop: local interpolation operators} we have
  \begin{align*}
   \|\tilde w_n\|_{C^2}\leq C\|\tilde \pi_n^\beta w\|_{C^2}
  \end{align*}
  Moreover, $\tilde w_n$ has a vanishing gradient at $x_0$
  \begin{align*}
    \nabla \tilde w_n(x_0) = 0.
  \end{align*}
  Given $x\in M$, let $\hat x \in \tilde G_n$ denote some point such that $d(x,\tilde G_n)=d(x,\tilde x)$. Then, from the positivity assumption on $w$, we have $\tilde w_n(\hat x) \geq -l(\nabla \tilde \pi_n^\beta w(x_0),x_0;\hat x)\phi_0(d(x,x_0)^2)$ for any $x \in \tilde G_n$. Then, given $x \in B_{\delta/2}(x_0)$, we have
  \begin{align*}
    \tilde w_n(x) & \geq \tilde w_n(\hat x)-C\|\tilde w_n\|_{C^1}d(x,\hat x)\\
	& \geq -l(\nabla \tilde \pi_n^\beta w(x_0),x_0;\hat x)\phi_0(d(x,x_0)^2)-C \|w\|_{C^3}d(x,\hat x),\\
	& \geq -C|\nabla \pi^\beta_n w(x_0)|_{g_{x_0}}d(\hat x,x_0)\phi_0(d(x,x_0)^2)-C \|w\|_{C^3}h_n
  \end{align*}
  For such $x$, we have that $d(\hat x,x_0)\leq d(x,x_0)+d(\hat x,x) \leq \delta+h_n$, therefore
  \begin{align*}
    \tilde w_n(x) \geq -C|\nabla \pi^\beta_n w(x_0)|_{g_{x_0}}(\delta+\tilde h_n) -C \|w\|_{C^3}h_n.
  \end{align*}
  Using Proposition \ref{prop:local minima in Gn almost gives zero gradient nonnegative Hessian}, we conclude that,
  \begin{align}\label{eqFinDimApp:Almost order preserving inequality 1}
    \tilde w_n(x) \geq -C\|w\|_{C^3}h_n,\;\;\forall\;x\in M.
  \end{align}
  Next, we use that $\tilde w_n(x_0)=0$ and $\nabla \tilde w_n(x_0) = 0$, together with Proposition \ref{prop:Appendix Cbeta implies bound on Taylor remainder}, to obtain the bound  
  \begin{align}\label{eqFinDimApp:Almost order preserving inequality 2}
    \tilde w_n(x) & \geq -C\|w\|_{C^3}d(x,x_0)^2,\;\;\forall\;x\in M.
  \end{align}
  The idea is to combine these two estimates to construct the desired function, using \eqref{eqFinDimApp:Almost order preserving inequality 1} away from $x_0$, and \eqref{eqFinDimApp:Almost order preserving inequality 2} near $x_0$. We define a preliminary function $\tilde R_{\beta,n,w,x_0}$ as follows,
  \begin{align*}
    \tilde R_{\beta,n,w,x_0}(x) & := C\|w\|_{C^3}h_n \eta\left (\frac{d(x,x_0)^2}{h_n} \right ).
  \end{align*}
  Here, $\eta:\mathbb{R}_+\to\mathbb{R}$ is an auxiliary smooth, nondecreasing function such that
  \begin{align*}
    \eta(t) = t \textnormal{ in } [0,1/2],\;\;\;\eta(t) \equiv 1 \textnormal{ in } [1,\infty).
  \end{align*}
  Then, if $d(x,x_0)^2 \geq Ch_n$, from \eqref{eqFinDimApp:Almost order preserving inequality 1} we have
  \begin{align*}
    \tilde w_n +\tilde R_{\beta,n,w,x_0} \geq 0.
  \end{align*}
  On the other hand, if $d(x,x_0)^2 \leq Ch_n$ we use \eqref{eqFinDimApp:Almost order preserving inequality 2} to obtain
  \begin{align*}
    \tilde w_n(x)+\tilde R_{\beta,n,w,x_0}(x) & \geq -C\|w\|_{C^3}d(x,x_0)^2+R_{\beta,n,w,x_0}(x)\\
	  & \geq -C\|w\|_{C^3}d(x,x_0)^2+C\|w\|_{C^3}d(x,x_0)^2\\
	  & \geq 0.
  \end{align*}
  Moreover, 
  \begin{align*}
    \nabla_a \tilde R_{\beta,n,w,x_0}(x) = C\|w\|_{C^3} \eta'\left (\frac{d(x,x_0)^2}{h_n} \right ) 2d(x,x_0)\nabla_a d(x,x_0).
  \end{align*}
  Thus, if $d(x,x_0)^2 \geq h_n$, $\nabla_a R_{\beta,n,w,x_0}(x) =0$. If $d(x,x_0)^2 \leq h_n$ then
  \begin{align*}
    |\nabla_a \tilde R_{\beta,n,w,x_0}(x)-\nabla_a \tilde R_{\beta,nw,x_0}(x')| \leq C\|w\|_{C^3}d(x,x').
  \end{align*}
  This may be rewritten as,
  \begin{align*}
    \frac{|\nabla_a \tilde R_{\beta,n,w,x_0}(x)-\nabla_a \tilde R_{\beta,n,w,x_0}(x')|}{d(x,x')^{\beta-1}} \leq C\|w\|_{C^3}h_n^{2-\beta}.
  \end{align*}
  In conclusion, letting $R_{\beta,n,w,x_0} := \tilde R_{\beta,n,w,x_0}(x)-l(\nabla \tilde \pi_n^\beta w(x_0),x_0;x)\phi_0(d(x,x_0)^2)$ it follows that $\tilde \pi_n^\beta w + R_{\beta,n,w,x_0}\geq 0$ everywhere and
  \begin{align*}
    \|R_{\beta,n,w,x_0}\|_{C^{\beta}} \leq C h_n^{2-\beta} \|w\|_{C^3}.
  \end{align*}
  Thus $R_{\beta,n,w,x_0}$ as constructed has the desired properties.
 
  \textbf{The case $\beta \in [2,3)$}. This time, we must get rid of the first and second order parts of $\tilde \pi_n^\beta w$ near $x_0$. Therefore, we write
  \begin{align*}
    \tilde w_n(x) := \tilde \pi_n^\beta w(x)-\left ( l(\nabla \tilde \pi_n^\beta w(x_0),x_0;x)+q(\nabla^2 \tilde \pi_n^\beta w(x_0),x_0;x) \right ) \phi_{0}(d(x,x_0)^2).
  \end{align*}
  Where $\phi_{0}$ is the same function from the case $\beta \in [1,2)$. Then, as in the previous case we have two inequalities,
  \begin{align*}
    \tilde w_n(x) \geq -C\|w\|_{C^3}h_n,
  \end{align*}	  
  and
  \begin{align*}
    \tilde w_n(x) \geq -C\|w\|_{C^3}d(x,x_0)^3
  \end{align*}	  
  Then, we introduce the function
  \begin{align*}
    \tilde R_{\beta,n,w,x_0}(x) & := C\|w\|_{C^3}h_n \eta\left (\frac{d(x,x_0)^3}{h_n} \right ).
  \end{align*}
  where $\eta$ is the same function from the previous case. If $d(x,x_0)^3\geq h_n$ it follows that
  \begin{align*}
    \tilde w_n(x) + \tilde R_{\beta,n,w,x_0}(x)  = \tilde w_n(x) + C\|w\|_{C^3}h_n \geq 0.
  \end{align*}	  
  On the other hand, if $d(x,x_0)^3\leq h_n$, then
  \begin{align*}
    \tilde w_n(x) + \tilde R_{\beta,n,w,x_0}(x)  \geq \tilde w_n(x) + C\|w\|_{C^3}d(x,x_0)^3 \geq 0.
  \end{align*}	  
  Letting $R_{\beta,n,w,x_0} := \tilde R_{\beta,n,w,x_0}(x)-\left ( l(\nabla \tilde \pi_n^{\beta}w(x_0),x_0;x)+q(\nabla^2 \tilde \pi_n^{\beta}w(x_0),x_0;x)\right )\phi_0(d(x,x_0)^2)$ we conclude that $\tilde \pi_n^\beta w+R_{\beta,n,w,x_0}\geq 0$ in $M$. 
\end{proof}

\subsection{Remaining proofs for the case where $\beta\geq 1$}\label{subsec:proofs beta bigger than one} Here we present the proof of the more technical cases in Proposition \ref{prop:BoundaryEstimate}, Proposition \ref{prop:InteriorEstimate}, and Lemma \ref{lem:EnTnEstToUInC3}.

\begin{proof}[Proof of Proposition \ref{prop:BoundaryEstimate} for $\beta\geq 1$]

  \textbf{The case $\beta \in [1,2)$.} In this case, $f$ has the form
    \begin{align*}
      f(x) = \sum \limits_{k} \left ( u(\hat y_{n,k})+l( \nabla^1_n u (\hat y_{n,k}),\hat y_{n,k};x)\right )\phi_{n,k}(x).
    \end{align*}
    Let $x\in M\setminus \tilde G_n$ and $\hat x \in \tilde G_n$ be such that $d(x,\hat x)=d(x,\tilde G_n)$. 
    \begin{align*}
      \nabla_{a}f(x) & = \sum \limits_{k} \nabla_a \left ( l( \nabla^1_n u (\hat y_{n,k}),\hat y_{n,k};x)\right )\phi_{n,k}(x)\\
  	& \;\;\;\;+\sum \limits_{k} \left ( u(\hat y_{n,k})+l( \nabla^1_n u (\hat y_{n,k}),\hat y_{n,k};x)\right )\nabla_a \left ( \phi_{n,k}(x) \right ) \\
  & = \sum \limits_{k} \nabla_a \left ( l( \nabla^1_n u (\hat y_{n,k}),\hat y_{n,k};x)\right )\phi_{n,k}(x)+\sum \limits_{k} \left (u(\hat y_{n,k})+l( \nabla^1_n u (\hat y_{n,k}),\hat y_{n,k};x)\right )\nabla_a \left ( \phi_{n,k}(x) \right ). 	
    \end{align*}
    For $\hat x\in \tilde G_n$, we have
    \begin{align*}
      \nabla_{a}f(\hat x) = \nabla^1_nu(\hat x),
    \end{align*}
    which is not too difficult to show. Since the proof of this fact essentially follows the same argument used later on in the proof of Theorem \ref{thm:EnTnContOperatorInCbetaNorm} --in the case $\beta \in [1,2)$--, we omit the proof.
	
    Then, using \eqref{eqFinDim:Partition of Unity Differentiation Identity} in the above expression for $\nabla_a f(x)$, we see that
    \begin{align*}
      \nabla_{a}f(x)-\nabla_{a}f(\hat x)  & = \sum \limits_{k} \left ( \nabla_a \left ( l( \nabla^1_n u (\hat y_{n,k}),\hat y_{n,k};x)\right ) -\nabla_{a}f(\hat x)\right )\phi_{n,k}(x)\\
  	& \;\;\;\;+\sum \limits_{k} \left (u(\hat y_{n,k})+l( \nabla^1_n u (\hat y_{n,k}),\hat y_{n,k};x)-u(\hat x)\right )\nabla_a \left ( \phi_{n,k}(x) \right ).	
    \end{align*}
    Recall that the only non-zero terms above are those with $k\in K_x$ (defined in Lemma \ref{lem:CubesOnM}). For such $k$, thanks to Definition \ref{def:LinearAndQuadPolynomials} and Proposition \ref{prop: discrete derivatives are controlled by continuous derivatives} we have 
    \begin{align*}
      |\nabla_a \left ( l( \nabla^1_n u (\hat y_{n,k}),\hat y_{n,k};x)\right ) -\nabla_{a}f(\hat x)| & \leq |\nabla_a \left ( l( \nabla^1_n u (\hat y_{n,k}),\hat y_{n,k};x)\right )-\nabla_a \left ( l( \nabla^1_n u (\hat y_{n,k}),\hat y_{n,k};\hat x)\right )|\\
  	& \;\;\;\;+|\nabla_a \left ( l( \nabla^1_n u (\hat y_{n,k}),\hat y_{n,k};\hat x)\right )-\nabla_{a}f(\hat x)|\\
  	& \leq C\|u\|_{C^\beta}d(x,\hat x)^{\beta-1}.
    \end{align*}
    Adding these for every $k\in K_x$, and using that $\#K_x\leq N$,	
    \begin{align}
      \left | \sum \limits_{k} \left ( \nabla_a \left ( l( \nabla^1_n u (\hat y_{n,k}),\hat y_{n,k};x)\right ) -\nabla_{a}f(\hat x)\right )\phi_{n,k}(x)\right | & \leq C\|u\|_{C^\beta}d(x,\hat x)^{\beta-1}.\label{eqFinDimApp:boundary estimate beta larger than 1 first term}
    \end{align}	  	
    Let us bound the remaining terms (compare with \cite[Chp 6, Sec 2.3.2]{Stei-71}). Let $k \in K_x$, we seek a bound for the quantity
    \begin{align*}
      |u(\hat y_{n,k})+l(\nabla^1_nu(\hat y_{n,k}),\hat y_{n,k};x)-u(\hat x)|.		
    \end{align*}			
    Assume that $\hat y_{n,k}\neq \hat x$ (otherwise the quantity is zero and there is nothing to prove). By the triangle inequality, to bound this quantity it suffices to bound the sum
    \begin{align*}
      & |u(\hat y_{n,k})+l(\nabla u(\hat y_{n,k}),\hat y_{n,k};\hat x)-u(\hat x)|+|l(\nabla^1_nu(\hat y_{n,k}),\hat y_{n,k};x)-l(\nabla u(\hat y_{n,k}),\hat y_{n,k};\hat x)|.	  
    \end{align*}			
    Thanks to Proposition \ref{prop:Appendix Cbeta implies bound on Taylor remainder} in the Appendix, for each $k\in K_x$, we have the bound
    \begin{align}		
      |u(\hat y_{n,k})+l(\nabla u(\hat y_{n,k}),\hat y_{n,k};\hat x)-u(\hat x)|\leq \|u\|_{C^\beta}d(\hat x,\hat y_{n,k})^\beta.	\label{eqFinDimApp:Boundary Estimate gradient bound 1}
    \end{align}	
    At the same time, Lemma \ref{lem:InterpolatedDerivativesAreCloseToRealDerivatives}  yields $|\nabla^1_n u(\hat y_{n,k})-\nabla u(\hat y_{n,k})|_{g_x}\leq C\|u\|_{C^\beta} h_n^{\beta-1}$ (see Appendix). Then,  from the definition of the operators $l$ (see also Remark \ref{rem:FinDimApp l and q derivatives in a chart}), it follows that
    \begin{align*}	
      |l(\nabla^1_nu(\hat y_{n,k}),\hat y_{n,k};x)-l(\nabla u(\hat y_{n,k}),\hat y_{n,k};\hat x)| \leq C\|u\|_{C^\beta}h_n^{\beta-1}d(x,\hat y_{n,k}). 
    \end{align*}	
    Since $\hat y_{n,k}\neq \hat x$, we have $d(\hat x,\hat y_{n,k}) \geq \lambda h_n$, thanks to \eqref{eqFinDimApp: lambda def}. Thus $h_n^{\beta-1} \leq \lambda^{1-\beta}d(x,\hat y_{n,k})^{\beta-1}$ and we conclude there is some universal constant $C$ such that
    \begin{align}	
      |l(\nabla^1_nu(\hat y_{n,k}),\hat y_{n,k};x)-l(\nabla u(\hat y_{n,k}),\hat y_{n,k};\hat x)| \leq C\|u\|_{C^\beta}d(x,\hat y_{n,k})^{\beta}.\label{eqFinDimApp:Boundary Estimate gradient bound 2}
    \end{align}		
   Then, as argued earlier to obtain \eqref{eqFinDim:hat x and hat Ynk distance is bounded by distance to tilde Gn} (using Remark \ref{rem: Whitney local diameter of balls} once again)  we have 
   \begin{align*}
     & d(\hat x,\hat y_{n,k})\leq 16d(x,\tilde G_n),\;\;\forall\;k\in K_x,
   \end{align*}
   which trivially implies the bound $d(x,\hat y_{n,k})\leq 17d(x,\tilde G_n)$ for every $k\in K_x$. Combining this with \eqref{eqFinDimApp:Boundary Estimate gradient bound 1} and \eqref{eqFinDimApp:Boundary Estimate gradient bound 2}, we obtain the bound
    \begin{align*}
      |u(\hat y_{n,k})+l(\nabla^1_nu(\hat y_{n,k}),\hat y_{n,k};x)-u(\hat x)| \leq C\|u\|_{C^\beta}d(x,\hat x)^\beta \;\;\forall\;k\in K_x.
    \end{align*}			
    Given that $|\nabla \phi_{n,k}|_{g_x}\leq C\diam(P_{n,k}^*)^{-1}$ (Lemma \ref{lem:PartitionOfUnityGn}), the last inequality above, and the fact that $\#K_x \leq N$ (Lemma \ref{lem:CubesOnM} ), it follows that
    \begin{align}
      \left|\sum_k \left(l( \nabla^1_n u (\hat y_{n,k}),\hat y_{n,k};x)-l( \nabla^1_n u (\hat x),\hat x;x)\right) \grad_a \phi_{n,k}(x) \right| \leq C\|u\|_{C^\beta}d(x,\hat x)^{\beta-1}.\label{eqFinDimApp:boundary estimate beta larger than 1 second term}
    \end{align}	
    Combining these we conclude that
    \begin{align*}
      |\nabla_{a}f(x)-\nabla_{a}f(\hat x)  |\leq C\|u\|_{C^\beta(M)} d(x,\hat x)^{\beta-1},\;\;\forall\;x\in M\setminus \tilde G_n,
    \end{align*}
    as we wanted.
    	
  \textbf{The case $\beta \in [2,3)$}. Finally, in this case we have
  \begin{align*}
    \nabla^{2}_{ab}f(x) & = \sum\limits_{k} \nabla^{2}_{ab}\left ( l( \nabla^1_n u (\hat y_{n,k}),\hat y_{n,k};x)\right )\phi_{n,k}(x)\\
	&\;\;\;\;+\sum\limits_{k} \nabla_{a}\left ( l( \nabla^1_n u (\hat y_{n,k}),\hat y_{n,k};x)\right )\nabla_{b}\phi_{n,k}(x)+\nabla_{b}\left ( l( \nabla^1_n u (\hat y_{n,k}),\hat y_{n,k};x)\right )\nabla_{a}\phi_{n,k}(x)\\
	& \;\;\;\;+\sum\limits_{k} \left ( l( \nabla^1_n u (\hat y_{n,k}),\hat y_{n,k};x)\right )\nabla^{2}_{ab}\phi_{n,k}(x).
  \end{align*}
  The argument from this point on is entirely analogous to the one for $\beta \in [1,2)$. We only sketch the details. One uses \eqref{eqFinDim:Partition of Unity Differentiation Identity} and the above identity to write an expression $\nabla^2_{ab}f(x)-\nabla_{ab}^2f(\hat x)$. This expression is itself separated into various sums grouped according to whether a term has a factor of $\phi_{n,k}$, $\nabla_a\phi_{n,k}$, or $\nabla_{ab}\phi_{n,k}$. Then, one proceeds to use Proposition \ref{prop:Appendix Cbeta implies bound on Taylor remainder} and Lemma \ref{lem:InterpolatedDerivativesAreCloseToRealDerivatives} to obtain bounds for the various terms, in a manner analogous to the case $\beta \in [1,2)$. In conclusion, one arrives at the desired bound,
  \begin{align*}
    |\nabla^{2}_{ab}f(x)-\nabla^{2}_{ab}f(\hat x)| \leq C\|u\|_{C^\beta}d(x,\hat x)^{\beta-2}.
  \end{align*}
\end{proof}

\begin{proof}[Proof of Proposition \ref{prop:InteriorEstimate} for $\beta \geq 1$]
  We recall some of the setup, already used in the case $\beta<1$. We let $x_1,x_2 \in B_r(x)$, where $B_{4r}(x)\subset M \setminus \tilde G_n$, so that $d(x_i,\tilde G_n)\geq r$ for $i=1,2$. Let $x(t)$ denote again the geodesic going from $x_1$ to $x_2$, parametrized with arc length, so that $x(0)=x_1,x(L)=x_2$ where $L=d(x_1,x_2)$. Under these circumstances, we have
  \begin{align*}
    d(x(t),\tilde G_n) \geq r,\;\;\forall\;t\in[0,L].    	    
  \end{align*}	  
  We now consider each of the remaining cases.

  \textbf{The case $\beta \in [1,2)$.}  
  Invoking the chain rule, and Proposition \ref{prop:RegularityOfEnTnDistToGn} as done for $\beta<1$, we have
  \begin{align*}
    \left |\frac{d}{dt}\nabla_a f(x(t)) \right | \leq C\|u\|_{C^\beta}r^{\beta-2}. 
  \end{align*}
  In particular, integrating from $t=0$ to $t=L$ we have
  \begin{align*}
    |\nabla_af(x_1) -\nabla_af(x_2)|  \leq \int_0^L \left | \frac{d}{dt}\nabla_a f(x(t)) \right |\;dt & \leq Cr^{\beta-2}\|u\|_{C^\beta(M)} d(x_1,x_2).
  \end{align*}  
  Since $\beta-2<0$ and $d(x_1,x_2)\leq 2r$, it follows that $r^{\beta-2}\leq 2^{2-\beta}d(x_1,x_2)^{\beta-2}$. Then,
  \begin{align*}
    |\nabla_af(x_1) -\nabla_af(x_2)|  \leq C\|u\|_{C^\beta(M)} d(x_1,x_2)^{\beta-1}.
  \end{align*}

  \textbf{The case $\beta \in [2,3)$.} This time we use the third derivative estimate from Proposition \ref{prop:RegularityOfEnTnDistToGn}, which yields
  \begin{align*}
    \left |\frac{d}{dt}\nabla_{ab}^2 f(x(t)) \right | \leq C\|u\|_{C^\beta}r^{\beta-3}. 
  \end{align*}
  Then, 
  \begin{align*}
    |\nabla^{2}_{ab}f(x_1)-\nabla^{2}_{ab}f(x_2)| \leq \int_0^L \left | \frac{d}{dt}\nabla^{2}_{ab}f(x(t)) \right |\;dt \leq Cr^{\beta-3}\|u\|_{C^\beta(M)}d(x_1,x_2)
  \end{align*}	
 This time, since $\beta-3<0$ and $d(x_1,x_2)\leq 2r$, we have $r^{\beta-3}\leq 2^{3-\beta}d(x_1,x_2)^{\beta-3}$ and therefore
  \begin{align*}
    |\nabla^{2}_{ab}f(x_1)-\nabla^{2}_{ab}f(x_2)| \leq C\|u\|_{C^\beta(M)}d(x_1,x_2)^{\beta-2}.
  \end{align*}

\end{proof}

\begin{proof}[Proof of Lemma \ref{lem:EnTnEstToUInC3} for $\beta \geq 1$]
  \textbf{The case $\beta\in [1,2)$}. In this case we need to go further and bound the H\"older seminorm of $\nabla_a f$, for every index $a$. Observe that
  \begin{align*}
    |\nabla_{a} f(x)-\nabla_{a} u(x)| \leq |\nabla_{a} f(x)-\nabla_{a}f(\hat x)|+ |\nabla_af(\hat x)-\nabla_a u(\hat x)| + |\nabla_{a}u(\hat x) -\nabla_{a} u(x)|.   
  \end{align*}
  Evidently,
  \begin{align*}
    & |\nabla_{a} f(x)-\nabla_{a}f(\hat x)| \leq  C\|u\|_{C^2}d(x,\hat x),\\
    &  |\nabla_{a}u(\hat x) -\nabla_{a} u(x)| \leq \|u\|_{C^2}d(x,\hat x).
  \end{align*}
  Where we have used that $\|f\|_{C^2} \leq C\|u\|_{C^2}$ in the first inequality. According to Lemma \ref{lem:InterpolatedDerivativesAreCloseToRealDerivatives}, $|\nabla f(\hat x)-\nabla u(\hat x)|_{g_x}$, is bounded from above by $C\|u\|_{C^2}h_n$  (recall that $\grad \tilde \pi_n^\beta w$ and $\grad^1_n u$ agree at points in $\tilde G_n$). Since $d(x,\hat x_n)\leq h_n$, we conclude that
  \begin{align*}
    \sup \limits_{x\in M} |\nabla f(x)-\nabla u(x)|_{g_x} \leq C\|u\|_{C^2}h_n \leq C\|u\|_{C^3}h_n.
  \end{align*}	  
  The H\"older seminorm of $\nabla f(x)-\nabla u(x)$ is estimated using an argument analogous to the one used in the case $\beta \in [0,1)$. Let $x_1,x_2 \in M$, and let $\nabla_a$ be as usual. Suppose first that $d(x_1,x_2) \leq h_n$. Then, using that $\|f-u\|_{C^2} \leq \|f\|_{C^2}+\|u\|_{C^2} \leq C\|u\|_{C^2}$ (by Theorem \ref{thm:EnTnContOperatorInCbetaNorm}),
  \begin{align*}
    |\nabla_af(x_1)-\nabla_a u(x_1)-(\nabla_af(x_2)-\nabla_a u(x_2))|\leq C\|u\|_{C^2}d(x_1,x_2).
  \end{align*}
  Using that $2-\beta<0$ and $d(x_1,x_2) \leq h_n$, it follows that
  \begin{align*}
    \frac{|\nabla_af(x_1)-\nabla_a u(x_1)-(\nabla_af(x_2)-\nabla_a u(x_2))|}{d(x_1,x_2)^{\beta-1}} & \leq C\|u\|_{C^2}d(x_1,x_2)^{2-\beta}\\
      & \leq C\|u\|_{C^2}h_n^{2-\beta}.	
  \end{align*}
  Next, let us consider what happens if $x_1,x_2$ are such that $d(x_1,x_2) > h_n$. First, we note that
  \begin{align*}
    & |\nabla_af(x_1)-\nabla_a u(x_1)-(\nabla_af(x_2)-\nabla_a u(x_2))|\\
    & \leq |\nabla_af(x_1)-\nabla_a u(x_1)| +|\nabla_af(x_2)-\nabla_a u(x_2)|.	  
  \end{align*}	  
  To estimate these two terms, we decompose each of them again. We have, for $i=1,2$
  \begin{align*}
    |\nabla_af(x_i)-\nabla_a u(x_i)| \leq |\nabla_af(x_i)-\nabla_af(\hat x_i)|+ |\nabla_a f(\hat x_i)-\nabla_a u(\hat x_i)|+|\nabla_a u(\hat x_i)-\nabla_a u(x_i)|.
  \end{align*}
  Now, on one hand we have the estimates
  \begin{align*}
    |\nabla_af(x_i)-\nabla_af(\hat x_i)| & \leq C\|u\|_{C^2}d(x_i,\hat x_i),\\
    |\nabla_au(x_i)-\nabla_au(\hat x_i)| & \leq \|u\|_{C^2}d(x_i,\hat x_i),
  \end{align*}
  while on the other hand Lemma \ref{lem:InterpolatedDerivativesAreCloseToRealDerivatives} says that $|\nabla_a f(\hat x_i)-\nabla_a u(\hat x_i)| \leq C\|u\|_{C^2}d(x_i,\hat x_i)$. Gathering these bounds and using that $d(x_1,x_2)>h_n \geq d(x_i,\hat x_i)$, we conclude that
  \begin{align*}
    |\nabla_af(x_1)-\nabla_a u(x_1)-(\nabla_af(x_2)-\nabla_a u(x_2))| & \leq C\|u\|_{C^2}h_n.
  \end{align*}  
  Then, since $\beta \in [1,2)$, 
  \begin{align*}
    \frac{|\nabla_af(x_1)-\nabla_a u(x_1)-(\nabla_af(x_2)-\nabla_a u(x_2))|}{d(x_1,x_2)^{\beta-1}} & \leq C\|u\|_{C^2}h_n^{2-\beta}.
  \end{align*}  
  In conclusion, for $x_1,x_2\in M$ with $x_1\neq x_2$ we have
  \begin{align*}
    & \frac{|\nabla_{a} f(x_1)-\nabla_{a}u(x_1)-(\nabla_{a}f(x_2)-\nabla_{a}f(x_2))|}{d(x_1,x_2)^{\beta-1}} \leq C\|u\|_{C^2}h_n^{2-\beta}.
  \end{align*}
  Therefore, as in the case $\beta \in [0,1)$, we conclude that
  \begin{align*}
    \|f-u\|_{C^\beta} = \|f-u\|_{L^\infty}+\|\nabla f-\nabla u\|_{L^\infty}+[\nabla f-\nabla u]_{C^{\beta-1}} \leq C\|u\|_{C^2}h_n^{2-\beta}, 	  
  \end{align*}	  
  proving the estimate in this case.
  
  \noindent \textbf{The case $\beta \in [2,3)$}. In this case we must also take into account the values of $\nabla^2f$. Similarly as in the previous cases, we use a triangle inequality to estimate the difference $\nabla^2 f(x)-\nabla^2 u$. Let $a,b$ be indices in one of the usual exponential system of coordinates, then
  \begin{align*}
    |\nabla_{ab}^2 f(x)-\nabla_{ab}^2 u(x)|\leq |\nabla_{ab}^2 f(x)-\nabla_{ab}^2 f(\hat x)|+|\nabla_{ab}^2 f(\hat x)-\nabla_{ab}^2 u(\hat x)|+|\nabla_{ab}^2 u(x)-\nabla_{ab}^2 u(\hat x)|.
  \end{align*}
  On the other hand, we have, from Theorem \ref{thm:EnTnContOperatorInCbetaNorm}
  \begin{align*}
    & |\nabla_{ab}^2 f(x)-\nabla_{ab}^2 f(\hat x)| \leq  C\|u\|_{C^3}d(x,\hat x),\\
    &  |\nabla_{ab}^2 u(\hat x) -\nabla_{ab}^2 u(x)| \leq \|u\|_{C^3}d(x,\hat x),
  \end{align*}
  and, again from Lemma \ref{lem:InterpolatedDerivativesAreCloseToRealDerivatives} in the Appendix, 
  \begin{align*}
    |\nabla_{ab}^2 f(\hat x)-\nabla_{ab}^2 u(\hat x)|\leq Ch_n\|u\|_{C^3}.
  \end{align*}  
  Combining these inequalities, we conclude that
  \begin{align*}
    & \sup |\nabla_{ab}^2 f(x)-\nabla_{ab}^2 u(x)|\leq C\|u\|_{C^3}h_n.
  \end{align*}
  Now we consider the H\"older seminorm. Fix $x_1,x_2\in M$. As before, consider first the case where $d(x_1,x_2)\leq h_n$, in which case it is clear that
  \begin{align*}
    |\nabla_{ab} f(x_1)-\nabla_{ab} u(x_1)-(\nabla_{ab} f(x_2)-\nabla_{ab} u(x_2))| \leq \|u\|_{C^3}d(x_1,x_2).
  \end{align*} 	  
  Then, using that $3-\beta<0$, and that $d(x_1,x_2) \leq h_n$, it follows that 
  \begin{align*}
    \frac{|\nabla_{ab} f(x_1)-\nabla_{ab} u(x_1)-(\nabla_{ab} f(x_2)-\nabla_{ab} u(x_2))|}{d(x_1,x_2)^{\beta-2}}  \leq C\|u\|_{C^3}d(x_1,x_2)^{3-\beta}\leq C\|u\|_{C^3}h_n^{3-\beta}.	 		
  \end{align*} 	  
  Let us now take the opposite case, that is when $d(x_1,x_2)>h_n$. Then
  \begin{align*}
    & |\nabla_{ab} f(x_1)-\nabla_{ab} u(x_1)-(\nabla_{ab} f(x_2)-\nabla_{ab} u(x_2))|\\
    & \leq |\nabla_{ab} f(x_1)-\nabla_{ab} u(x_1)|+|\nabla_{ab} f(x_2)-\nabla_{ab} u(x_2)| 
  \end{align*} 	  
  As before, 
  \begin{align*}
    &  |\nabla_{ab} f(x_i)-\nabla_{ab} u(x_i)| \leq |\nabla_{ab} f(x_i)-\nabla_{ab} f(\hat x_i)|+|\nabla_{ab} f(\hat x_i)-\nabla_{ab} u(\hat x_i)|+|\nabla_{ab} u(x_i)-\nabla_{ab} u(\hat x_i)| 
  \end{align*} 	  
  Next, we have 
  \begin{align*}
    |\nabla_{ab} f(x_i)-\nabla_{ab} f(\hat x_i)| & \leq C\|u\|_{C^3}d(x_i,\hat x_i),\\
    |\nabla_{ab} u(x_i)-\nabla_{ab} u(\hat x_i)| & \leq \|u\|_{C^3}d(x_i,\hat x_i).
  \end{align*}
  These inequalities, together with the bound $|\nabla_{ab} f(\hat x_i)-\nabla_{ab} u(\hat x_i)|\leq C\|u\|_{C^3}h_n$ from Lemma \ref{lem:InterpolatedDerivativesAreCloseToRealDerivatives} in the Appendix, yield
  \begin{align*}
    |\nabla_{ab} f(x_1)-\nabla_{ab} u(x_1)-(\nabla_{ab} f(x_2)-\nabla_{ab} u(x_2))| \leq C\|u\|_{C^3}h_n.	  
  \end{align*}
  Using that $d(x_1,x_2)>h_n$, we see that in this case
  \begin{align*}
    \frac{|\nabla_{ab} f(x_1)-\nabla_{ab} u(x_1)-(\nabla_{ab} f(x_2)-\nabla_{ab} u(x_2))|}{d(x_1,x_2)^{\beta-2}} \leq C\|u\|_{C^3}h_n^{3-\beta}.	  
  \end{align*}
  The rest of the proof is entirely analogous to the previous case, and the Lemma is proved.
\end{proof}


\section{The Min-max formula in infinite dimensions: Functions on $(M,g)$}\label{sec: Min-Max formula for M,g}

This section has two goals: defining a ``finite dimensional'' approximation to $I$; and showing that the approximation can be used, along with Section \ref{sec: Min-Max finite dimension}, to prove Theorem \ref{thm: Min-Max Riemannian Manifold}.  First we develop the approximation, and second we establish Theorem \ref{thm: Min-Max Riemannian Manifold}.

\subsection{Approximations to $I$ and their structure}\label{sec:ApproximationsIn}

We are now ready to introduce the finite dimensional approximations to the Lipschitz map $I: C^{\beta}_b(M) \to C_b(M)$. Recall that in Definition \ref{def:WhitneyExt} we introduced the restriction and extension operators $\tilde T_n$ and $\tilde E_n^\beta$, below we introduce slight modifications of these operators, which have the advantage that they depend only on the values of $u$ over $G_n$, and not all of $\tilde G_n$.

\begin{DEF}\label{def:WhitneyExtFinite}
  For each $n$, we define
  \begin{enumerate}	
    \item  The restriction operator $T_n:C^\beta_b(M) \to C(G_n)$, defined by
    \begin{align*}	
      T_n(u,x) := u(x)\ \forall\;x\in G_n. 	
    \end{align*}	
    \item The extension operator of order $\beta$, $E_n^\beta:C(G_n) \to C^\beta_b(G_n)$, defined by
    \begin{align*}
      E_n^\beta(u,x) := E_n^\beta(\tilde u,x)		
    \end{align*}				
    where $E_n^\beta$ is the extension operator from Definition \ref{def:WhitneyExt}, and $\tilde u \in C(G_n)$ denotes the function which agrees with $u$ in $G_n$ and is defined to be zero in $\tilde G_n \setminus G_n$.
    \item Again, we have a projection map, which we denote $ \pi_n^\beta$, and is defined by $\pi_n^\beta := E_n^\beta \circ T_n$.
  \end{enumerate}	
\end{DEF}

From Remark \ref{rem:extension preserves compact support}, it becomes clear that, if $u \in C^\beta_b(M)$ has compact support, then if $n$ is large enough, then $E_n^\beta \circ T_n = \tilde E_n^\beta \circ \tilde T_n u$. Therefore
\begin{align}\label{eqMMonM:ProjectionsAgree CompactlySupported Functions}
  u\in C^\beta_c(M) \Rightarrow \tilde \pi_n^\beta(u,x) = \pi_n^\beta(u,x)\;\;\forall\;x\in M,\;\;\textnormal{ for all large enough } n.
\end{align}
Using this, we can use the apply the results about $\tilde \pi_n^\beta$ from Section \ref{sec: finite dimensional approximations} to $\pi_n^\beta$ when dealing with functions supported in some compact set $K$ and $n$ large enough (depending on $K$).

We will create two approximations, which we call $i_n$ and $I_n$.  The distinction is that $i_n$ is legitimately defined on the finite dimensional space, $C(G_n)$, whereas $I_n$ will be defined on $C^\beta_b(M)$, but is finite dimensional in the sense that it returns the same value for any two functions that agree on $G_n$.  Introducing $I_n$ will be important so that both $I_n$ and $I$ have the same domain and co-domain. 

To this end, we let $T_n, E_n^\beta$, and $\pi_n^\beta$ from Definition \ref{def:WhitneyExtFinite}, and now we define
\begin{align*}
  I_n:C_b^{\beta}(M)\to C_b(M),\ \    I_n := \pi_n^{0} \circ I \circ \pi_n^{\beta}, 
\end{align*}
that is to say 
\begin{align}\label{eqMMonM:InDef}
  I_n(u,x) := E_n^{0}T_n I(E^{\beta}_nT_nu,x) .
\end{align}
The approximations $I_n$ will be seen to well approximate $I$ on a set that is dense with respect to local uniform convergence in $C^{\beta}_b(M)$ (as opposed to norm convergence).

\begin{DEF}\label{def:PiNandXbeta} 
  Define, for $\beta\in[0,2]$, the finite dimensional subspace $X^\beta_n\subset C^\beta_b(M)$ by
  \begin{align*}
    X_n^\beta := E_n^{\beta}(C(G_n)).	
  \end{align*}
\end{DEF}

\begin{prop}[Convergence of $I_n$ on $C^3_c(M)$]\label{prop:InConvergesUniformlyToI}
  With $I_n$ defined in \eqref{eqMMonM:InDef} and for every compact $K\subset M$ and any $R>0$, we have
  \begin{align*}
    \lim\limits_{n\to\infty}\sup \limits_{\|u\|_{C^{3}_c(K)}\leq R} \|I_nu- Iu\|_{L^\infty(M)} = 0.
  \end{align*}	  	
\end{prop}

\begin{proof}[Proof of Proposition  \ref{prop:InConvergesUniformlyToI}]
  First of all, let us recall that that $\pi_n^\beta u = E_n^\beta T_n u$.  Then, we have
  \begin{align*}  
    \|I(u)-I(\pi_n^\beta u)\|_{L^\infty(M)} \leq \norm{I}_{\textnormal{Lip}(C^{\beta}_b,C)}\|u-\pi_n^\beta u\|_{C^\beta(M)}.
  \end{align*}
  On the other hand, successive applications of Theorem \ref{thm:EnTnContOperatorInCbetaNorm} and the linearity of $\pi_n^0$ imply that
  \begin{align*}
    \|\pi_n^0I(u)-\pi_n^0 I(\pi_n^\beta u)\|_{L^\infty(M)} & \leq C\|I(u)-I(\pi_n^\beta u)\|_{L^\infty(M)}\\
	& \leq C\norm{I}_{\textnormal{Lip}(C^{\beta}_b,C)}\|u-\pi_n^\beta u\|_{C^\beta(M)}.
  \end{align*}	
  It follows that
  \begin{align*}
    &\norm{I_nu - Iu}_{L^\infty(M)} \ \ \leq  C\norm{I}_{\textnormal{Lip}(C^{\beta}_b,C)}\|u-\pi_n^\beta u\|_{C^\beta(M)}.
  \end{align*}
  At this point, we can apply Lemma \ref{lem:EnTnEstToUInC3} to the right hand side of the last inequality (using \eqref{eqMMonM:ProjectionsAgree CompactlySupported Functions}), and we see that for sufficiently large $n$,
	\begin{align*}
		&\norm{I_nu - Iu}_{L^\infty(M)} \leq C\norm{I}_{\textnormal{Lip}(C^{\beta}_b,C)}\tilde h_n^\gam\norm{u}_{C^3_b(M)}
	\end{align*}
	where $\tilde h_n$ is as defined in \eqref{eqFinDimApp: tilde hn def} and $\gamma$ is as in Lemma \ref{lem:EnTnEstToUInC3}.  It follows that for all large $n$, and all $u\in C^3_b(M)$ which are compactly supported in $K$, we have
	\begin{align*}
	  \|I_nu-I_u\|_{L^\infty(M)} \leq C\tilde h_n^\gamma \norm{I}_{\textnormal{Lip}(C^{\beta}_b,C)}R,
        \end{align*}		
        and the Proposition is proved.		
\end{proof}

For each $n$, the map $I_n$ may be thought of as a finite dimensional approximation to $I$ in the following sense. We define the map
\begin{align}\label{eqMMonM:LittleInDef}
  i_n : C(G_n) \to C(G_n),\;\;i_n := T_n \circ I \circ E^\beta_n.
\end{align}	
Thus, $I_n$ and $i_n$ are related by
\begin{align*}
  I_n = E_n^{0} \circ i_n \circ T_n.
\end{align*}
In particular, this shows that although $I_n:C^\beta_b(M)\to C_b(M)$, $I_n$ is uniquely determined by its values on functions in $X_n^\beta$, and functions in $X_n^\beta$ are uniquely determined by their values on $G_n$.

\begin{rem}
	As suggested by the results in Section \ref{sec:WhitneyAlmostOrderPreserving} in particular Lemma \ref{lem: extension operators preserve order up to small error}, except for when $\beta\in (0,1)$, it is not expected that $i_n$ or $I_n$ will enjoy the GCP.  However, this is not a set-back because the GCP is recovered in the limit as $n\to\infty$.  The potential failure of the GCP originates with the composition by $E_n^\beta$, and the latter operator may not be order preserving when $\beta\geq 1$.
\end{rem}

\begin{lem}\label{lem:inAndInNormBounds}
	There is a universal constant, $C$, so that $\norm{i_n}_{Lip(C(G_n),C(G_n))}\leq C\norm{I}_{Lip(C^\beta_b,C_b)}$ and $\norm{I_n}_{Lip(C^\beta_b,C_b)}\leq C\norm{I}_{Lip(C^\beta_b,C_b)}$.
\end{lem}

\begin{proof}[Comments on Lemma \ref{lem:inAndInNormBounds}]
	This is a straightforward consequence of the bound in Theorem \ref{thm:EnTnContOperatorInCbetaNorm}, that $\norm{E_n^\beta T_n u}_{C^\beta}\leq C\norm{u}_{C^\beta}$, and the definitions of both $i_n$ and $I_n$.
\end{proof}

The advantage of this presentation is that we may now use the results from Section \ref{sec: Min-Max finite dimension} to obtain a min-max formula for $I_n$, via the theory applied to $i_n$.  First, we make an observation that relates the differentiability properties of $i_n$ and $I_n$.

\begin{lem}\label{lem:DiffOfLittleInAndBigIn}
	Assume that $u\in X^\beta_n(M)$ and $u_n=T_nu$.  The map, $i_n$, is Fr\`echet differentiable at $u_n$ if and only if $I_n$ is Fr\`echet differentiable at $u=E_n^\beta u_n$.  Furthermore,
	\begin{align*}
		DI_n|_u = E_n^0\circ Di_n|_{u_n}\circ T_n.
	\end{align*}
\end{lem}

\begin{proof}[Comments on Lemma \ref{lem:DiffOfLittleInAndBigIn}]
	This is a straightforward consequence of the uniqueness of $DI_n$ and $Di_n$ as well as the chain rule.  We omit the details.
\end{proof}

We define the analog of the Clarke differential for $i_n$ in the context of $I_n$.

\begin{DEF}\label{def:DInDefOfDifferential}
	The differential, $\D I_n$, is defined as
    \begin{align*}
      \mathcal{D}I_n= \hull \big\{ L \in \mathcal{L}(C^{\beta}_b,C_b)\  &:\  \exists \{u^n_k\}_k \textnormal{ s.t. } DI_n|_{u^n_k} \textnormal{ exists }\;\forall\;k \textnormal{ and }\\ 
  	&\ \ \lim\limits_{k\to \infty} DI_n|_{u^n_k}(f,x) = L(f,x)\ \forall f\in C^\beta_b, \forall\; x\in G_n \big\}. 	  
    \end{align*}
\end{DEF}

An immediate corollary of this definition and Lemma \ref{lem:DiffOfLittleInAndBigIn} is
\begin{cor}\label{cor:RelateDifferentialSetOfLittleAndBigIn}
	Composition by $E_n^0$ and $T_n$ over $\D i_n$ gives $\D I_n$:
	\begin{align*}
		\D I_n = \left\{ E_n^0 l T_n\ :\ l \in\D i_n     \right\}.
	\end{align*}
\end{cor}

\begin{lem}\label{lem:InMinMax}
  For each $n$, the map $I_n: C^{\beta}_b(M) \to C_b(M)$ admits a min-max formula when evaluated over the set $G_n$; i.e. 
  \begin{align}\label{eqMMonM:InMinMax}
    \forall u\in C^\beta_b(M),\ \forall\; x\in G_n,\ \ I_n(u,x) = \min\limits_{v \in C^\beta_b(M)}\max\limits_{L \in \mathcal{D}I_n} \{ I_n(v,x)+L(u-v,x) \}.	  
  \end{align}	  
  Here $\D I_n$ is as in Definition \ref{def:DInDefOfDifferential}.
\end{lem}

\begin{proof}[Proof of Lemma \ref{lem:InMinMax}]
	To begin the proof, we make a few simple but useful observations about the range of $i_n$.  First,
	\begin{align*}
		v_n\in C(G_n)\ \iff v_n=T_n v\ \text{for some}\ v\in C^{\beta}_b(M),
	\end{align*}
	and second
	\begin{align*}
		\left\{ T_nIE_n^\beta v_n\ :\ v_n\in C(G_n)    \right\}
		 = \left\{ T_nIE_n^\beta T_nv\ :\ v\in C^{\beta}_b(M)  \right\}.
	\end{align*}
	Applying Lemma \ref{lem:FiniteDimensionalMinMax} to $i_n$, we see that for all $u_n\in C(G_n)$ and all $x\in G_n$,
    \begin{align*}
      i_n(u_n,x) &= \min\limits_{v_n \in C(G_n)}\max\limits_{L_n \in \mathcal{D}i_n} \{ i_n(v_n,x) + L_n(u_n-v_n,x) \}\\
	  &= \min\limits_{v_n \in C(G_n)}\max\limits_{L_n \in \mathcal{D}i_n} \{ T_nIE_n^\beta(v_n,x) + L_n(u_n-v_n,x) \}\\
	  &= \min\limits_{v \in C^\beta_b(M)}\max\limits_{L_n \in \mathcal{D}i_n} \{ T_nIE_n^\beta T_n(v,x) + L_n(u_n-T_n v,x) \}.
    \end{align*}
	Thus, replacing $u_n$ by $T_n u$, we see that for all $u\in C^\beta_b$ and $x\in G_n$,
	\begin{align*}
		i_n(T_n u,x) = \min\limits_{v \in C^\beta_b(M)}\max\limits_{L \in \mathcal{D}i_n} \{ T_nIE_n^\beta T_n(v,x) + L(T_n(u-v),x) \}.
	\end{align*}
	This shows that for all $u,v\in C^\beta_b(M)$ and for all $x\in G_n$, the inequality:
	\begin{align*}
		i_n(T_nu,x)\leq i_n(T_nv,x) + \max_{L_n\in \D i_n}\{ L_n(T_n(u-v),x) \};
	\end{align*}
	and unraveling the notation for $i_n$, we see that
	\begin{align*}
		T_nIE_n^\beta T_n(u,x)\leq T_nIE_n^\beta T_n(v,x) + \max_{ L_n\in \D i_n}\{ L_nT_n((u-v),x) \}.
	\end{align*}
	Thanks to the fact that $E_n^0$ is monotone and linear, as well as Corollary \ref{cor:RelateDifferentialSetOfLittleAndBigIn}, we have
	\begin{align*}
		E_n^0T_nIE_n^\beta T_n(u,x)&\leq E_n^0T_nIE_n^\beta T_n(v,x) + E_n^0\max_{ L_n\in \D i_n}\{ L_nT_n((u-v),x) \}\\
		&\leq E_n^0T_nIE_n^\beta T_n(v,x) + \max_{ L_n\in \D i_n}\{ E_n^0L_nT_n((u-v),x) \}\\
		&=E_n^0T_nIE_n^\beta T_n(v,x) + \max_{ \tilde L_n\in \D I_n}\{ \tilde L_n((u-v),x) \},
	\end{align*}
	Where we note in the middle inequality that if $L_n^{(x)}$ is a collection that point-by-point attains the max, then by Corollary \ref{cor:RelateDifferentialSetOfLittleAndBigIn}, $E_n^0L_n^{(x)}$ is an admissible family in $\D I_n$.
	Thus, by definition of $I_n$, we see that for all $u,v\in C^\beta_b(M)$ and all $x\in G_n$,
	\begin{align*}
		I_n(u,x)\leq I_n(v,x) + \max_{ \tilde L_n\in \D I_n}\{ \tilde L_n((u-v),x) \}.
	\end{align*}
	Taking a min over $v\in C^\beta_b(M)$, we have achieved \eqref{eqMMonM:InMinMax} for all $x\in G_n$.
\end{proof}

\begin{lem}\label{lem:DInOperatorNormBound}
	If $L_n\in \D I_n$, then 
	$\displaystyle \norm{L_n}_{C^\beta_b\to C_b}\leq C\norm{I}_{C^\beta_b\to C_b}$.
\end{lem}

\begin{proof}[Proof of Lemma \ref{lem:DInOperatorNormBound}]
	First, assume that $u\in C^\beta_b(M)$ and that $I_n$ is Fr\'echet differentiable at $u$.  Let $\phi\in C^\beta_b(M)$ and $\psi\in C^\beta_b(M)$, and let $t>0$.
	\begin{align*}
		\norm{\frac{I(u+t\phi)-I(u)}{t}-\frac{I(u+t\psi)-I(u)}{t}}_{L^\infty(M)} &= \norm{ \frac{I(u+t\phi)-I(u+t\psi)}{t} }_{L^\infty(M)}\\
		&\leq \frac{1}{t}\norm{I_n}_{Lip(C^\beta_b,C_b)}\norm{t(\phi-\psi)}_{C^\beta(M)}\\
		&\leq C\norm{I}_{Lip(C^\beta_b,C_b)}\norm{\phi-\psi}_{C^\beta(M)}.
	\end{align*}
	Letting $t\to 0$ establishes the bound for $DI_n(u)$.
	We also note that in a Banach space, norm bounds are closed under convex combinations and weak limits, hence they also hold for $\D I_n$.
\end{proof}

It will be useful to know that the assumption \eqref{eqIntro:AssumptionQuatifiedLocalUnifConv1} is also obeyed by the operators $I_n$ uniformly in $n$.  This is indeed the case up to a slight enlargement factor, which is due to the result of the finite range of dependence of the operators $E_n^\beta T_n$, proved in Lemma \ref{lem:EnTnRangeOfDependence}.

\begin{lem}\label{lem:InSatisfiesExtraAssumptions}
  There is a universal constant, $C$, such that for $\om$ as in assumption \eqref{eqIntro:AssumptionQuatifiedLocalUnifConv1}, $I_n$ inherits a slightly modified version of \eqref{eqIntro:AssumptionQuatifiedLocalUnifConv1} in the form of
    \begin{align}\label{eqMMonM:AssumptionQuatifiedLocalUnifConv1MODIFIED}
  	  \forall u,v\in C^\beta_b,\ \ \norm{I_n(u)-I_n(v)}_{L^\infty(B_r)}
  	  \leq C\norm{u-v}_{C^\beta(\overline{B_{2r+3}})} + C\om(r)\norm{u-v}_{L^\infty(M)},
    \end{align}
\end{lem}

\begin{proof}[Proof of Lemma \ref{lem:InSatisfiesExtraAssumptions}]
This is immediate from two applications of Lemma \ref{lem:EnTnRangeOfDependence}, combined with the assumptions \eqref{eqIntro:AssumptionQuatifiedLocalUnifConv1}.
\end{proof}

\begin{lem}\label{lem:DInSatisfiesExtraAssumptions}
	Up to a uniform constant, any $L_n\in \D I_n$ also inherits the properties of Lemma \ref{lem:InSatisfiesExtraAssumptions}. 
\end{lem}

\begin{proof}[Comments on Lemma \ref{lem:DInSatisfiesExtraAssumptions}]
	This follows in a similar way as the proof of Lemma \ref{lem:InSatisfiesExtraAssumptions}, combined with the observations of the proof of Lemma \ref{lem:DInOperatorNormBound}.
\end{proof}

\begin{lem}\label{lem:In has an approximate GCP}
  Let $L \in \mathcal{D}I_n$. Suppose that $w \in C^3_b(M)$ is nonnegative and $w(x_0)=0$, $x_0\in G_n$. Then
  \begin{align*}
    L(w,x_0) \geq -C h_n^\gam \|w\|_{C^3(M)}	  
  \end{align*}	  
  where $\lim h_n=0$ and $h_n$ is defined in \eqref{eqFinDimApp: hn def}.
\end{lem}

\begin{proof}[Proof of Lemma \ref{lem:In has an approximate GCP}]
  Since the lower bound for $L$ is preserved under convex combinations and limits, then, given the definition of $\mathcal{D}I_n$ it is clear that it suffices to prove the inequality when $L$ is the classical derivative of $I_n$ at points of differentiability for $I_n$. To this end, let us fix $u \in C^\beta_b$, an arbitrary point of differentiability of $I_n$, and let $L_u$ denote the respective derivative. 
  
  We apply Lemma \ref{lem: extension operators preserve order up to small error} to $w$, to obtain the remainder polynomial, $R_{\beta,n,m,x_0}$ and conclude that for any $t>0$ we have
  \begin{align*}
    u+tw+tR_{\beta,n,w,x_0} \geq u\;\;\; \forall\;x \in M,
  \end{align*}	  
  with equality at $x=x_0$ (recall that $R_{\beta,n,w,x_0}(x_0)=0$). Since $I$ has the global comparison property, it follows that
  \begin{align*}
     I(u+tw+tR_{\beta, n,w,x_0},x_0) \geq I(u,x_0)\;\;\forall\;t>0.
  \end{align*}
  Furthermore, since $I$ is a Lipschitz map,
  \begin{align*}
    I(u+tw,x_0) & \geq I(u+tw+tR_{\beta,n,w,x_0},x_0)-tC\|R_{\beta,n,w,x_0}\|_{C^\beta}\\
	    & \geq I(u,x_0)-tC\|R_{\beta,n,w,x_0}\|_{C^\beta}.	  
  \end{align*}	
  It follows that     
  \begin{align*}
    L_u(w,x_0) = \frac{d}{dt}_{t=0^+}I(u+tw,x_0) \geq -C \|R_{\beta,n,w,x_0}\|_{C^\beta}  
  \end{align*}	  
  Since $w \in C^3_c(M)$, Lemma \ref{lem: extension operators preserve order up to small error} also says that
  \begin{align*}
    \|R_{\beta,n,w,x_0}\|_{C^\beta} \leq C h_n^\gam \|w\|_{C^3}
  \end{align*}
  Thus,
  \begin{align*}
    L_u(w,x_0) \geq -C h_n^\gam \|w\|_{C^3}  
  \end{align*}	  
  This holds for every $u$ where $I_n$ is differentiable. Therefore, by the Definition \ref{def:DInDefOfDifferential} of $\D I_n$ in  it also holds for any $L \in \mathcal{D}I_n$.
\end{proof}

\subsection{Some nice properties of $I$, $I_n$, and $\pi_n$}\label{sec:NicePropertiesIInPin}

Here we will collect some useful observations about $I$, $I_n$, and $\pi_n^\beta$.  They seem to be useful in their own right, and we hope they will appear elsewhere, but they are also essential for extracting limits of operators in $\D I_n$, and so we mention them here.

For the remainder of this section, we will use many times a function $\rho$, which is simply a smooth function that behaves like $t\mapsto \min\{t,1\}$.  We define it below.

\begin{DEF}\label{def:RhoTheSmoothApproxOfMin(1,t)}
	Let $\rho$ be fixed from here until the end of this section as a function that satisfies 
	\begin{align*}
		\rho(s) = s\ \forall s\in [0,1),\ \ \rho(s)\equiv 3/2\ \forall s\in[2,\infty),\ \text{and}\ \abs{\rho'} + \abs{\rho''}\leq 4.
	\end{align*}
\end{DEF}

\begin{lem}\label{lem:NormEstimateLocalizedToNearAndFarTermsWithTestFunction}
	Let $x\in M$, and let $\phi\in C^\beta_b(M)$ be any function such that $0\leq\phi\leq 1$ and $\phi(x)=0$.  Then for any $u,v\in C^\beta_b(M)$,
	\begin{align}\label{eqMMonM:NormEstimateOnlyLinfty}
		\abs{I(\phi u,x)-I(\phi v,x)}\leq 
		(\norm{I}_{Lip(C^\beta_b,C_b)}\cdot \norm{\phi}_{C^\beta(M)})\cdot \norm{u-v}_{L^\infty(\textnormal{spt}(\phi))},
	\end{align}
	as well as
	\begin{align}\label{eqMMonM:NormEstimateCbetaAndLinfty}
		\abs{I(u,x)-I(v,x)}
		\leq \norm{I}_{Lip(C^\beta_b,C_b)}
		\left(
		\norm{(1-\phi)(u-v)}_{C^\beta(M)}
		+\norm{\phi}_{C^\beta(M)}\cdot\norm{u-v}_{L^\infty(\textnormal{spt}(\phi))}
		\right).
	\end{align}
\end{lem}

\begin{proof}[Proof of Lemma \ref{lem:NormEstimateLocalizedToNearAndFarTermsWithTestFunction}]
	First we establish \eqref{eqMMonM:NormEstimateOnlyLinfty}.  Note that for all $y\in \text{spt}(\phi)$,
	\begin{align*}
		u(y)-v(y)\leq \norm{u-v}_{L^\infty(\text{spt}(\phi))},
	\end{align*}
	 and so for all $y\in M$, we also have 
	\begin{align*}
		\phi(y)u(y)-\phi(y)v(y)\leq \phi(y)\norm{u-v}_{L^\infty(\text{spt}(\phi))}.
	\end{align*}
	This says that the function $\phi v + \phi\norm{u-v}_{L^\infty(\text{spt}(\phi))}$ touches $\phi u$ from above at any $x$ such that $\phi(x)=0$.  By the GCP, we have
	\begin{align*}
		I(\phi u,x)\leq I(\phi v + \phi\norm{u-v}_{L^\infty(\text{spt}(\phi))},x),
	\end{align*}
	so that
	\begin{align*}
		I(\phi u,x) - I(\phi v,x)&\leq I(\phi v + \phi\norm{u-v}_{L^\infty(\text{spt}(\phi))},x) - I(\phi v,x)\\
		&\leq \norm{I}_{Lip(C^\beta_b,C_b)} \cdot \norm{\left(\phi\norm{u-v}_{L^\infty(\text{spt}(\phi))}\right)}_{C^\beta(M)}\\
		&= \norm{I}_{Lip(C^\beta_b,C_b)} \cdot \norm{\phi}_{C^\beta(M)}\cdot \norm{u-v}_{L^\infty(\text{spt}(\phi))}.
	\end{align*}
	The proof of \eqref{eqMMonM:NormEstimateCbetaAndLinfty} is similar, working with the inequality
	\begin{align*}
		(1-\phi)u+\phi u - \phi v
		\leq (1-\phi)u + \phi\norm{u-v}_{L^\infty(\text{spt}(\phi))},
	\end{align*}
	which becomes an equality at any $x$ such that $\phi(x)=0$.  Thus, the GCP gives
	\begin{align*}
		I((1-\phi)u+\phi u,x)
		\leq I((1-\phi)u + \phi v + \phi\norm{u-v}_{L^\infty(\text{spt}(\phi))},x),
	\end{align*}
	and after subtracting from both sides, we have
	\begin{align*}
		&I((1-\phi)u+\phi u,x) - I((1-\phi)v+\phi v,x)\\
		&\leq I((1-\phi)u + \phi v + \phi\norm{u-v}_{L^\infty(\text{spt}(\phi))},x)
		- I((1-\phi)v+\phi v,x)\\
		&\leq \norm{I}_{Lip(C^\beta_b,C_b)} \cdot \norm{\left((1-\phi)(u-v)+\phi\norm{u-v}_{L^\infty(\text{spt}(\phi))}\right)}_{C^\beta(M)}\\
		&\leq \norm{I}_{Lip(C^\beta_b,C_b)} \cdot\left( \norm{(1-\phi)(u-v)}_{C^\beta(M)} + \norm{\phi}_{C^\beta(M)}\cdot \norm{u-v}_{L^\infty(\text{spt}(\phi))} \right).
	\end{align*}
\end{proof}

In particular, using the results and proof of Lemma \ref{lem:NormEstimateLocalizedToNearAndFarTermsWithTestFunction}, after choosing an appropriate $\phi$ to approximate $B_R(x)$, we have as a corollary,

\begin{cor}\label{cor:NormEstimateLocalizedToNearAndFarTerms}
	Given and $R>0$, there exists a constant, $C(R)$, depending only on dimension such that for any $x$ fixed, $r>0$,
	\begin{align*}
		\norm{I(u)-I(v)}_{L^\infty(B_R)}
		\leq C(R)\norm{I}_{Lip(C^\beta_b,C_b)}
		\left(
		\norm{(u-v)}_{C^\beta(B_{R+1})}
		+\norm{u-v}_{L^\infty((B_R)^C)}
		\right),
	\end{align*}
	as well as $C(R,r)$ which blows up as $R$, $r$ are both small,
	\begin{align*}
		\abs{I(u,x)-I(v,x)}
		\leq \norm{I}_{Lip(C^\beta_b,C_b)}
		\left(
		C(R)\norm{(u-v)}_{C^\beta(B_{R+r}(x))}
		+C(R,r)\norm{u-v}_{L^\infty((B_R(x))^C)}
		\right),
	\end{align*}
	where $\om(r)\to0$ as $r\to\infty$ and comes from the limit in the extra assumption. 
	
\end{cor}

\begin{proof}[Sketch of the proof of Corollary ]
	
	We just comment that this follows by making an appropriate choice of test functions in Lemma \ref{lem:NormEstimateLocalizedToNearAndFarTermsWithTestFunction}.
\end{proof}

A very useful estimate, somewhat related to Corollary \ref{cor:NormEstimateLocalizedToNearAndFarTerms}, involves the Whitney extension and touching a function from above. The proof of this uses Lemma \ref{lem: extension operators preserve order up to small error} to a great degree. We record it as a proposition for later use.

\begin{prop}\label{prop:WhitneyObeysBetaTouchingFromAbove}
	Let $x_0 \in G_n$ be fixed, and let $f\in C_b(M)$ be such that $f(x_0)=0$. Let $\beta\in [0,3)$ and $\ep\in[0,1)$.  Consider the function $w(x):=f(x)\rho(d(x,x_0)^{\beta+\ep})$.  There is a dimensional constant, $C$, and a function $R_{n,x_0}$ such that $R_{n,x_0}(x_0)=0$, $\|R_{n,x_0}\|_{C^\beta} \to 0$ as $n\to\infty$, and
	\begin{align*}
		\pi_n^\beta(w,x)\leq C \norm{f}_{L^\infty}(\rho(d(x,x_0)^{\beta+\ep})+R_{n,x_0}(x)).
	\end{align*}
	Here, $\rho$ is the function introduced in Definition \ref{def:RhoTheSmoothApproxOfMin(1,t)}.
\end{prop}

\begin{proof}
 For the sake of brevity, we only provide the details for the case where $\beta>2$, the other cases are simpler and the details are left to the reader. It will be convenient to introduce the following two functions 
  \begin{align*}
    w_0(x) & := \rho(d(x,x_0)^{\beta+\varepsilon}),\\	  
    \tilde w(x) & := \|f\|_{L^\infty}w_0(x)-f(x)w_0(x) = \|f\|_{L^\infty}w_0(x)-w(x).
  \end{align*}
  Clearly, $w_0(x), \tilde w(x)\geq 0$ for all $x\in M$ and $w_0(x_0) = \tilde w(x_0)=0$.  Using the definition of $\pi_n^\beta$, and the positivity of $\|f\|_{L^\infty}-f(x)$, it is not difficult to show that
  \begin{align*}
    \pi_n^\beta( \tilde w,x) & \geq -C\|f\|_{L^\infty}d(x,x_0)^{\beta+\varepsilon},\\
    \pi_n^\beta( \tilde w,x) & \geq -C\|f\|_{L^\infty}h_n.	
  \end{align*}
  Then, imitating the argument used in Lemma \ref{lem: extension operators preserve order up to small error}, we can construct a function function $\tilde R_{n,\beta+\varepsilon,x_0}$ such that $\tilde R_{n,\beta+\varepsilon,x_0}(x_0)=0$, $\|\tilde R_{n,\beta+\varepsilon,x_0}\|_{C^\beta} \to 0$ as $n\to\infty$, and
    \begin{align*}
      & \pi_n^\beta(\tilde w,x)+\|f\|_{L^\infty}\tilde R_{n,\beta+\varepsilon,x_0}(x) \geq 0 \Rightarrow \|f\|_{L^\infty(M)}(\pi_n^\beta(w_0,x)+\tilde R_{n,\beta+\varepsilon,x_0}(x)) \geq \pi_n^\beta(w,x).
    \end{align*}
  On the other hand, from Proposition \ref{prop:Appendix Cbeta implies bound on Taylor remainder}, we have that if $d(x,x_0)\leq 4\delta \sqrt{d}$, then
  \begin{align*}
    \pi_n^\beta(w_0,x) \leq l( \nabla \pi_n^\beta(w_0),x_0;x)+q( \nabla^2 \pi_n^\beta(w_0),x_0;x)+C\|w_0\|d(x,x_0)^{\beta+\varepsilon}.
  \end{align*}
  Using that $\nabla w_0(x_0)=0$ for $\beta>1$, $\nabla^2 w_0(x_0)=0$ for $\beta>2$, together with Lemma \ref{lem:InterpolatedDerivativesAreCloseToRealDerivatives}, it is easy to see that (cf. Proposition \ref{prop: local interpolation operators})
  \begin{align*}
    \| l( \nabla \pi_n^\beta(w_0),x_0;\cdot)\|_{C^\beta(B_{\delta}(x_0))} & \leq C\tilde h_n \|w_0\|_{C^\beta},\textnormal{ for } \beta>1,\\
    \| q( \nabla^2 \pi_n^\beta(w_0),x_0;\cdot)\|_{C^\beta(B_{\delta}(x_0))} & \leq C\tilde h_n^{\beta-2}\|w_0\|_{C^\beta},\textnormal{ for } \beta>2.
  \end{align*} 
  Combining this with the function $\tilde R_{n,\beta+\varepsilon,x_0}$, it is not hard to see there is a function $R_{n,x_0}$ vanishing at $x_0$, such that 
  \begin{align*}
    & \pi_n^\beta(w,x) \leq C\|f\|_{L^\infty(M)}\left ( d(x,x_0)^{\beta+\varepsilon}+R_{n,x_0} (x)\right ),\\
    & \lim \limits_{n\to 0}\|R_{n,x_0}\|_{C^\beta} = 0,
  \end{align*}	  
  and the proposition is proved.
\end{proof}

The estimate in Proposition \ref{prop:WhitneyObeysBetaTouchingFromAbove} lead , via the GCP, to a useful estimate for $I$ and $I_n$.

\begin{lem}\label{lem:IInRespectComparisonWithBetaGrowthFromAbove}
  Let $x\in G_n$, $\beta \in (0,2]$, $f\in C_b(M)$, and $u\in C^\beta_b(M)$ be fixed.  Define the function, $w$, to be $w(y)=f(y)\rho(d(x,y)^\beta)$. Then for a universal $C$ it holds that
  \begin{align*}
    I(\pi_n^\beta u+\pi_n^\beta w,x)- I(\pi_n^\beta u, x)
      \leq C\norm{f}_{L^\infty(M)}\norm{\rho(d(x,\cdot)^\beta)+R_{n,x}(\cdot)}_{C^\beta(M)},
  \end{align*}
  and
  \begin{align*}
    E_n^0T_nI(\pi_n^\beta u+\pi_n^\beta w,x)- E_n^0T_nI(\pi_n^\beta u, x)
      \leq C\norm{f}_{L^\infty(M)}\norm{\rho(d(x,\cdot)^\beta)+R_{n,x}(\cdot)}_{C^\beta(M)},
  \end{align*}
  where $R_{n,x}$ is as in Proposition \ref{prop:WhitneyObeysBetaTouchingFromAbove}.
\end{lem}

\begin{proof}[Proof of Lemma \ref{lem:IInRespectComparisonWithBetaGrowthFromAbove}]
	We note that by Proposition \ref{prop:WhitneyObeysBetaTouchingFromAbove}  there is the touching of the two functions at $x$:
	\begin{align*}
		\pi_n^\beta u(y) + \pi_n^\beta w(y) 
		\leq \pi_n^\beta u(y) + C\norm{f}_{L^\infty}(\rho(d(x,y)^\beta)+R_{n,x}(y)).
	\end{align*}
	Thus, by the GCP,
	\begin{align*}
		I(\pi_n^\beta u + \pi_n^\beta w, x) 
		\leq I(\pi_n^\beta u + C\norm{f}_{L^\infty}(\rho(d(x,\cdot)^\beta)+R_{n,x}(\cdot)), x).
	\end{align*}
	Subtracting $I(\pi_n^\beta,x)$ from both sides, and using the Lipschitz assumption on $I$, we see that
	\begin{align*}
		I(\pi_n^\beta u + \pi_n^\beta w, x) - I(\pi_n^\beta u, x) &\leq I(\pi_n^\beta u + C\norm{f}_{L^\infty}(\rho(d(x,\cdot)^\beta)+R_{n,x}(\cdot)), x) - 
		I(\pi_n^\beta u, x)\\
		&\leq C\norm{I}_{Lip(C^\beta_b,C_b)}\cdot \norm{f}_{L^\infty}
		\cdot \norm{\rho(d(x,\cdot)^\beta)+R_{n,x}(\cdot)}_{C^\beta(M)}.
	\end{align*}
	We remark that also, the operator
	\begin{align*}
		E_n^0 T_n I,
	\end{align*}
	is an operator with the GCP if one only considers contact points belonging to $G_n$.  Therefore, substituting $E_n^0 T_n I$ instead of $I$ in the previous calculation preserves the result.
\end{proof}

\subsection{The Structure of $\D I_n$, compactness, and weak limits}\label{sec:StructureOfDIn}

In this section, we investigate in more detail the structure of the operators, $L_n\in \D I_n$.  In particular, for $x\in G_n$, each $L_n(\cdot,x)$ is expressed as the sum of an (approximately) local part and a nonlocal part (see Lemma \ref{lem:DiscreteMeasuresDInLocalAndNonlocalParts}).  The local and nonlocal parts are given in terms of a discrete measure associated to $L_n$ and $x$, and using this we obtain compactness properties and other limiting properties for the $L_n$.

\begin{lem}\label{lem:DiscreteMeasuresDIn}
	For all $L_n\in \D I_n$ and for all $x\in G_n$, there exist discrete signed Borel measures, $\mu_x^n$, and functions $C^n(x)$ such that for all $u\in C^\beta_b(M)$,
	\begin{align*}
		\forall x\in G_n,\ \ L_n(u,x) = C^n(x)u(x)
		+\int_{M\setminus\{x\}}u(y)-u(x)\; \mu_x^{n}(dy).
	\end{align*} 
	Moreover,
	\begin{align}\label{eqMMonM:Cn}
		C^n(x) = L_n(1,x),
	\end{align}
	and
	\begin{align*}
		\mu_x^n(dy) = \sum_{y\not=x} K^n(x,y)\del_y(dy),\ \ \text{with}\ 
		K(x,y) = L_n(E_n^\beta e_y,x),
	\end{align*}
	where $e_y \in C(G_n)$ are the ``basis'' functions introduced in Section \ref{sec: Min-Max finite dimension}.
\end{lem}

\begin{proof}[Proof of Lemma \ref{lem:DiscreteMeasuresDIn}]
	The proof is immediate from Section \ref{sec: Min-Max finite dimension} for any $l_n\in \D i_n$, which are linear mappings from $C(G_n)\to C(G_n)$.  However, thanks to Corollary \ref{cor:RelateDifferentialSetOfLittleAndBigIn}, we know that any such $L_n$ is of the form $E_n^0 \circ \ l_n \circ T_n$, for some $l_n\in \D i_n$.  Since $E_n^0$ is, by definition, an extension operator, and the Lemma only uses information of $u$ and $l_n$ restricted to $G_n$, we see that indeed the formula from Lemma \ref{lem: Finite Dimensional Courrege} and $l_n$ is preserved for $L_n$ as well.
\end{proof}

Recall that if $x \in G_{n'}$ for some $n'$, then $x\in G_n$ all $n\geq n'$. This means that, for $x \in \cup G_n$, and for any sequence of operators $L_n$ with $L_n \in \D I_n$, we have a respective sequence of Borel measures $\{ \mu_x^{n}\}_{n\geq n'}$ . Therefore, we are interested in obtaining bounds on these measures that allow us to obtain some kind of limit (at least along subsequences) as $n\to \infty$. These bounds are obtained in the following two lemmas.

\begin{lem}\label{lem:LnObeysBetaTouchingFromAboveBound}
	Let $\beta \in(0,2]$. If $L_n\in \D I_n$, then $L_n$ obeys the estimate of Corollary \ref{cor:NormEstimateLocalizedToNearAndFarTerms}. Moreover, given $x\in G_{n'}$ fixed, , $\ep\in[0,1)$, $f(x)=0$, and $w(y)=f(y)\rho(d(x,y)^{\beta+\ep})$, we have, for $n\geq n'$,
	\begin{align*}
		L_n(w,x)\leq C\norm{f}_{L^\infty(M)}\norm{\rho(d(x,\cdot)^{\beta+\ep})+R_{n,x}(\cdot)}_{C^\beta(M)}.
	\end{align*} 
\end{lem}

\begin{proof}[Proof of Lemma \ref{lem:LnObeysBetaTouchingFromAboveBound}]
	This follows by an argument entirely analogous to the one in Lemma \ref{lem:DInOperatorNormBound}. In this case, one invokes Lemma \ref{lem:IInRespectComparisonWithBetaGrowthFromAbove} to establish the estimates on any $L_n\in\D I_n$ that is an actual Fr\`echet derivative, $L_n=DI_n|_u$ at some $u\in C^\beta_b$, then pass the resulting estimate by density and convexity to all other elements of $\D I_n$.
\end{proof}

\begin{lem}\label{lem:MuNTotalVariation}
	Let $L_n\in\D I_n$, and  $\{\mu_x^n\}_{x\in G_n}$ the respective signed measures associated to $L_n$ by Lemma \ref{lem:DiscreteMeasuresDIn}.  If $m_x^n$ is the signed measure defined as
	\begin{align*}
		m_n^x(dy)=\rho(d(x,y)^{\beta})\mu_x^n(dy),
	\end{align*} 
	then, the total variation of $m_x^n$, denoted $\displaystyle\abs{m_x^n}$, is bounded independently of $n$ and $x$.
	
	When dealing with $C^\beta = C^1_b$, we replace $m^n_x$ by 
	\begin{align*}
		m_n^x(dy)=\rho(d(x,y)^{1+\ep})\mu_x^n(dy)\ \text{for}\ \ep\in(0,1).
	\end{align*}
\end{lem}

\begin{proof}
We note that for $x$ fixed and for any $f\in C^\beta_b(M)$, the function $f(y)\rho(d(x,y)^\beta)\in C^\beta_b(M)$.  Furthermore, since $f(x)\rho(d(x,x)^\beta)=0$, we obtain via Lemma \ref{lem:DiscreteMeasuresDIn} that
\begin{align*}
	L_n(f\rho(d(x,\cdot)^\beta),x) = \int_{M} f(y)\rho(d(x,y)^\beta)\mu_x^n(dy).
\end{align*}
Thus, the estimate of Lemma \ref{lem:LnObeysBetaTouchingFromAboveBound} immediately shows that
\begin{align*}
	\int_{M} f(y) m_x^n(dy)\leq C\norm{f}_{L^\infty}\norm{\rho(d(x,\cdot)^\beta)+R_{n,x}(\cdot)}_{C^\beta(M)},
\end{align*}	
and we obtain the bound taking the supremum over $f$ with $\|f\|_{L^\infty}\leq 1$, by duality.

\end{proof}

\begin{DEF}\label{def:EtaAndTildeEtaDef}
	We use smooth approximations to the indicator and bump functions.  Let $x$ be fixed, with $\eta^\epsilon_x$ and $\tilde\eta^\epsilon_x$ be smooth functions satisfying
	\begin{align*}
		& 0\leq \eta^\epsilon_x(y)\leq 1,\ \ \eta^\epsilon_x(y)\geq \Indicator_{B_{r_0}}(y),\ \ \eta^\epsilon_x(y)\searrow \Indicator_{B_{r_0}(x)}(y), \textnormal{ as }\epsilon \to 0
		\\
		& 0\leq \tilde\eta^\epsilon_x(y)\leq 1,\ \ \tilde\eta^\epsilon_x(y) \geq \Indicator_{B_{\epsilon}(x)}(y),\ \ \tilde\eta^\epsilon_x(y)\searrow \Indicator_{\{x\}}(y), \textnormal{ as }\epsilon \to 0.
	\end{align*}
\end{DEF}

\begin{DEF}\label{def:TaylorPolynomial}
	For $\beta\in[0,3)$ and $\epsilon\in(0,1)$, the $\epsilon$-Taylor ``polynomial'' of $u$ centered at $x$ is the function $T^{\epsilon,\beta}_x(u,y) \in C^\beta_b(M)$ given by
	\begin{align*}
		T^{\epsilon,\beta}_x(u,y) = 
		\begin{cases}
			u(x)\ &\text{if}\ \beta\in(0,1)\\
			u(x) + \eta^\epsilon_x(y)l(x,\grad u(x);y)\ &\text{if}\ \beta\in[1,2)\\
			u(x) + \eta^\epsilon_x(y) l(x,\grad u(x);y) + \tilde\eta^\epsilon_x(y) q(x,\nabla^2u(x);y)\ &\text{if}\ \beta \in [2,3).
		\end{cases}
	\end{align*}
\end{DEF}

\begin{DEF}\label{def:Adeln and Bdeln}
  Fix $L_n \in \D I_n$ and $x\in G_n$, and let $\mu^n_x$ be the measure from Lemma \ref{lem:DiscreteMeasuresDIn} and let $I$ denote the identity matrix $(TM)_x \to (TM)_x$.
  
  \noindent Then, we define $A^{\epsilon,n}(x):(TM)_x \to (TM)_x$ by
  \begin{align}\label{eqMMonM:Adeln}
    A^{\epsilon,n}(x) = \int_{M} \tilde \eta_x^{\epsilon}(y)q(I,x;y)\mu_x^n(dy).		    
  \end{align}	  
  Furthermore, using duality, we define $B^{\epsilon,n}(x) \in (TM)_x$ as the unique vector in $(TM)_x$ such that 
  \begin{align}\label{eqMMonM:Bdeln}	
    (B^{\epsilon,n}(x),p)_{g_x} = \int_{M} \eta_x^{\epsilon}(y)l(p,x;y)\mu_x^n(dy),\;\;\;\forall\;p\in (TM)_x.
  \end{align}	
	
\end{DEF}
 
\begin{lem}\label{lem:DiscreteMeasuresDInLocalAndNonlocalParts}

  Let $L_n\in\D I_n, x\in M$ and $u\in C^3_b(M)$. Then, for some ``remainder term'', denoted $\textnormal{(Error)}_{L_n,u,x}$, we have the following representation for $L_n(u,x)$: If $\beta=2$, then 
  \begin{align*}
    L_n(u,x) & = \Tr(A^{\epsilon,n}(x)\nabla ^2u(x))+(B^{\epsilon,n}(x),\nabla  u(x))_{g_x}+C^{n}(x)u(x)\\
    &\ \ \ \ \ \ \  +\int_{M}u(y)-T^{\epsilon,\beta}_x(u,y)\;\mu_x^{n}(dy)+\textnormal{(Error)}_{L_n,u,x};
  \end{align*}
  if $\beta\in [1,2)$, then
  \begin{align*}
    L_n(u,x) = (B^{\epsilon,n}(x),\nabla u(x))_{g_x}+C^{n}(x)u(x)
      +\int_{M}u(y)-T^{\epsilon,\beta}_x(u,y)\;\mu_x^{n}(dy)+\textnormal{(Error)}_{L_n,u,x};
  \end{align*}
  and if $\beta\in (0,1)$, then (note there is no remainder term in this case)
  \begin{align*}
    L_n(u,x)= C^{n}(x)u(x) + \int_{M}u(y)-u(x)\;\mu_x^{n}(dy).
  \end{align*}
  Moreover, for every $\epsilon>0$ fixed, the term $\textnormal{(Error)}_{L_n,u,x}$ satisfies the estimate
  \begin{align*}
    |\textnormal{(Error)}_{L_n,u,x}| \leq Ch_n\|u\|_{C^3(M)}.  	  
  \end{align*}	         
  While $A^{\epsilon,n}(x)$, $B^{\epsilon,n}(x)$, and $C^{n}(x)$ satisfy the estimates
  \begin{align*}
    |A^{\epsilon,n}(x)|_{g_x} \leq C,\;\;|B^{\epsilon,n}(x)|_{g_x} \leq C,\;\;|C^{n}(x)|\leq C. 
  \end{align*}	  
  In all cases $C$ denoting a universal constant.
\end{lem}

\begin{proof}[Proof of Lemma \ref{lem:DiscreteMeasuresDInLocalAndNonlocalParts}]
  When $\beta\in(0,1)$, then we just apply Lemma \ref{lem:DiscreteMeasuresDIn} directly to $L_n$, and the Lemma in this case is trivial. For $\beta>1$, the key observation is that we can write, for fixed $x\in G_{n'}$ and $n\geq n'$,
  \begin{align*}
    L_n(\cdot,x) = L_n(\cdot,x)\circ T^{\epsilon,\beta}_x + L_n(\cdot,x)\circ(\Id - T^{\epsilon,\beta}_x).
  \end{align*}
  Then, the first three terms in the desired expression for $L_n(u,x)$ arise from $L_n(\cdot,x)\circ T^{\epsilon,\beta}_x$, using Definition \ref{def:TaylorPolynomial} to obtain $A^{\epsilon,n}(x)$ and $B^{\epsilon,n}(x)$. The term $\textnormal{(Error)}_{L_n,u,x}$ arises simply due to the perturbation of the gradient and Hessian made when applying $\pi_n^\beta$. However, Lemma \ref{lem:InterpolatedDerivativesAreCloseToRealDerivatives} guarantees the error made is bounded by $Ch_n\|u\|_{C^3(M)}$.
  
  As for the term $L_n(\cdot,x)\circ(\Id - T^{\epsilon,\beta}_x)$, note that by definition 
  \begin{align*}
    u(x)-T^{\epsilon,\beta}_x(u,x)=0,
  \end{align*}
  and so the terms $C^n(x)$ from Lemma \ref{lem:DiscreteMeasuresDIn} are not present in the representation of the second term.
\end{proof}

The next lemma yields lower bounds for $A^{\epsilon,n}$ and $\mu^{n}_x$. These bounds say that for large $n$, (and for fixed $\epsilon$ and $x\in \cup G_n$), $A^{\epsilon,n}$ is almost a positive semi-definite matrix, and $\mu^{n}_x$ is almost a positive measure.

\begin{lem}\label{lem:DiscreteIngredientBounds}  There is a universal constant $C$, such that if $x\in G_{n'}$, and $n\geq n'$, then:

  With $I$ denoting the identity map $(TM)_x\to (TM)_x$, we have
  \begin{align*}
    A^{\epsilon,n}(x)\geq -Ch_n^\gam\epsilon^{-3}I.
  \end{align*}
  Moreover,  for all $f\in C^3_b(M)$ such that $f\geq 0$ and $f(x)=0$, we have
  \begin{align*}
    \int_M f(y) \mu^n_x(dy)\geq -Ch_n^\gam\norm{f}_{C^3(M)}.
  \end{align*}
  Here, $h_n$ is as defined in \eqref{eqFinDimApp: hn def}, and $\gamma$ is as in Lemma \ref{lem:EnTnEstToUInC3}.
\end{lem}

\begin{proof}[Proof of Lemma \ref{lem:DiscreteIngredientBounds}]
  Both of these results are immediate consequences of Lemma \ref{lem:In has an approximate GCP}.  Indeed, for the case of $A^{\epsilon,n}(x)$, consider a fixed unit vector, $v\in (TM)_x$, and the function
  \begin{align*}
    w(y)=\tilde\eta^\epsilon_x(y)q(v\otimes v,x;y),
  \end{align*}		
  where $\tilde \eta^\epsilon_x(y)$ is the function from Definition \ref{def:EtaAndTildeEtaDef}.		
		
  On the other hand, from the definition of $q$, we have that $\grad^2w(x)=v\otimes v$ and $\grad w(x)=0$, see Remark \ref{rem:FinDimApp l and q derivatives in a chart}.  It is also clear that  $w(y)\geq0$ for all $y$ and that $w(x)=0$. Then, applying Lemma \ref{lem:DiscreteMeasuresDIn} to $w$,  it follows that 
  \begin{align*}
    L_n(w,x) =\int_{M} \tilde\eta^\epsilon_x(y)q(v\otimes v,x;y) \mu^n_x(dy).
  \end{align*}
  In light of the formula \eqref{eqMMonM:Adeln}, we have that
  \begin{align*}
    L_n(w,x) = \Tr(A^{\epsilon,n}(x) v\otimes v).
  \end{align*}
  Then, using Lemma \ref{lem:In has an approximate GCP} to bound $L_n(w,x)$, we conclude that
  \begin{align*}
     \Tr(A^{\epsilon,n}(x) v\otimes v) \geq -Ch_n^\gamma \|w\|_{C^3(M)}.
  \end{align*}
  Using that $\tilde\eta^\epsilon_x$, $\norm{w}_{C^3}\leq C\epsilon^{-3}$, as well as Proposition \ref{prop: local interpolation operators}, we obtain the lower bound for $A^{\epsilon,n}(x)$.
  
  It remains to prove the bound for $\mu^n_x$. We use Lemma \ref{lem:DiscreteMeasuresDIn} once again, and apply to a function $f\in C^3(M)$ such that $f(x)=0$, which yields
  \begin{align*}
    L_n(f,x) = \int_{M} f(y)\mu^n_x(dy).
  \end{align*}
  Then, Lemma \ref{lem:In has an approximate GCP} applied to the left hand side yields the desired bound.
\end{proof}

The next lemma is concerned with the ``pointwise'' limits for sequences $\{L_n\}$ where for each $n$ we have $L_n \in \D I_n$ for each $n$. The lemma says essentially the following: given $x \in \bigcup G_{n'}$, the sequence $\{L_n(\cdot,x)\}_{n\geq n'}$, seen as a sequence of linear functionals $C^\beta_c(M) \to \mathbb{R}$, must converge along a subsequence to a functional of Levy type based at $x$ (recall Definition \ref{def:LevyTypeFunctional}).

\begin{lem}\label{lem:DInCompactness}
  Let $x \in G_{n'}$, and for every $n\geq n'$ let $L_n \in \D I_n$.  There is a subsequence $n_k\to \infty$ such that $L_{n_k}(\cdot,x)$ converges weakly to some $L_x:C^\beta_b \to \mathbb{R}$, that is, 
  \begin{align*}
    \lim\limits_{k\to \infty} L_{n_k}(u,x) = L_x(u),\;\;\;\forall\;u\in C^\beta_c(M),
  \end{align*}
  where $L_x$ is a functional of L\'evy-type based at $x$. Furthermore, the functional $L_x$ inherits an analogue of \eqref{eqMMonM:AssumptionQuatifiedLocalUnifConv1MODIFIED}, namely, there is a universal $C$ such that
  \begin{align*}
    |L_x(u)| \leq C\|u\|_{C^\beta(B_{2r+3})}+C\omega(r)\|u\|_{L^\infty(M)}.    	  
  \end{align*}	  
\end{lem}

\begin{rem}
  The proof below will actually say more than what was stated in Lemma \ref{lem:DInCompactness}, and it shall highlight how Levy operators arise naturally as the limits of the Laplacian on sequences of weighted graphs that are becoming large as $n\to\infty$. 
  
  Concretely, fix $x\in G_{n'}$. Let $C^n(x), A^{\epsilon,n}(x)$, and $B^{\epsilon,n}(x)$ be as in \eqref{eqMMonM:Cn},\eqref{eqMMonM:Adeln}, and \eqref{eqMMonM:Bdeln}. Then, as shown below, there are subsequences $n_k \to \infty,\epsilon_j \to 0$ such that: 1) we have the limits
  \begin{align*}
    & A(x) := \lim\limits_{j}\lim\limits_{k}A^{\epsilon_j,n_k}(x),\;\;\; B(x) := \lim\limits_{j}\lim\limits_{k}B^{\epsilon_j,n_k}(x),\;\;\; C(x) := \lim\limits_{k}C^{n_k}(x),
  \end{align*}	  
  2) $\mu_{x}^{n_k}$ converges weakly in compact subsets of $M\setminus \{x\}$ to a positive measure $\mu_x$ and 3) for every $u \in C^\beta_c(M)$ we have
  \begin{align*}
    \lim\limits_{k\to \infty} L_{n_k}(u,x) & = \Tr(A(x) \nabla^2u(x))+(B(x),\nabla u(x))_{g_x}+C(x)u(x)\\
	  & \;\;\;\;+\int_{M\setminus \{x\} }u(y)-u(x)-\chi_{B_{r_0}}(y)(\nabla u(x),\exp_x^{-1}(y))_{g_x}\;\mu_x(y).	  
  \end{align*}	  
\end{rem}

\begin{proof}[Proof of Lemma \ref{lem:DInCompactness}]
	For this proof, we only demonstrate the case of $C^\beta_b=C^2_b$ as it includes all of the details.  The other four cases of $\beta$ follow from a similar and simpler argument.
	
	\textbf{The case, $C^\beta_b =C^2_b$.}

	Let $x\in\union_n G_{n}$ be fixed.  Since $G_n$ are increasing, we know $x\in G_n$ for all $n\geq n'$ for some $n'$.  Also, by Lemma \ref{lem:MuNTotalVariation}, we know that the measures $m_x^n$ have bounded variation in $M\setminus \{x\}$, so we are free to use the Jordan decomposition to write
	\begin{align*}
		m_x^{n} = (m_x^{n})^+ - (m_x^{n})^-.
	\end{align*}
	Furthermore, both of the measures $(m_x^{n})^+$ and $(m_x^{n})^-$ are uniformly bounded in $x$, and $n$, for $n\geq n'$ given by Lemma \ref{lem:MuNTotalVariation}.

	\underline{Step 1: extracting weak limits in $n$ for $\epsilon$ fixed.}
	
	We can use the compactness of Radon measures, e.g. \cite[p. 55]{EvGa-92} to extract weakly convergent subsequences of $(m_x^n)^+$ and $(m_x^n)^-$, and hence also $m_x^n$.  We will label by $n_k$, and we will call the weak limiting signed measure as $\bar m_x$, i.e.
	\begin{align*}
		m^{n_k}_x \rightharpoonup \bar m_x,
	\end{align*}
	but we note that a posteriori we will validate that $\bar m_x\geq0$.
	For the moment, we keep $\epsilon$ fixed.

	Let $u\in C^2_c(M)$.  By Lemma \ref{lem:DiscreteMeasuresDInLocalAndNonlocalParts}, we have
	\begin{align}
	    L_n(u,x) & = \Tr(A^{\epsilon,n}(x)\nabla^2u(x))+(B^{\epsilon,n}(x),\nabla u(x))_{g_x} + C^{n}(x)u(x)\label{eqMMonM:CompactnessLocalPart1}\\
		&\ \ \ \ \ \ \  +\int_{M\setminus\{x\}}u(y)-T^{\epsilon,\beta}_x(u,y)\;\mu_x^{n}(dy)+\textnormal{(Error)}_{L_n,u,x}\label{eqMMonM:CompactnessNonlocalPart1}. 
	\end{align}  
	First we work on the nonlocal part, \eqref{eqMMonM:CompactnessNonlocalPart1}.  We see that
	\begin{align*}
		&\int_{M\setminus\{x\}}u(y)-T^{\epsilon,\beta}_x(u,y)\;\mu_x^{n}(dy)\\
		&= \int_{M\setminus\{x\}}\frac{u(y)-T^{\epsilon,\beta}_x(u,y)}{\rho((d(x,y)^2)}\rho((d(x,y)^2)\mu_x^{n}(dy)\\
		&= \int_{M\setminus\{x\}}\frac{u(y)-T^{\epsilon,\beta}_x(u,y)}{\rho((d(x,y)^2)}m_x^{n}(dy).
	\end{align*}
	At this point, we note that by the $C^2$ nature of $u$, the function,
	\begin{align*}
		\frac{u(y)-T^{\epsilon,\beta}_x(u,y)}{\rho((d(x,y)^2)},
	\end{align*}
	does in fact extend to a continuous function on $M$.  Hence, by the weak limit of $m_x^{n_k}$, we see then that
	\begin{align*}
	\lim_{k\to\infty} \int_{M\setminus\{x\}}\frac{u(y)-T^{\epsilon,\beta}_x(u,y)}{\rho((d(x,y)^2)}m_x^{n_k}(dy)
	= \int_{M\setminus\{x\}}\frac{u(y)-T^{\epsilon,\beta}_x(u,y)}{\rho((d(x,y)^2)}\bar m_x(dy).
	\end{align*}
	We define the limiting L\'evy measure on $M\setminus\{x\}$ as 
	\begin{align*}
		\bar \mu_x(dy) = (\rho(d(x,y)^2))^{-1}\bar m_x(dy),
	\end{align*}
	and we note that by Lemma \ref{lem:DiscreteIngredientBounds} we also know that $\bar \mu_x$ is indeed non-negative and satisfies, by definition the integrability condition independent of $x$
	\begin{align*}
		\int_{M\setminus \{x\} } \min(d(x,y)^2,1)\bar \mu_x(dy)\leq C
	\end{align*}
	because by definition $\bar m_x$ are finite measures with total mass independent of $x$.

	Next, we move on to the local part of $L_n$, given in Lemma \ref{lem:DiscreteMeasuresDInLocalAndNonlocalParts}, which we recorded in \eqref{eqMMonM:CompactnessLocalPart1}.  We will establish that the matrices $A^{\epsilon,n}(x)$ and vectors $B^{\epsilon,n}(x)$ are all uniformly bounded in $\epsilon$, $n$, $x$.  Thus, weak limits are immediate (as bounded sequences in Euclidean space). First, we note by a direct calculation that for $x$ fixed, as functions of $y$, $\tilde\eta^\epsilon_x q(x,e_i\otimes e_j;y)$ are in $C^2_b(M)$, independent of $x$, $n$, and $\epsilon$.  Furthermore, the functions $\eta^\epsilon_x(y)l(x,e_i;y)$ have a bounded $C^2$ norm inside, e.g. $y\in \overline{B_{1/2}(x)}$.  Thus the bounds for $A^{\epsilon,n}(x)$ follow from Lemma \ref{lem:DInOperatorNormBound}, and the bounds for $B^{\epsilon,n}(x)$ follow from Lemma \ref{lem:NormEstimateLocalizedToNearAndFarTermsWithTestFunction}, equation \eqref{eqMMonM:NormEstimateCbetaAndLinfty}.
	
	This means that we also have coefficients that depend on $\epsilon$
	\begin{align*}
		\bar A^\epsilon(x),\ \ \bar B^\epsilon(x),\ \ \bar C(x),
	\end{align*}
	such that along a subsequence, again labeled as $n_k$, we have (recall, $x$ is fixed)
	\begin{align*}
		\lim_{k\to\infty} L_n(T^{\epsilon,\beta}_x u, x) = 
		 \Tr(\bar A^\epsilon(x)\nabla^2u(x))+(\bar B^\epsilon(x),\nabla u(x))_{g_x}+\bar C(x)u(x).
	\end{align*}
	Furthermore, by Lemma \ref{lem:DiscreteIngredientBounds}, we see that 
	\begin{align*}
		\bar A^\epsilon(x)\geq 0.
	\end{align*}

	\underline{Step 2:  removing the $\epsilon$ dependence.}
	
	We note that the definition of the $\epsilon$-Taylor expansion (Definition \ref{def:TaylorPolynomial}) requires smooth approximations of $\Indicator_{B_{r_0}(x)}$ and $\Indicator_{\{x\}}$, with $\epsilon$ being a small parameter.
	First, we note that in the previous paragraph, it was established that $A^{\epsilon,n}(x)$ and $B^{\epsilon,n}(x)$ are bounded independently of $\epsilon$, $n$, and $x$.  Thus the limits $\bar A^\epsilon(x)$ and $\bar B^\epsilon(x)$ are still bounded independently of $\epsilon$ and $x$.  Invoking once again the compactness of bounded closed sets in finite dimensional spaces, we obtain a subsequence in $\epsilon$, along which 
	\begin{align*}
		\lim_{\epsilon_k\to0} 
		 &\Tr(\bar A^{\epsilon_k}(x)\nabla^2u(x))+(\bar B^{\epsilon_k}(x),\nabla u(x))_{g_x}+\bar C(x)u(x)\\
		 &=
		 \Tr(\bar A(x)\nabla^2u(x))+(\bar B(x),\nabla u(x))_{g_x}+\bar C(x)u(x),
	\end{align*}
	and again, we preserve 
	\begin{align*}
		\bar A(x)\geq 0.
	\end{align*}
	Next we conclude with the $\epsilon\to0$ limits for 
	\begin{align*}
		\int_{M\setminus\{x\}} u(y)-T^{\epsilon,\beta}_x(u,y) \bar \mu_x(dy).
	\end{align*}
	Using the bound on the error term in the Taylor expansion (see Proposition \ref{prop:Appendix Cbeta implies bound on Taylor remainder}), and since $u\in C^2_c$, we have that as $y\to x$,
	\begin{align*}
          |u(y)-u(x)-\eta^\epsilon(y)l(x,\grad u(x);y)|\leq \|u\|_{C^2}\rho(d(x,y)^2).
	\end{align*}
	Hence, by dominated convergence, we see that (recall Definition \ref{def:EtaAndTildeEtaDef} for $\eta^\epsilon$)
	\begin{align*}
		\lim_{\epsilon\to0} &\int_{M\setminus\{x\}} u(y)-u(x)-\eta^\epsilon(y)l(x,\grad u(x);y)\bar \mu_x(dy)\\
		&= \int_{M\setminus\{x\}} u(y)-u(x)-\Indicator_{B_{r_0}(x)}l(x,\grad u(x);y)\bar \mu_x(dy).
	\end{align*}
	For the quadratic term, $q(x,\nabla^2 u(x);y)$, we note that
	\begin{align*}
		\abs{\tilde\eta^\epsilon(y)q(x,\nabla^2 u(x);y)}\leq C\norm{u}_{C^2(M)}\rho(d(x,y)^2)\Indicator_{B_{2\epsilon}(x)}.
	\end{align*}
	Hence, since $\bar\mu_x\geq 0$,
	\begin{align*}
		\int_{M\setminus\{x\}}\abs{\tilde\eta^\epsilon(y)q(x,\nabla^2 u(x);y)} \bar\mu_x(dy)
		\leq C\norm{u}_{C^2(M)}\int_{B_{2\epsilon}(x)\setminus\{x\}} \rho(d(x,y)^2)\bar\mu_x(dy). 
	\end{align*}
	Since $\bar m_x$ is a finite measure, we see by the continuity of $\bar m_x$ that necessarily
	\begin{align*}
		\bar m_x(B_{2\epsilon}(x)\setminus\{x\})\to 0\ \text{as}\ \epsilon\to0.
	\end{align*}
	We conclude then that
	\begin{align*}
		\lim_{\epsilon\to0} \int_{M\setminus\{x\}}\abs{\tilde\eta^\epsilon(y)q(x,\nabla^2 u(x);y)} \bar\mu_x(dy) = 0.
	\end{align*}
	
	This means that after the subsequential limits first in $n$ followed by $\epsilon$, we do indeed recover for $u\in C^2_c(M)$,
	\begin{align*}
		\lim_{\epsilon_j\to0}\lim_{n_k\to\infty} L_{n_k}(u,x) = \bar L_x(u),
	\end{align*}
	and $\bar L_x$ is a functional of the L\'evy form (Definition \ref{def:LevyTypeFunctional}, \eqref{eqIntro:LevyTypeDef}).  This concludes the lemma for the case $\beta=2$.

	Now we make a few remarks as to how the remaining cases follow from the proof for $\beta=2$.  This is the only part in the proof in which there is a true distinction between them, and it all rests on the ability to extend continuously the function
	\begin{align*}
		\frac{u(y)-T^{\epsilon,\beta}_x(u,y)}{\rho((d(x,y)^\beta)}.
	\end{align*}
	
	\textbf{The case, $\beta=1$, $C^1_b(M)$.}
	This case is completely analogous to $\beta=2$, and on one hand simpler because $A^{\epsilon,n}(x)\equiv0$, but on the other hand, complicated by Lemma \ref{lem:MuNTotalVariation}.  Now, we let $\ep\in(0,1)$ be given, and we take $u\in C^{1,\ep}_b$, and we invoke Lemma \ref{lem:MuNTotalVariation} with $1+\ep/2$.  Taylor's theorem applies in exactly the same way for the continuity of the quantity
	\begin{align*}
		\frac{u(y)-T^{\epsilon,1}_x(u,y)}{\rho((d(x,y)^{1+\ep/2})},
	\end{align*}
	at $y=x$, where now the numerator has slightly stronger decay, by choice of $u\in C^{1,\ep}$.
	
	\textbf{The cases of $C^\beta_b(M)=C^{0,1}(M)$ and $C^\beta_b(M)=C^{1,1}(M)$.}
	These cases go in the same way as respectively the cases of $C^1$ and $C^2$ because we limit ourselves to only checking the formula for $u\in C^1$ and respectively $C^2$.  Hence, the respective continuity of, e.g.
	\begin{align*}
          \;\frac{u(y)-T^{\epsilon,1}_x(u,y)}{\rho((d(x,y))} \textnormal{ and }  \frac{u(y)-T^{\epsilon,2}_x(u,y)}{\rho((d(x,y)^2)},
	\end{align*}
	is unchanged.
	
	\textbf{The other cases of $\beta\in(0,2)$.}
	
	The only real difference here is that in these cases, we are applying the argument to $u\in C^{\beta+\ep}_b(M)$ for some small $\ep>0$. In this case, the slightly larger H\"older exponent, $\beta+\ep$, is what gives the continuity of
	\begin{align*}
		\frac{u(y)-T^{\epsilon,\beta}_x(u,y)}{\rho((d(x,y)^\beta)},
	\end{align*} 
	because the numerator is of the order $d(x,y)^{\beta+\ep}$.
\end{proof}

In the case $I$ satisfies the equicontinuity assumption \eqref{eqIntro:Equicontinuity Assumption}, one can do better than Lemma \ref{lem:DInCompactness}: one can show the compactness of the elements of $L_n \in \mathcal{D}I_n$ as linear operators. Moreover, the proof is rather straightforward, it boils down to the Arzel\'a-Ascoli theorem. 
\begin{lem}\label{lem:DInCompactness Strong}
  Suppose that $I$ satisfies \eqref{eqIntro:Equicontinuity Assumption}. Then, given a sequence $\{L_n\}$ with $L_n \in \D I_n$ for every $n$, there exists a subsequence $L_{n_k}$ and a bounded linear operator $L:C^\beta_b \to C_b$ such that
  \begin{align*}
    \lim\limits_{k\to\infty} L_{n_k}(u,x) =  L(u,x),\;\;\forall\;u\in C^3_c(M),\;\;x\in M.
  \end{align*}	  
\end{lem}

\begin{proof}
  Fix $K\subset M$ be a compact set, and let $\mathcal{B}$ denote the set
  \begin{align*}
    \mathcal{B} := \{ u \in C^3_{b}(M)  \mid u\equiv 0 \textnormal{ outside } K, \|u\|_{C^3(M)}\leq 1\}.
  \end{align*}	
  It is clear that $\mathcal{B}$ is a compact subset of $C^\beta_b$, for each $\beta<3$. From the assumption \eqref{eqIntro:Equicontinuity Assumption}, the continuity of $\pi_n^\beta$ (Theorem \ref{thm:EnTnContOperatorInCbetaNorm}), and the convergence of $\pi_n^\beta$ to the identity in $C^3_b$ (Lemma \ref{lem:EnTnEstToUInC3}), it follows that if $L_n \in \D I_n(v_n)$, where $v_n \in \mathcal{B}$, then 
  \begin{align*}	
    \{L_{n}(u,\cdot)\}_{u\in \mathcal{B}} \textnormal{ is equicontinuous }   	
  \end{align*}	
  In other words, the real valued functions given by
  \begin{align*}
    (u,x) \in K\times \mathcal{B} \to L_{n}(u,x),	   
  \end{align*}	  
  form an equicontinuous family of functions from $K\times \mathcal{B}$ to $\mathbb{R}$. In particular, this family of functions is precompact in $C(K\times \mathcal{B})$ with respect to uniform convergence. Therefore, there is some subsequence $n_k$ and some $L \in C(K\times \mathcal{B})$ such that
  \begin{align*}
    L_{n_k} \to L \textnormal{ uniformly in } K\times \mathcal{B}.
  \end{align*}
  By homogeneity, $L_{n_k}$ converges as a function defined for all functions $u \in C^3$ which are compactly supported on $K$. Moreover, using the linearity of the $L_{n_k}$ it is clear that $L$ is also a linear operator. Then, taking an increasing sequence of compacts $K_n$ which cover $M$, one can apply a Cantor diagonalization argument to obtain the desired sequence.
\end{proof}

\subsection{Limits of the finite dimensional min-max-- the proof of Theorem \ref{thm: Min-Max Riemannian Manifold} and Proposition \ref{prop:CasesOnBetaForMinMax} }

Now that we have collected various facts about $\D I_n$, we have enough information to finish the proof of Theorems \ref{thm: Min-Max Riemannian Manifold} and \ref{thm: Min-Max Riemannian Manifold With Operators}.  The last remaining step is to pass to the limit ``inside'' of the min-max.

\begin{proof}[Proof of Theorem \ref{thm: Min-Max Riemannian Manifold}]

The key point of this proof is to use the compactness established in Lemma \ref{lem:DInCompactness} to go from the min-max formula for $I_n$ to one for $I$. We introduce the family depending on $I$, 
\begin{align}
  \K_{Levy}(I) := &\hull\big( \big\{L:C^\beta_b \to \mathbb{R}\ : \exists \ n_k \to \infty \textnormal{ and } L_{n_k}\in \D I_{n_k}, \notag \\ 
		&   x_{k} \in G_{n_k} \text{ s.t.}\ L(f) = \lim_{k\to\infty} L_{n_k}(f,x_k)\ \forall f\in C^\beta_c(M)\label{eqMMonM:Def KLevy} \big\}   \big).
\end{align}
Among the implications of Lemma \ref{lem:DInCompactness}, $\K_{Levy}(I) \neq \emptyset$, and every element of $\K_{Levy}(I)$ is an operator of Levy type based at some $x\in M$. Then, our aim is to prove the following: for every $x\in M$, and every pair $u,v\in C^\beta_b(M)$, there is some $L \in \K_{Levy}(I)$ based at $x$ such that
\begin{align}\label{eqMMonM:ThisIsWhereYouWin}
    I(u,x) \leq I(v,x) + L(u-v).
\end{align}	
We proceed to prove \eqref{eqMMonM:ThisIsWhereYouWin} in increasing order of generality: 1) for all $u,v \in C^\beta_c$ and $x\in G_n$ for some $n$, 2) for all $u,v \in C^\beta_c$ and any $x \in M$, and finally 3) for all $u,v \in C^\beta_b$ (that is, $u,v$ that may not be compactly supported) and any $x\in M$.

Fix $u,v \in C^\beta_c(M)$, and let $x\in G_{n'}$, for some $n' \in \mathbb{N}$. Since the $G_n$ are increasing, we have that $x \in G_n$ for all $n\geq n'$. The min-max formula for $I_n$ with $n\geq n'$ (Lemma \ref{lem:InMinMax}) yields the existence of some $L_{n,u,x} \in \mathcal{D}I_n$ such that
\begin{align*}
  I_n(u,x) \leq I_n(v,x) + L_{n,u,x}(u-v,x).
\end{align*}
Given that $u,v \in C^\beta_c(M)$, Proposition \ref{prop:InConvergesUniformlyToI} guarantees that 
\begin{align*}
  \lim \limits_{n} I_n(u,x) = I(u,x),\;\;\lim \limits_n I_n(v,x) = I(v,x),
\end{align*}
and in particular,
\begin{align*}
  I(u,x) \leq I(v,x) + \limsup_{n} L_{n,u,x}(u-v,x).
\end{align*}
Applying Lemma \ref{lem:DInCompactness}, and the definition of $\K_{Levy}(I)$  \eqref{eqMMonM:Def KLevy}, we conclude the following: for any $x\in \bigcup G_n$, and $u,v\in C^\beta_c(M)$, there is a functional $L\in \K_{Levy}(I)$, based at $x$,  such that
\begin{align*}
  I(u,x) \leq I(v,x) + L(u-v).
\end{align*}
More generally, if $x\in M$, then we can choose a sequence of points $x_m$ with $x_m \to x$ and $x_m \in \bigcup G_n$. Then, for each $m$ there is some $L_{x_m}$ based at $x_m$ such that
\begin{align*}
  I(u,x_m) \leq I(v,x_m) +L_{x_m}(u-v)
\end{align*}
Once again, passing to the limit in $m$ (and using again the compactness of $\K_{Levy}(I)$), and using the continuity of $I(u,\cdot)$ and $I(v,\cdot)$, we conclude that there exists some $L\in \K_{Levy}(I)$, based at $x$, and such that
\begin{align*}
  I(u,x) \leq I(v,x) +L(u-v).
\end{align*}
Finally, we need to extend \eqref{eqMMonM:ThisIsWhereYouWin} to all $u,v\in C^\beta_b(M)$, and not just those with compact support. Fix $u,v \in C^\beta_b(M)$, and $x\in M$. Consider sequences $u_k,v_k \in C^\beta_c(M), k\in\mathbb{N}$, which are such that
\begin{align*}
  \|u_k-u\|_{C^\beta(B_{2k}(x_*))} \leq 1/k,\;\;\; \|v_k-v\|_{C^\beta(B_{k}(x_*))} \leq 1/k.  
\end{align*}
Then, for each $k$ we have some $L_k \in \K_{Levy}(I)$ such that
\begin{align*}
  I(u_k,x) \leq I(v_k,x) + L_{k}(u_k-v_k).
\end{align*}
The assumption \eqref{eqIntro:AssumptionQuatifiedLocalUnifConv1} and Lemma \ref{lem:DInSatisfiesExtraAssumptions} imply that for all sufficiently large $k$,
\begin{align*}
  I(u,x) & \leq I(u_k,x) + C\|u-u_k \|_{C^\beta(B_{2k}(x_*))}+C\|u-u_k\|_{L^\infty(M)},\\
  I(v_k,x) & \leq I(v,x) + C\|v-v_k \|_{C^\beta(B_{2k}(x_*))}+C\|v-v_k\|_{L^\infty(M)},\\
  L_{k}(u_k-v_k) & \leq L_{x,k}(u-v) +C\|u-v-(u_k-v_k)\|_{C^\beta(B_{2k}(x_*))} \\
  & \;\;\;\;+C\omega(k)\|u-v-(u_k-v_k)\|_{L^\infty(M)}
\end{align*}
Therefore, 
\begin{align*}
  I(u,x) \leq I(v,x) + L_{k}(u-v) + C\frac{1}{k}+C\omega(k),\;\;\forall\;k\in\mathbb{N}.
\end{align*}
Then, after possibly taking a subsequence of the $L_k$, we obtain \eqref{eqMMonM:ThisIsWhereYouWin} in the limit in this final case. Since \eqref{eqMMonM:ThisIsWhereYouWin} trivially yields equality for $v=u$, we conclude that for any $x\in M$
\begin{align*}
  I(u,x) = \min\limits_{v\in C_b^\beta(M)} \max \limits_{L \in \K_{Levy}(I)} \{  I(v,x) +  L(u-v)\},
\end{align*}
and this finishes the proof.

\end{proof}

The nature of the set $\K_{Levy}(I)$ and its dependence on $I$ is a direct and trivial outcome of the proof of Theorem \ref{thm: Min-Max Riemannian Manifold}, we record it as a Proposition.

\begin{prop}\label{prop:ReduceMinMaxToWeakLimitsOfDIn}
  The family $\K_{Levy}(I)$ appearing in Theorem \ref{thm: Min-Max Riemannian Manifold} has the form
  \begin{align*}
    \K_{Levy}(I) := &\hull\big( \big\{L:C^\beta_b \to \mathbb{R}\ : \exists \ n_k \to \infty \textnormal{ and } L_{n_k}\in \D I_{n_k},\\ 
		&   x_{k} \in G_{n_k} \text{s.t.}\ L(f) = \lim_{k\to\infty} L_{n_k}(f,x_k)\ \forall f\in C^\beta_c(M) \big\}   \big).
  \end{align*}
\end{prop}

Finally, we comment on the minor modifications needed to obtain the stronger min-max result, under assumption \eqref{eqIntro:Equicontinuity Assumption}.
\begin{proof}[Proof of Theorem \ref{thm: Min-Max Riemannian Manifold With Operators}]
  The proof is exactly as that of the previous Theorem, except we invoke Lemma \ref{lem:DInCompactness Strong} in place of Lemma \ref{lem:DInCompactness}, which is made possible once we have \eqref{eqIntro:Equicontinuity Assumption}. In this case, we obtain convergence as operators of subsequences of $L_n$, where $L_n\in \mathcal{D}I_n$ for every $n$. We define
  \begin{align*}
    \mathcal{L} := \{ L \mid \exists \{n_k\}_k, n_k\to \infty, \textnormal{ and } L_{n_k} \in \mathcal{D} I_{n_k} \textnormal{ such that } L(u,x) = \lim\limits_{k} L_{n_k}(u,x) \forall u\in C^3_c(M)\}. 
  \end{align*}	  
  The min-max formula using the operators in $\mathcal{L}$ is proved as before, and the fact that for each $x$ we have $L(\cdot,x)\in L_{Levy}(I)$ is immediate in light of Proposition \ref{prop:ReduceMinMaxToWeakLimitsOfDIn}.
  
\end{proof}

\begin{rem}
	In order to illustrate the difference between $I: C^\beta_b(M)\to C_b(M)$ and maps on a finite dimensional space, we point the reader to \eqref{eqMMonM:ThisIsWhereYouWin}.  If $I$ were differentiable on a dense set of functions, one can basically go straight to this point-- see e.g. Proposition \ref{prop: mean value theorem for Clarke's subdifferential} and the proof of Lemma \ref{lem:FiniteDimensionalMinMax}.  However, for generic Lipschitz $I$ in infinite dimensional spaces, Fr\'echet differentiability on a dense set is not expected to hold.  Thus, most of the difficulty was contained in obtaining \eqref{eqMMonM:ThisIsWhereYouWin}.
\end{rem}

\begin{rem}
	One may ask how it is that including such a large set of linear functionals centered at $x$ as $\K_{Levy}(I)$ in the max of the min-max formula \eqref{eqIntro:MinMax!} does not corrupt simpler operators that may not use all such linear functionals.  Suppose that $I$ is a simpler operator of the form
	\begin{align*}
		I(u,x) = \max\{L_a(u,x), L_b(u,x)\},
	\end{align*}
	where $L_a$ and $L_b$ are simply two fixed operators that have the GCP and properties \eqref{eqIntro:AssumptionQuatifiedLocalUnifConv1}.  The reader can check in a straightforward fashion that indeed
	\begin{align*}
		\min_{v\in C^\beta_b} \max_{L_x\in \K_{Levy}(I)} \left( I(v,x) + L_x(u-v)    \right)
		= \max\{L_a(u,x), L_b(u,x)\}.
	\end{align*}
	The main points are that choosing $v=u$ in the minimum immediately gives one inequality, and the reverse inequality comes from the fact that if $L_a$ and $L_b$ are linear maps from $C^\beta_b(M)\to C_b(M)$ with the GCP, then for $x$ fixed, the linear functionals $L_a(\cdot,x)$ and $L_b(\cdot,x)$ are both of L\'evy type, and hence in $\K_{Levy}(I)$.
\end{rem}


\subsection{Convex operators} If the Lipschitz operator $I$ is assumed to be \emph{convex}, then it may be represented simply as a maximum of linear operators of the same type as those appearing in the min-max formula from Theorem \ref{thm: Min-Max Riemannian Manifold}. First, let us recall what it means for an operator to be convex. 

\begin{DEF} An operator $I$ is said to be convex if for any two functions $u,v$, and $x \in M$, and any $\lambda \in (0,1)$ the following inequality holds
  \begin{align*}
    I(\lambda u+(1-\lambda)v,x)\leq \lambda I(u,x) +(1-\lambda)I(v,x).
  \end{align*}	
  The operator is said to be concave if the above inequality is reversed. 
\end{DEF}

The convexity condition can clearly be restated as
\begin{align*}
  t^{-1}\left ( I(v+t(u-v),x)-I(v,x) \right ) \leq I(u)-I(v)  \;\;\forall\;t\in[0,1].
\end{align*}
Taking $s\in [0,1]$ and applying the above inequality to the functions $v$ and $v+s(u-v)$, one sees that convexity of $I$ is equivalent to the condition
\begin{align*}
  t^{-1}\left ( I(v+t(u-v),x)-I(v,x) \right ) \leq s^{-1}\left ( I(v+s(u-v),x)-I(v,x) \right ),\;\;\forall\; 0\leq t\leq s\leq 1. 
\end{align*}

\begin{lem}\label{lem: max formula convex operator} Let $M$ and $I$ be as in Theorem \ref{thm: Min-Max Riemannian Manifold}. If in addition, $I$ is known to be convex, then
  \begin{align*}
    I(u,x) = \max_{v,L_x} \{ I(v,x)+L_x(u-v)\}. 	  
  \end{align*}	  
  Here the maximum is over some family of pairs $(v,L_x)$ where $v\in C^\beta_b$, and each $L_x$ lies in the same family of functionals as in Theorem \ref{thm: Min-Max Riemannian Manifold}. Likewise, if $I$ is concave, an analogous statement holds with a minimum instead of a maximum.
\end{lem}
\begin{proof}
  Let $I_n$ be the finite dimensional approximation to $I$. By its construction, it is clear that $I_n$ is convex if $I$ is convex. We shall show that the min-max formula for $I_n$ reduces to a max formula when $I_n$ is convex. From this point on, the proof of the Lemma follows the argument used to obtain \eqref{eqMMonM:ThisIsWhereYouWin} in the proof of Theorem \ref{thm: Min-Max Riemannian Manifold}.
  
  Fix $u,v \in C^\beta_b$ and $x \in \tilde G_n$ for some $n$. Assume further that $v$ is such that $I_n$ is differentiable at $v$, with derivative $L_v(\cdot)$. Then, due to the convexity of $I_n$, the function
  \begin{align*}	
    t \to t^{-1}\left ( I_n(v+t(u-v),x)-I(v,x) \right ),	
  \end{align*}	
  is nondecreasing for $t>0$. Therefore, 
  \begin{align*}
    t^{-1}\left ( I_n(v+t(u-v),x)-I(v,x) \right ) & \geq \limsup \limits_{t\to 0^+} \{ t^{-1}\left ( I_n(v+t(u-v),x)-I_n(v,x)  \right )\}\\
      & = L_v(u-v,x).	
  \end{align*}
  In particular, for $t=1$
  \begin{align*}
    I_n(u,x) \geq I(v,x)+L_v(u-v,x),\;\;\forall\;x\in \tilde G_n,\;u\in C^\beta_b(M).
  \end{align*}	  	
  If $I_n$ is not differentiable at $v$, we take a sequence $v_k \to v$ with $I_n$ differentiable at each $v_k$. Then,
  \begin{align*}
    I_n(u,x)\geq I_n(v_k,x)+L_{v_k}(u-v_k,x),\;\;\forall\;k,
  \end{align*}
  passing to the limit $k \to \infty$ 
  \begin{align*}
    I_n(u,x) & \geq I_n(v,x)+\limsup \limits_{k}L_{v_k}(u-v_k,x).
  \end{align*}
  From here, it follows that for every $v\in C^\beta_b$ there is some $L$ such that
  \begin{align*}
    I_n(u,x)\geq I_n(v,x)+L_x(u-v).
  \end{align*}
  Since $v$ is arbitrary and the above becomes an equality whenever $u=v$, it follows that we have
  \begin{align*}
    I_n(u,x) = \max \limits_{v,L} \{ I_n(v,x)+L(u-v,x)\},
  \end{align*}	  	
  the maximum being over some family of pairs $(v,L)$. This proves the maximum for each of the finite dimensional approximations $I_n$. As mentioned at the beginning of the proof, to obtain the maximum formula for $I$, one proceeds by the same limiting argument used in the proof of Theorem \ref{thm: Min-Max Riemannian Manifold}, we leave the details to the reader. 
\end{proof}


\subsection{Extremal operators}\label{sec:ExtremalOperators}

An elementary consequence of the min-max formula for $I$, is that one can bound the difference $I(u,x)-I(v,x)$ via ``extremal operators''. Namely, since
  \begin{align*}	
    I(u,v)-I(v,x) & = \min \limits_{v'}\max\limits_{L} \left \{ I(v',x)+L(u-v',x)\right \} -  I(v,x),\\
	 \textnormal{(take } v'=v \textnormal{)} & \leq \max\limits_{L} \left \{ L(u-v,x)\right\}.	
  \end{align*}	
  Likewise,
  \begin{align*}
    I(u,x)-I(v,x) & \geq -\max\limits_{L} \left \{ L(v-u,x)\right \},\\
	  & \geq \min \limits_{L} \left \{ L(u-v,x)\right \}.
  \end{align*}	  
  Therefore, we call the following the extremal inequalities for $I$:
  \begin{align}\label{eqMMonM:ExtremalInequality}
    \min\limits_{L} \{L(u-v,x)\}\leq I(u,x)-I(v,x) \leq \max \limits_{L} \{L(u-v,x)\};	
  \end{align}	
and given a family of linear functionals, $\mathcal{L}$, we define the extremal operators
  \begin{align}\label{eqMMonM:ExtremalOpDef}
    M^+_{\mathcal{L}}(u,x) = \sup\limits_{L\in\mathcal{L}} \{ L(u,x)\}
	\ \ \text{and}\ \ 
    M^-_{\mathcal{L}}(u,x) = \inf\limits_{L\in\mathcal{L}} \{ L(u,x)\}.
  \end{align}
  Note, these extremal operators have made important appearances in PDE and control theory for decades (and most likely in other fields).  For second order equations they can be traced back to Pucci \cite{Pucci-1965SuLeEquazioniEstremanti}, see also their importance in Caffarelli \cite{Caff-1989InteriorEstimates} or in the book of Caffarelli-Cabr\'e \cite[Chp 2-4]{CaCa-95}.  They also play a fundamental role in much of the theory for integro-differential equations for both linear and nonlinear operators (a very abridged list is e.g. \cite{CaSi-09RegularityIntegroDiff},  \cite{ChDa-2012NonsymKernels}, \cite{ChDa-2012RegNonlocalParabolicCalcVar}, \cite{KassRangSchwa-2013RegularityDirectionalINDIANA}, \cite{SchwabSilvestre-2014RegularityIntDiffVeryIrregKernelsAPDE}, etc...).

  Now, one may ask if the reverse holds. Namely, if the \eqref{eqMMonM:ExtremalInequality} holds, does it follow that $I$ can be written as a min-max of operators belonging to the class $\mathcal{L}$?. The next lemma gives a partial answer to this question --which will be useful in a forthcoming work dealing with Dirichlet to Neumann maps. 

\begin{prop}\label{prop:ExtremalBoundsTheOperatorsInMinMax}
	Assume that $I$ is as in Theorem \ref{thm: Min-Max Riemannian Manifold}, and suppose further that there exists a class of functionals $\L$, so that $I$ obeys the extremal inequalities \eqref{eqMMonM:ExtremalInequality} with respect to $\L$.  Then, with $\K_{Levy}(I)$ as in Proposition \ref{prop:ReduceMinMaxToWeakLimitsOfDIn}, it holds that for all $L_x\in\K_{Levy}(I)$
	\begin{align*}
		\forall\ \phi\in C^3_c(M),\ \forall\ x\in M,\ \ M^-_\L(\phi,x)\leq  L_x(\phi).
	\end{align*}
\end{prop} 

We will prove this proposition via two more basic (and possibly also useful) facts separately, where both of them invoke the finite dimensional operators. 
  
\begin{lem}\label{lem:ApproximateExtremalInequality}
	Let $I$ and $\L$ be as in Proposition \ref{prop:ReduceMinMaxToWeakLimitsOfDIn}. Let $I_n$, $(M^-_\L)_n$, and $(M^+_\L)_n$ be the finite dimensional approximations defined in \eqref{eqMMonM:InDef} for respectively $I$, $M^-_\L$, and $M^+_\L$ from \eqref{eqMMonM:ExtremalOpDef}, and let $u,v\in C^\beta_b(M)$.  Then
	\begin{align*}
		(M^-_\L)_n(u-v)\leq \I_n(u)-\I_n(v) \leq (M^+_\L)_n(u-v),
	\end{align*}
	i.e. the approximation \eqref{eqMMonM:InDef} preserves extremal inequalities.
\end{lem}

\begin{lem}\label{lem:ApproxLowerBoundOnFrechetDeriv}
	Let $I$, $\L$, and $M^-_\L$ be as in Proposition \ref{prop:ExtremalBoundsTheOperatorsInMinMax}.  Let $n$ be fixed, let $I_n$ be defined in \eqref{eqMMonM:InDef}, and assume $I_n$ is Fr\'echet differentiable at $u\in C^\beta_b(M)$ with derivative $D\I_{n,u}$, let $x\in M$, and let $\phi\in C^3_b(M)$.  Then the following estimate is true
	\begin{align*}
		-h_n^\gam \norm{\phi}_{C^3} + M^-_\L(\phi,x)\leq D I_{n,u}(\phi,x),
	\end{align*}
	where $h_n^\gam\to0$ arises from Lemma \ref{lem:EnTnEstToUInC3} and is defined in \eqref{eqFinDimApp: hn def}.
\end{lem}

For notational reasons, it will be easiest to simply present the proofs of Lemmas \ref{lem:ApproximateExtremalInequality} and \ref{lem:ApproximateExtremalInequality} together.

\begin{proof}[Proof of Lemmas \ref{lem:ApproximateExtremalInequality} and \ref{lem:ApproxLowerBoundOnFrechetDeriv}]
	First, let $u,v\in C^\beta_b(M)$.  We will use the fact that restriction/extension compositions
	\begin{align*}
		E^\beta\circ T_n\ \text{and}\ E^0\circ T_n
	\end{align*}
	are both linear operators, and furthermore that $E^0\circ T_n$ preserves ordering.  Using the extremal inequality of  \eqref{eqMMonM:ExtremalOpDef}, we see that since $E^\beta_n\circ T_n u$ and $E^\beta_n\circ T_n v$ are again in $C^\beta_b(M)$, it holds that
	\begin{align*}
		M^-_\L(E^\beta_n\circ T_n (u-v))\leq I(E^\beta_n\circ T_n u)-I(E^\beta_n\circ T_n v),
	\end{align*}
	(and we have used linearity of $E^\beta_n\circ T_n$ on the left).  Now we may apply $E^0\circ T_n$ to both sides, and we use the monotonicity and linearity to conclude 
	\begin{align*}
		E^0\circ T_n \left(M^-_\L(E^\beta_n\circ T_n (u-v))\right)\leq E^0\circ T_n\left(I(E^\beta_n\circ T_n u)\right)-E^0\circ T_n\left(I(E^\beta_n\circ T_n v)\right).
	\end{align*}
	Hence by the definition of $(M^-_\L)_n$ and $I_n$ in \eqref{eqMMonM:InDef}, we have obtained half of Lemma \ref{lem:ApproximateExtremalInequality}.  The other inequality follows the same proof.
	
	Now to obtain the estimate on $D I_n$, let $t>0$, and $u$ and $\phi$ be as in Lemma \ref{lem:ApproxLowerBoundOnFrechetDeriv}.  In the preceding equation, we may now replace $u$ by $u+t\phi$ and $v$ by $u$.  Invoking the positive $1$-homogeneity of $M^-_\L$ and $(M^-_\L)_n$, we obtain
		\begin{align*}
			t(M^-_\L)_n(\phi)\leq I_n(u + t\phi)-I_n( u).
		\end{align*}
		Now we can invoke the approximation estimate in Lemma \ref{lem:EnTnEstToUInC3} applied to $(M^-_\L)_n$, and rearrange to see that
		\begin{align*}
			-h_n\norm{\phi}_{C^3}+M^-_\L(\phi,x)\leq (M^-_\L)_n(\phi,x)\leq \frac{1}{t}\left(I_n(u + t\phi,x)-I_n( u,x)\right).
		\end{align*}
		Hence, taking the limit as $t\to0^+$, we conclude Lemma \ref{lem:ApproxLowerBoundOnFrechetDeriv}.
\end{proof}

Now we justify Proposition \ref{prop:ExtremalBoundsTheOperatorsInMinMax}.

\begin{proof}[Proof of Proposition \ref{prop:ExtremalBoundsTheOperatorsInMinMax}]
	
	Let $\phi$ and $x$ be given.  By the definition of $L_x\in \K_{Levy}(I)$ via Proposition \ref{prop:ReduceMinMaxToWeakLimitsOfDIn}, we see that $L_x$ is a limit of convex combinations of operators, $\tilde L$ such that there exist $u_n$ and $x_k$ so that 
	\begin{align*}
		\tilde L(\phi,x) = \lim_{k\to\infty}\lim_{n\to\infty} DI_{n,u_n}(\phi,x_k).
	\end{align*}
	As Lemma \ref{lem:ApproxLowerBoundOnFrechetDeriv} is independent of $u_n$, and from the fact that $M^-_\L(\phi,\cdot)$ is continuous in $x$,  we see that 
	\begin{align*}
		M^-_\L(\phi,x)\leq \tilde L(\phi,x).
	\end{align*}
	This inequality is preserved under further convex combinations over $\tilde L$, and thus we conclude it also holds that
	\begin{align*}
		M^-_\L(\phi,x)\leq L_x(\phi).
	\end{align*}
\end{proof}

\section{Some Questions}\label{sec: Further Questions}
\setcounter{equation}{0}

Here we take the time to mention some additional questions that arise from the min-max representation.   
\begin{question}
  In the Introduction, among the examples for maps satisfying the GCP, we mentioned the Dirichlet-to-Neumann map for a fully nonlinear equation in a bounded smooth domain $\Omega \subset \mathbb{R}^d$. Our main theorem yields the representation
  \begin{align*}
    \partial_\nu U = \min\limits_{a}\max \limits_{b} \left \{ f^{ab}(x)+L_{ab}(u,x) )  \right \},
  \end{align*}
  where $\{f^{ab}\}_{ab}$ is a bounded family of functions in $C(\partial \Omega)$, and each $L_{ab}(\cdot,x)$ has the form \eqref{eqIntro:LevyTypeDef}. Then, we ask: are the Levy measures $\mu^{ab}$ appearing in the min-max formula formula absolutely continuous with respect to the surface measure of $\partial \Omega$?. In other words, find out whether there are measurable functions $k^{ab}: \Omega \times \Omega \to \mathbb{R}$ such that
  \begin{align*}
    \mu^{ab}_x(dy) = k^{ab}(x,y)d\textnormal{vol}_g(y).
  \end{align*}
  Furthermore, deriving further properties for the kernels $k^{ab}$, such as pointwise bounds with respect to the  kernel $|x-y|^{-d}$, would be very useful. Such bounds would mean that the equation $\partial_\nu U=0$ is closely related to existing regularity results for nonlocal elliptic equations, i.e. \cite{BaLe-2002TransitionProb,CaSi-09RegularityIntegroDiff,ChDa-2012NonsymKernels,SchwabSilvestre-2014RegularityIntDiffVeryIrregKernelsAPDE}.
\end{question}

\begin{question}
  Going in the opposite direction, are there Dirichlet to Neumann maps --even in the linear case-- for which the resulting integro-differential operator on $\partial\Om$ that has a singular L\'evy measure?  This seems a possibility for linear operators with low-regularity coefficients, as suggested by the existence of well known examples of elliptic operators for which the associated $L$-harmonic measure is singular.
\end{question}

\begin{question}
  Let $M=\mathbb{R}^d$. If it is assumed that $I$ is a translation invariant operator, can you show that it suffices to only use translation invariant linear operators in the min-max formula of Theorem \ref{thm: Min-Max Riemannian Manifold}?.
\end{question}

\begin{question}
	Can the min-max formula be extended to degenerate or singular operators such as the infinity-Laplace or the p-Laplace?  These operators are not bounded from $C^2\to C$, but nonetheless they enjoy good existence / uniqueness and partial regularity theory for weak solutions of equations defined by them.
\end{question}

\begin{question}
	The axiomatic image processing work of Alvarez-Guichard-Lions-Morel \cite{AlvarezLionsGuichardMorel-1993AxiomsImageProARMA} showed that if a semi-group on the space of continuous functions satisfies certain axioms, most notably locality and comparison, then in fact the semi-group must be characterized as the (viscosity) solution operator for some fully nonlinear (degenerate) parabolic equation.  This is notable because one recovers a representation using weak solutions. Is it possible to make an analog of the paper \cite{AlvarezLionsGuichardMorel-1993AxiomsImageProARMA} to the context of Theorem \ref{thm: Min-Max Riemannian Manifold} presented here?  This would be an extension of Theorem \ref{thm: Min-Max Riemannian Manifold} to both the parabolic setting and the setting of weak solutions.
\end{question}

\appendix

\section{Discretization of the gradient and the Hessian on $M$}

First off, we shall construct proper discretizations for the covariant gradient and Hessian given $M$ and $\tilde G_n$. Our point of view will be to think of a sufficiently smooth function $u:M\to\mathbb{R}$ as given. Then, the discrete gradient and Hessian of $u$ will be defined at points in $\tilde G_n$ using only the values of $u$ at points in $\tilde G_n$. We will see that the regularity of the original function $u$ will control how far are these discrete operators from their continuum counterparts (Lemma \ref{lem:InterpolatedDerivativesAreCloseToRealDerivatives}). Moreover, the regularity of $u$ will control the regularity of discrete gradient and Hessian themselves, in a manner which is independent of the mesh size (Proposition \ref{prop: discrete derivatives are controlled by continuous derivatives} and \ref{prop:InterpolatedDerivativesAreRegular}).

\begin{rem} Before proceeding further, it is worthwhile to note that the discrete gradient and Hessian defined below are standard, and that this appendix has been made with the chief purpose of making the paper as self contained as possible. In fact, as with the discussion of Whitney extension, we failed to find a direct reference where the discretization of the gradient and Hessian is done in the context of a Riemannian manifold. Furthermore, for the purposes of this paper, we only need rather minimal properties of our discretization --essentially, their ``consistency''. As such, the arguments and estimates here are far less optimal than what may be found in the numerical analysis literature where subtler issues are considered.\end{rem}

As we can only use the values of $u$ at points of $\tilde G_n$, our first order of business is to single out admissible directions at $x\in \tilde G_n$ along which a (discrete) derivative may be computed. This is done in the following proposition.

\begin{prop}\label{prop:Appendix Admissible Directions}
  Given $x \in \tilde G_n$  	there are vectors
  \begin{align*} 
    V_{n,1}(x),\ldots,V_{n,d}(x) \in (TM)_x. 
  \end{align*}
  Satisfying the following properties,
  \begin{enumerate}
    \item For each $k$,
    \begin{align*}
      \exp_{x}(V_{n,k}(x)) \in \tilde G_n.		
    \end{align*}		
    \item Also for each $k$,
    \begin{align*}
       98 \tilde h_n \leq |V_{n,k}(x)|_{g_x} \leq 102 \tilde h_n.
    \end{align*}	
			  
    \item Finally, the family $\{V_{n,k}\}_{k=1}^{d}$ forms a basis which is ``almost orthogonal''. To be concrete, for sufficiently large $n$, we have
    \begin{align*}
      |(\hat V_{n,l}(x),\hat V_{n,k}(x))_{g_x}|\leq \frac{1}{20},\;\;\textnormal{ if } k\neq l.
    \end{align*}
    Here, $\hat V$ denotes the unit vector in the direction of $V$, that is $\hat V:= V/|V|_{g_x}$. 	
		
  \end{enumerate}
\end{prop}

\begin{proof}
  Let us recall the constant $\delta\in(0,1)$ introduced in Remark \ref{rem: good balls cover M}, as well as $\tilde h_n$ (see  \eqref{eqFinDimApp: tilde hn def}) which was given by
  \begin{align*}
    \tilde h_n := \sup \limits_{x\in M} d(x,\tilde G_n),\;\;\forall\;n,
  \end{align*}
  and which is such that $\lim\limits_n \tilde h_n = 0$. Next, recall that by \eqref{eqFinDimApp: tilde hn def}, we have
  \begin{align*}
    500 \tilde h_n < \delta.
  \end{align*}
  Fix $x \in \tilde G_n$ and let $e_1,\ldots,e_d$ be an arbitrary orthonormal basis of $(TM)_x$. By definition of $\tilde h_n$, 
  \begin{align*}
    d(\exp_{x}(100\tilde h_ne_k),\tilde G_n)\leq \tilde h_n,\;\;k=1,\ldots,d.
  \end{align*}
  In particular, for each $x$ and each $k$, it is possible to pick a point $x_k$ such that
  \begin{align*}
    x_k \in \tilde G_n \textnormal{ and } d(\exp_{x}(100\tilde h_ne_k),x_k)\leq \tilde h_n.    	  
  \end{align*}
  Having made such a selection for each $x \in \tilde G_n$, we define
  \begin{align*}
    V_{n,k}(x) := (\exp_{x})^{-1}(x_k),\;\;\;k=1,\ldots,d.	  
  \end{align*}	  	  
  Thus, the first property holds by construction. Next, observe that since $100 \tilde h_n<\delta$, both $x_k,x$ and $\exp_{x}(100 \tilde h_n e_k)$ all lie in a ball of radius $4\delta\sqrt{d}$. Therefore, using Remark \ref{rem: good balls cover M} we can compare $|V_{n,k}(x)|_{g_x}$ and $|100\tilde h_n e_k|_{g_x}$.   
  In particular, we have
  \begin{align}\label{eqAppendix: difference between Vnk and 100ek}
    |V_{n,k}-100 \tilde h_n e_k|_{g_x} \leq \tfrac{101}{100} d(\exp_{x}(100\tilde h_ne_k),x_k) \leq \tfrac{101}{100} \tilde h_n.  
  \end{align}	
  Then, the triangle inequality yields,
  \begin{align*}
    |V_{n,k}(x)|_{g_x} & \leq |100\tilde h_n e_k|_{g_x} + |V_{n,k}(x)-100\tilde h_n e_k|_{g_x} \leq 100 \tilde h_n + \tfrac{101}{100} \tilde h_n, \leq 102 \tilde h_n,\\	  
    |V_{n,k}(x)|_{g_x} & \geq |100\tilde h_n e_k|_{g_x} -|V_{n,k}(x)-100\tilde h_n e_k|_{g_x} \geq 100 \tilde h_n - \tfrac{101}{100} \tilde h_n \geq 98 \tilde h_n.
  \end{align*}
  This proves the second property. It remains to prove the third one. For the sake of brevity, let us omit the $x$ dependence in the computations below. 
  
  Let us express the inner product $(V_{n,l},V_{n,k})_{g_x}$ in terms of the orthonormal basis $e_k$,
  \begin{align*}
    (V_{n,l},V_{n,k})_{g_x} & = (V_{n,l}-100\tilde h_ne_l+100\tilde h_ne_l,V_{n,k}-100\tilde h_ne_k+100\tilde h_ne_k)_{g_x}\\
	  & = (V_{n,l}-100\tilde h_ne_l,V_{n,k}-100\tilde h_ne_k+100\tilde h_ne_k)_{g_x}\\
	  & \;\;\;\;+(100\tilde h_ne_l,V_{n,k}-100\tilde h_ne_k+100\tilde h_ne_k)_{g_x}\\
	  & = (V_{n,l}-100\tilde h_ne_l,V_{n,k}-100\tilde h_ne_k)_{g_x}+(V_{n,l}-100\tilde h_ne_l,100\tilde h_ne_k)_{g_x}\\
	  & \;\;\;\;+(100\tilde h_ne_l,V_{n,k}-100\tilde h_ne_k)_{g_x}+(100\tilde h_ne_l,100\tilde h_ne_k)_{g_x}.	  
  \end{align*} 
  Since the $e_k$ are orthonormal, for $k\neq l$ it follows that 
  \begin{align*}
    (V_{n,l},V_{n,k})_{g_x} & = (V_{n,l}-100\tilde h_ne_l,V_{n,k}-100\tilde h_ne_k)_{g_x}+(V_{n,l}-100\tilde h_ne_l,100\tilde h_ne_k)_{g_x}\\
	  & \;\;\;\;+(100\tilde h_ne_l,V_{n,k}-100\tilde h_ne_k)_{g_x},\;\; k\neq l.
  \end{align*}    
  We apply the estimate \eqref{eqAppendix: difference between Vnk and 100ek} to this last identity, it follows that
  \begin{align*}
    |(V_{n,l},V_{n,k})_{g_x}| & \leq | V_{n,l}-100\tilde h_ne_l |_{g_x} |V_{n,k}-100\tilde h_ne_k|{g_x}+|V_{n,l}-100\tilde h_ne_l|_{g_x} |100\tilde h_ne_k|_{g_x}\\
	  & \;\;\;\;+|100\tilde h_ne_l|_{g_x}|V_{n,k}-100\tilde h_ne_k|_{g_x} \\
	  & \leq (\tfrac{101}{100})^2\tilde h_n^2+ 2(\tfrac{101}{100}\tilde h_n)(100 \tilde h_n) \leq 204\tilde h_n.
  \end{align*}  
  Since $|V_{n,l}|_{g_x}^{-1}\geq 98 \tilde h_n$, it follows that
  \begin{align*}
    |(\hat V_{n,l},\hat V_{n,k})_{g_x}| & \leq 204 \tilde h_n^2  |V_{n,l}|_{g_x}^{-1}|V_{n,k}|_{g_x}^{-1} \leq 204 (98)^{-2}\leq \tfrac{1}{20},
  \end{align*}
  and the third property is proved. 
\end{proof}

From here on, for each $n$ and for every $x\in \tilde G_n$, we fix a selection of vectors $\{V_{n,1}(x),\ldots,V_{n,d}(x)\} \in (TM)_x$ as in the previous proposition. Moreover, we fix $u \in C^\beta_b(M)$ for the rest of this section.
\begin{DEF}\label{def:Appendix Discrete Gradient} (Discrete gradient)
  Given $x\in \tilde G_n$ and $u$, define $(\nabla_n)^1u(x) \in (TM)_x$ by solving the system of linear equations
  \begin{align*}
     (V_{n,k}(x),(\nabla_n)^1u(x))_{g_x} = u(\exp_{x}(V_{n,k}(x)))-u(x),\;\;\;\;k=1,\ldots,d.       
  \end{align*}
  Note that, as the $V_{n,k}(x)$ are linearly independent, the above system always has a unique solution.
\end{DEF}

\begin{rem}\label{remAppendix: Discrete Gradient Eucliean Case}
  Let us illustrate the above definition in a simple case. Let us take,
  \begin{align*}
    M=\mathbb{R}^d,\;\;\; \tilde G_n = (2^{-n}\mathbb{Z}^d),
  \end{align*}	
  and write $\tilde h_n = 2^{-n}$ and $V_{n,l}(x) = h_ne_l$, where $\{e_1,\ldots,e_d\}$ denote the standard orthonormal basis of $\mathbb{R}^d$. Then,
  \begin{align*}
    u(x+\tilde h_ne_k)-u(x) & = u(\exp_{x}(V_{n,k}(x)))-u(x)\\
      & = \sum \limits_{l=1}^{d} (\nabla_n)_l^1u(x) (V_{n,k}(x),\hat V_{n,l}(x))_{g_x}\\
      & = (\nabla_n)^1_ku(x) \tilde h_n.   
  \end{align*}
  Thus, in this case we have
  \begin{align*}
    (\nabla_n)^1_ku(x) = \frac{u(x+\tilde h_ne_k)-u(x)}{\tilde h_n} (\approx \partial_{x_k}u(x)),	  
  \end{align*}	  
  and the vector $(\nabla_n)^1u(x)$ is nothing but a discretization of the gradient.
\end{rem}

\begin{DEF}\label{def:Appendix Parallel Transport}
  Let $x,y \in M$ be such that $d(x,y)<r_0$. Then let $\Gamma_{x,y}$ denote the linear map
  \begin{align*}
    \Gamma_{x,y}: (TM)_y \to (TM)_x,
  \end{align*}
  given by parallel transport along the unique minimal geodesic connecting $x$ to $y$. We should recall this map is an isometry with respect to the inner products $g_x$ and $g_{y}$. If the point $y$ is understood from context, we shall simply write $\Gamma_x$.
  
\end{DEF}

\begin{DEF}  Let $V$ be a section of the tangent bundle $TM$. We say $V$ is of class $C^\alpha$ if
  \begin{align*} 
    [V]_{C^\alpha(M)} :=\sup \limits_{0<d(x,y)<r_0}\frac{|V(x)-\Gamma_{x,y}V(y)|_{g_x}}{d(x,y)^\alpha} < \infty
  \end{align*}
  Likewise, if $M:TM\to TM$, then
  \begin{align*}
    [M]_{C^\alpha(M)} := \sup \limits_{0<d(x,y)<r_0}\frac{|M(x)-M(y)\Gamma_{x,y}^{-1}|_{g_x}}{d(x,y)^\alpha}<\infty.
  \end{align*}
  These seminorms, when applied to $V=\nabla u$ and $M=\nabla^2u$ allows to define the $C^\beta$ norm of $u$ in the obvious manner.
\end{DEF}

\begin{rem}\label{rem: Holder regularity gradient and Hessian} Let $\beta \in [0,3)$ be given. The following is a useful characterization of H\"older continuity that will be used later on.  Let $x(t)$ denote a geodesic and $e(t)$ a parallel vector field along it with $|\dot x(t)|_{g_{x(t)}}=|e(t)|_{x(t)}=1$. Then,
\begin{align*}
  |(\nabla u(x(t)),e(t))_{x(t)}-(\nabla u(x(s)),e(s))_{x(s)}|\leq \|u\|_{C^\beta}|t-s|^{\min\{\beta-1,1\}},\;\;\textnormal{ if } \beta\geq 1,
\end{align*}
and
\begin{align*}
  |(\nabla^2 u(x(t))e(t),e(t))_{x(t)}-(\nabla^2 u(x(s))e(s),e(s))_{x(s)}|\leq \|u\|_{C^\beta}|t-s|^{\min\{\beta-2,1\}},\;\;\textnormal{ if } \beta \geq 2.
\end{align*}

\end{rem}

Defining the discrete Hessian requires further preparation, we define first the following ``second order difference'',
\begin{align*}
  \delta u_{x}(V_1,V_2) := u(\exp_{\exp_x(V_1)}(V_2))-u(\exp_x(V_1))-u(\exp_{x}(\Gamma_{x}V_2))+u(x).
\end{align*}
Here $\Gamma_{x}$ denotes the operation of parallel transport, as introduced in Definition \ref{def:Appendix Parallel Transport}.

\begin{DEF}\label{def:Appendix Discrete Hessian} (Discrete Hessian)
  Given $x\in \tilde G_n$ and $u$, we will define a linear transformation
  \begin{align*}
   (\nabla_n)^2u(x): (TM)_x \to (TM)_x.
  \end{align*}
  Given $k=1,\ldots,d$, define $(\nabla_n)^2u(x)V_{n,k}(x) \in (TM)_x$ as the solution $V$ to the linear system
  \begin{align*}
    (V,\Gamma_{x}V_{n,l}(x_k))_{g_x} = \delta u_{x}(V_{n,k}(x),V_{n,l}(x_k)),\;\;l=1,\ldots,d.
  \end{align*}	       
  Here, for the sake of brevity of notation, we have written
  \begin{align*}
    x_{k} = \exp_{x}(V_{n,k}(x)).	  
  \end{align*}	   
  Having indicated how $(\nabla_n)^2u(x)$ acts on the basis $\{V_{n,k}(x)\}_{k=1}^d$ of $(TM)_x$, the linear transformation is completely determined.
  
\end{DEF}

Let us elaborate on the linear algebra problem that was used to define $(\nabla_n)^2u$. Given a linear transformation $D:(TM)_x\to (TM)_x$, and a family of pairs of vectors $\{(V_k,W_k)\}_{k=1}^N$ for some $N$, we seek to recover the full matrix $D$ from the values
\begin{align*}
  (DV_k,W_k).
\end{align*}
We are given a basis $V_k$ ($k=1,\ldots,d$), and for each $k$ another basis $\{W_{k,l}\}$ ($l=1,\ldots,d$). Then, we seek to completely determine a linear transformation $M$ given the values
\begin{align*}
  (DV_k,W_{k,l}), \;\textnormal{ for } k,l=1,\ldots,d.  
\end{align*}

\begin{rem}
  Let us again see what this definition says in a simple case. Let $M, \tilde G_n, \tilde h_n$ and $\{V_{n,k}(x)\}$ be as in Remark \ref{remAppendix: Discrete Gradient Eucliean Case}. Then, given $x\in \tilde G_n$ and $k,l=1,\ldots,d$ we have
  \begin{align*}
    \delta u_x(V_{n,k}(x),V_{n,l}(x_k)) & = u(x+2^{-n}e_k+2^{-n}e_l)-u(x+2^{-n}e_k)-u(x+2^{-n}e_l)+u(x)\\
      & =  2^{-n}2^{-n}( (\nabla_n)^2u(x)e_k,e_l).
  \end{align*}
  It follows that the components of $(\nabla_n)^2u(x)$ are given by
  \begin{align*}
    (\nabla_n)_{kl}^2u(x) = \frac{u(x+2^{-n}e_k+2^{-n}e_l)-u(x+2^{-n}e_k)-u(x+2^{-n}e_l)+u(x)}{2^{-n}2^{-n}} (\approx \nabla^2_{kl}u(x) )  
  \end{align*}
  and the matrix $(\nabla_n)_{kl}^2u(x)$ is nothing but a discretization of the standard Hessian.

\end{rem}

\begin{rem}\label{rem:Appendix Locality of Discrete Derivatives}
  Let $x \in \tilde G_n$. Using the upper bound in part (2) of Proposition \ref{prop:Appendix Admissible Directions}, one notes that all the values of $u$ taken in evaluating $\nabla_n^1 u(x)$ and $\nabla_n^2u(x)$ lie within a ball of radius $<250\tilde h_n$ centered at $x$. In particular,  if $u\equiv 0$ in $B_{250h_n}(x)$, then
  \begin{align*}
    \nabla_n^1u(x) = 0,\;\;\nabla_n^2u(x) = 0.	  	   
  \end{align*}	  

\end{rem}  
  
The previous remark guarantees that the extension operator is somewhat ``local'', the locality becoming more and more exact as $n$ becomes larger, this is made rigorous in the following proposition. 
\begin{prop}\label{prop:Appendix Estimate on Locality of the Extension}
  Let $u \in C^\beta$, and $x_0 \in M$. Then,
  \begin{align*}
    u \equiv 0 \textnormal{ in } B_{400\tilde h_n}(x_0) \Rightarrow E_n^\beta (u,\cdot) \equiv 0 \textnormal{ in } B_{100\tilde h_n}(x_0). 
  \end{align*}
  
\end{prop}

\begin{proof}  
  First, we claim that
  \begin{align}
    x \in B_{100\tilde h_n}(x_0) \Rightarrow B_{250 \tilde h_n}(\hat y_{n,k}) \subset B_{400 \tilde h_n}(x_0),\;\;\forall\; k\in K_x. \label{eqn:Appendix Balls centered in hat Ynk lie close together}
  \end{align}
  Let us see how \eqref{eqn:Appendix Balls centered in hat Ynk lie close together} implies the proposition. Fix $x \in B_{100 \tilde h_n}(x_0)$, with $x\in M\setminus \tilde G_n$, then
  \begin{align*}
    E_n^\beta (u,x) & = \sum\limits_{k} p_{(u,k)}^\beta(x)\phi_{n,k}(x) = \sum\limits_{k\in K_x} p_{(u,k)}^\beta(x)\phi_{n,k}(x).
  \end{align*}
  Then, thanks to \eqref{eqn:Appendix Balls centered in hat Ynk lie close together}, we have that 
  \begin{align*}
    u \equiv 0 \textnormal{ in } B_{250 \tilde h_n}(\hat y_{n,k}),\;\;\forall\;k\in K_x,\;\forall\; x\in B_{100\tilde h_n}(x_0).
  \end{align*}
  In this case, Remark \ref{rem:Appendix Locality of Discrete Derivatives} guarantees that
  \begin{align*}
    p_{(u,k)}^\beta(x)\equiv 0,\;\;\forall\;k\in K_x,\;\forall\;x\in B_{100\tilde h_n}(x_0).
  \end{align*}
  In other words, 
  \begin{align*}
    E_n^\beta (u,x) =0,\;\;\forall\;x\in B_{100\tilde h_n}(x_0).
  \end{align*}	
  Which proves the proposition. It remains to prove \eqref{eqn:Appendix Balls centered in hat Ynk lie close together}. Fix $x \in B_{100 \tilde h_n}(x_0)$ and $k\in K_x$.   By the triangle inequality, and the definition of $\hat y_{n,k}$, we have
  \begin{align*}
    d(x,\hat y_{n,k}) \leq d(x,y_{n,k})+d(\hat y_{n,k},y_{n,k}) & = d(x,y_{n,k})+d(y_{n,k},\tilde G_n)\\
	  & \leq 2d(x,y_{n,k})+d(x,\tilde G_n)\\
	  & \leq 2\diam(P_{n,k}^*)+d(x,\tilde G_n).
  \end{align*}
  Then, thanks to Remark \ref{rem: Whitney local diameter of balls}, 
  \begin{align*}
    d(x,\hat y_{n,k}) \leq 15 d(x,\tilde G_n)\leq 15 \tilde h_n,\;\;\forall\;k\in K_x.	  
  \end{align*}	  
  Furthermore,
  \begin{align*}
    d(\hat y_{n,k},x_0) & \leq d(\hat y_{n,k},x)+d(x,x_0)\\
	  & \leq d(\hat y_{n,k},\hat x)+d(x,\hat x)+d(x,x_0).
  \end{align*}
  We now recall that $d(x,\hat x) = d(x,\tilde G_n) \leq \tilde h_n$, and $d(x,x_0)\leq 100\tilde h_n$. Furthermore, as shown in \eqref{eqFinDim:hat x and hat Ynk distance is bounded by distance to tilde Gn} in the proof of Proposition  we have $d(\hat y_{n,k},\hat x) \leq 16 d(x,\tilde G_n)$ for $k\in K_x$. Gathering these inequalities it follows that
  \begin{align*}
    d(\hat y_{n,k},x_0)  \leq 117 \tilde h_n,\;\;\forall\;k\in K_x.
  \end{align*}   
  From here, and the triangle inequality, we conclude that $B_{250 \tilde h_n}(\hat y_{n,k})$ lies inside $B_{400\tilde h_n}(x_0)$, that is, \eqref{eqn:Appendix Balls centered in hat Ynk lie close together}. This proves the proposition.
\end{proof}

In what follows, we will be using the functions $l$ and $q$, introduced in Definition \ref{def:LinearAndQuadPolynomials}. In $\mathbb{R}^d$ this is a completely straightforward calculation using the Taylor polynomial.  On a Riemannian manifold, we shall use the coordinates given by the exponential map. For the next proposition, we recall that the functions ``linear'' and ``quadratic'' functions $l$ and $q$ introduced in Definition \ref{def:LinearAndQuadPolynomials} are defined in a ball of of radius $4\delta \sqrt{d}$ around their base point, where $\delta$ is as in Remark \ref{rem: good balls cover M}

\begin{prop}\label{prop:Appendix Cbeta implies bound on Taylor remainder}
  Let $x_0,x \in M$ with $d(x,x_0)\leq 4\delta\sqrt{d}$, and $u \in C^\beta_b(M)$. Then,
  
  \noindent 1) If $C^{\beta}_b=C^1_b$, then
  \begin{align*}	
    u(x)-u(x_0)-l(\nabla u(x_0),x_0;x) = o(d(x,x_0)),
  \end{align*}

  where the $o(d(x,x_0))$ term is controlled by the modulus of continuity of $\nabla u$.
  
  \noindent 2) If $\beta \in [1,2]$, then
  \begin{align*}	
    |u(x)-u(x_0)-l(\nabla u(x_0),x_0;x)|\leq \|u\|_{C^\beta}d(x,x_0)^\beta.
  \end{align*}	  
  \noindent 3) If $C^\beta_b = C^2_b$, then
  \begin{align*}
    u(x)-u(x_0)-l(\nabla u(x_0),x_0;x)-q(\nabla^2u(x_0),x_0;x) = o(d(x,x_0)^2),
  \end{align*}

   where the $o(d(x,x_0))$ term is controlled by the modulus of continuity of $\nabla^2 u$.
 
  \noindent 4) If $\beta \in [2,3]$, then
  \begin{align*}	
    |u(x)-u(x_0)-l(\nabla u(x_0),x_0;x)-q(\nabla^2u(x_0),x_0;x)|\leq \|u\|_{C^\beta}d(x,x_0)^\beta.
  \end{align*}	  
\end{prop}

We omit the straightforward proof of Proposition \ref{prop:Appendix Cbeta implies bound on Taylor remainder}.

\begin{rem}\label{rem:ExpansionInEXPEst}
  From Definition \ref{def:LinearAndQuadPolynomials} it is immediate that Proposition \ref{prop:Appendix Cbeta implies bound on Taylor remainder} has the following equivalent formulation which will also be useful: given a unit vector $e \in (TM)_{x_0}$ and $h\leq 4\delta \sqrt{d}$, we have
  \begin{align*}
    & u(\exp_{x_0}(he))-u(x_0)-h(\nabla u(x_0),e)_{g_{x_0}} = o(h), \textnormal{ if } C^\beta_b = C^1_b,\\	  
    & |u(\exp_{x_0}(he))-u(x_0)-h(\nabla u(x_0),e)_{g_{x_0}}|\leq \|u\|_{C^\beta}h^{\beta}, \textnormal{ if } \beta \in (1,2],\\
    & u(\exp_{x_0}(he))-u(x_0)-h(\nabla u(x_0),e)_{g_{x_0}}-\frac{h^2}{2}(\nabla^2 u(x_0)e,e)_{g_{x_0}} = o(h^2), \textnormal{ if } C^\beta_b = C^2_b,\\	
    & |u(\exp_{x_0}(he))-u(x_0)-h(\nabla u(x_0),e)_{g_{x_0}}-\frac{h^2}{2}(\nabla^2 u(x_0)e,e)_{g_{x_0}}|\leq \|u\|_{C^\beta}h^{\beta}, \textnormal{ if } \beta \in [2,3].	
  \end{align*}
\end{rem}
 
\begin{proof}[Proof of Remark \ref{rem:ExpansionInEXPEst}]

  \emph{First estimate.} Fix a unit vector $e\in (TM)_{x_0}$. For $h \in [0,r_0]$ let $x(h) := \exp_{x_0}(he)$, and let 
  \begin{align*}
    \varepsilon(h) := u\left (x(h)\right )-u(x_0)-(\nabla u(x_0),he)_{g_{x_0}}.
  \end{align*}
  It is immediate that $\varepsilon(0)=0$, $d(x_0,x(h))=h$, and that
  \begin{align*}
    \varepsilon'(h) & = (\nabla u(x(h)),\dot x(h))_{g_{x(h)}}-(\nabla u(x_0),e)_{g_{x_0}}.
  \end{align*}
  Since $\dot x(0)= e$, we have $\varepsilon'(0) = 0$. Keeping in mind that $\dot x(h)$ is the parallel transport of $e$ along $x(h)$, the H\"older regularity of $\nabla u(x)$ yields
  \begin{align*}
    |\varepsilon'(h)| = |(\nabla u(x(h)),\dot x(h))_{x(h)}-(\nabla u(x_0),e)_{g_{x_0}}| & \leq \|u\|_{C^\beta}d(x_0,x(h))^{\beta-1}\\
	& = \|u\|_{C^\beta}h^{\beta-1}.
  \end{align*}
  Integrating this last inequality from $0$ to $h$, we obtain the first estimate, since
  \begin{align*}
    |\varepsilon(h)| = |\varepsilon(h)-\varepsilon(0)| \leq \|u\|_{C^\beta}h^{\beta}.
  \end{align*}
  \emph{Second estimate.} Let $x(h)$ be as before, with $h \in [0,r_0]$. This time we consider the function 
  \begin{align*}
    \varepsilon(h) := u\left (x(h) \right )-u(x_0)-h(\nabla u(x_0),e)_{g_{x_0}}-\frac{h^2}{2}( (\nabla^2 u(x_0))e,e)_{g_{x_0}}.
  \end{align*}
  Then, as before it is clear that $\varepsilon(0)=\varepsilon'(0)=0$ and
  \begin{align*}
    & \varepsilon'(h) = (\nabla u(x(h)),\dot x(h))_{x(h)}-(\nabla u(x_0),e)_{g_{x_0}}-h( (\nabla^2 u(x_0))e,e)_{g_{x_0}},\\
    & \varepsilon''(h) = ( (\nabla^2 u(x(h)))\dot x(h),\dot x(h))_{g_{x(h)}}-( (\nabla^2 u(x_0))e,e)_{g_{x_0}}.
  \end{align*}
  As before, we make use of the fact that $\dot x(h)$ is a parallel vector along $x(h)$, which leads to
  \begin{align*}
    |\varepsilon''(h)| = |  ( (\nabla^2 u(x(h)))\dot x(h),\dot x(h))_{g_{x(h)}}-( (\nabla^2 u(x_0))e,e)_{g_{x_0}}| \leq \|u\|_{C^\beta}h^{\beta-2}.	  
  \end{align*}	  
  Integrating this inequality twice (and using that $\varepsilon(0)=\varepsilon'(0)=0$) it follows that
  \begin{align*}
    |\varepsilon(h)| = |\varepsilon(h)-\varepsilon(0)| = \left | \int_0^h \varepsilon'(s)\;ds \right | \leq \|u\|_{C^\beta}h^\beta,
  \end{align*}	  
  which proves the second estimate.
  
\end{proof}

The next Lemma consists of a very important fact, namely, that the discrete difference operators $(\nabla_n)^1u$ and $(\nabla_n)^2u$ are ``consistent'' --i.e. they converge to the differential operators $\nabla u$ and $\nabla^2 u$. Furthermore, we have that the error made when estimating the derivatives by the discrete operator is a quantity controlled by the $C^\beta$ norm of  $u \in C^\beta_b(M)$.

\begin{lem}\label{lem:InterpolatedDerivativesAreCloseToRealDerivatives}
  Let $x \in \tilde G_n$ and $u\in C^\beta_b(M)$ then
  \begin{align*}
    |(\nabla_n)^1u(x)-\nabla u(x)|_{g_{x}} \leq C\|u\|_{C^\beta}\tilde h_n^{\beta-1}, \textnormal{ if } \beta \in (1,2],\\
    |(\nabla_n)^2 u(x)-\nabla^2 u(x)|_{g_{x}} \leq C\|u\|_{C^\beta}\tilde h_n^{\beta-2}, \textnormal{ if } \beta \in (2,3].	 	  
  \end{align*}	  
  Furthermore, if $C^\beta_b = C^1_b$ or $C^\beta_b = C^2_b$ then, we have, respectively
  \begin{align*}
    \lim \limits_{n\to \infty}\sup \limits_{x \in K\cap \tilde G_n} |(\nabla_n)^1u(x)-\nabla u(x)|_{g_{x}} = 0,\\
    \lim \limits_{n\to \infty}\sup \limits_{x \in K\cap \tilde G_n} |(\nabla_n)^2u(x)-\nabla^2 u(x)|_{g_{x}} = 0,	
  \end{align*}	  
  where $K$ is an arbitrary compact subset of $M$.
\end{lem}

\begin{proof}
  \emph{First estimate.} We may write
  \begin{align*}	
    \nabla u(x) = \sum \limits_{l=1}^d  \theta_l \hat V_{n,l}(x),
  \end{align*}	 
  where the numbers $\theta_1,\ldots,\theta_d$ are determined from the system of equations
  \begin{align*}
    (\nabla u(x),\hat V_{n,k}(x))_{g_x} = \sum\limits_{l=d}^d \theta_l (\hat V_{n,k}(x),\hat V_{n,l}(x))_{g_x}.
  \end{align*}
  Now, Proposition \ref{prop:Appendix Cbeta implies bound on Taylor remainder} says that
  \begin{align*}
    \left | \frac{u(\exp_{x}(V_{n,k}))-u(x)}{|V_{n,k}|_{g_x}} - \left (\nabla u(x),\hat V_{n,k} \right )_{g_x}\right |\leq C\|u\|_{C^\beta}|V_{n,k}|_{g_x}^{\beta-1},
  \end{align*}
  and, if $C^\beta_b = C^1_b$, it says that for any compact $K$,
  \begin{align*}
    \lim \limits_{n\to \infty} \sup \limits_{x\in K \cap \tilde G_n} \max \limits_{1\leq k\leq d}\left | \frac{u(\exp_{x}(V_{n,k}))-u(x)}{|V_{n,k}|_{g_x}} - \left (\nabla u(x),\hat V_{n,k} \right )_{g_x}\right | = 0,
  \end{align*}  
  the convergence in the limit being determined by  $K$, the continuity of $\nabla u$, and $M$. Then, 
  \begin{align*}
    |(\nabla_n)_{l}^1u(x)-\theta_l| & \leq C\|u\|_{C^\beta}\tilde h_n^{\beta-1} \;\;\; \forall\;x\in M,\textnormal{ if } \beta \in (0,1),\\
    \lim \limits_{n\to \infty}\sup \limits_{x \in K \cap \tilde G_n}|(\nabla_n)_{l}^1u(x)-\theta_l| & = 0,   \quad\quad\quad\quad\quad\ \forall\; K\subset\subset M, \textnormal{ if } C^\beta_b = C^1_b.
  \end{align*}
  The above holds for each $l=1,\ldots,d$. Combining these inequalities it is immediate that
  \begin{align*}
    |(\nabla_n)^1u(x)-\nabla u(x)|\leq C\|u\|_{C^\beta}\tilde h_n^{\beta-1},
  \end{align*}
  and, for $C^\beta_b = C^1_b$,
  \begin{align*}
    \lim \limits_{n\to \infty}\sup \limits_{x \in K\cap \tilde G_n}|(\nabla_n)^1u(x)-\nabla u(x)|  = 0.
  \end{align*}
     
  \noindent \emph{Second estimate.} First, we need an elementary observation about geodesics. Observe that
  \begin{align}\label{eqAppendix:Exponential Map Error Term 1}
    \exp_{\exp_x(V_{n,k}(x))}(V_{n,l}(x_k)) = \exp_{x}(V_{n,k}(x)+\Gamma_xV_{n,l}(x_k)+ \textnormal{(Error)}_0).
  \end{align}
  Where the term $\textnormal{(Error)}_0$ is term appearing due to possibly non-zero curvature. It turns out that this error term is at least a cubic error in terms of $\tilde h_n$, which is proved as follows: let $J(t)$ be the Jacobi field along the geodesic $\gamma(t) = \exp_x(t \hat V_{n,k})$ determined by $J(0)=0$ and $J(|V_{n,k}|_{g_x}) = \hat V_{n,l}(x_k)$. Then,  define $\sigma(t,s) \in (TM)_x $ by
  \begin{align*}
    \exp_{\gamma(t)}(s J(t) )  = \exp_{x}( \sigma(t,s)).
  \end{align*}
  Note that $\sigma\left ( |V_{n,k}|_{g_x},|V_{n,l}(x_k)|_{g_x} \right )$ must be equal to the argument in the exponential on the right hand side of \eqref{eqAppendix:Exponential Map Error Term 1}. Then, note that
  \begin{align*}
    \sigma(0,s) = 0,\;\;\forall\;s \Rightarrow \sigma(0,0)= \partial_{s}\sigma(0,0) =\partial_{ss}\sigma(0,0)=0.
  \end{align*}
  Furthermore, $\partial_t \sigma(0,0) = \hat V_{n,k}(x)$, so
  \begin{align*}
    \sigma(t,s) = t\hat V_{n,k}+st \partial_{ts}\sigma(0,0) + O((s^2+t^2)^{3/2}).
  \end{align*}
  Now, by contrasting the respective Jacobi and parallel transport equations, it can be shown that
  \begin{align*}
    |\partial_{ts}\sigma(0,0)-\Gamma_x \hat V_{n,k}(x_k)|\leq C\tilde h_n.
  \end{align*}
  Given that $|V_{n,k}|_{g_x},|V_{n,l}(x_k)|_{g_x}\leq h_n$, this leads to the bound
  \begin{align}\label{eqAppendix:Exponential Map Error Term 2}
    | \textnormal{(Error)}_0|_{g_x} \leq C\tilde h_n^3.
  \end{align}	  
  The constant $C$ depending only on the metric of $M$.

  Let us analyze the first three terms appearing in the second order difference $\delta u_{x}(V_{n,k}(x),V_{n,l}(x_k))$.  We consider the Taylor expansion and estimate the remainder via Proposition \ref{prop:Appendix Cbeta implies bound on Taylor remainder}. First of all, we have
  \begin{align*} 
    u(\exp_{\exp_x(V_{n,k}(x))}(V_{n,l}(x_k))) = u(\exp_{x}(V_{n,k}(x)+\Gamma_xV_{n,l}(x_k)+ \textnormal{(Error)}_0)).
  \end{align*}
  The estimate \eqref{eqAppendix:Exponential Map Error Term 2} guarantees in particular that $|V_{n,k}(x)+\Gamma_xV_{n,l}(x_k)+ \textnormal{(Error)}_0|\leq C\tilde h_n$. With this in mind, we apply Proposition \ref{prop:Appendix Cbeta implies bound on Taylor remainder} in order to obtain the expansion
  \begin{align*} 
    & u(\exp_{\exp_x(V_{n,k}(x))}(V_{n,l}(x_k)))\\
    & = u(x)+ (\nabla u(x),V_{n,k}(x)+\Gamma_xV_{n,l}(x_k)+ \textnormal{(Error)}_0)_{g_x}\\
    & \;\;\;\; + \tfrac{1}{2} \left (\nabla^2u(x)( V_{n,k}(x)+\Gamma_xV_{n,l}(x_k)+ \textnormal{(Error)}_0) , V_{n,k}(x)+\Gamma_xV_{n,l}(x_k)+ \textnormal{(Error)}_0 \right )_{g_x}\\
    & + \textnormal{(Error)},	
  \end{align*}
  where $\textnormal{(Error)}$, which denotes the remainder in the Taylor expansion, satisfies the bound 
  \begin{align*}
    | \textnormal{(Error)}| \leq C\|u\|_{C^\beta}\tilde h_n^\beta.  
  \end{align*}	  
 Expanding, we see that
  \begin{align*} 
    & u(\exp_{\exp_x(V_{n,k}(x))}(V_{n,l}(x_k)))\\
    & = u(x)+ (\nabla u(x),V_{n,k}(x))_{g_x}+(\nabla u(x),\Gamma_xV_{n,l}(x_k))_{g_x}\\
    & \;\;\;\; + \tfrac{1}{2} \left (\nabla^2u(x) V_{n,k}(x),V_{n,k}(x)\right )_{g_x}+\tfrac{1}{2} \left (\nabla^2u(x)\Gamma_xV_{n,l}(x_k),\Gamma_xV_{n,l}(x_k)\right )_{g_x}\\
    & \;\;\;\; + \left (\nabla^2u(x) V_{n,k}(x),\Gamma_xV_{n,l}(x_k)\right )_{g_x}\\ 	
    & \;\;\;\;+(\nabla u(x),\textnormal{(Error)}_0)_{g_x}+\frac{1}{2}\left (\nabla^2u(x)\textnormal{(Error)}_0,\textnormal{(Error)}_0\right )_{g_x}\\
    & \;\;\;\;+\left (\nabla^2u(x)\textnormal{(Error)}_0,V_{n,k}(x)+\Gamma_xV_{n,l}(x_k)\right )_{g_x}+ \textnormal{(Error)}.	
  \end{align*} 
  The terms involving a factor of $\textnormal{(Error)}_0$ may be absorbed into $\textnormal{(Error)}$. To see why, we use the estimate \eqref{eqAppendix:Exponential Map Error Term 2} and bound term by term
  \begin{align*}  
    |(\nabla u(x),\textnormal{(Error)}_0)_{g_x}| & \leq C\|u\|_{C^1}h_n^3,\\
    \left | \left (\nabla^2u(x)\textnormal{(Error)}_0,\textnormal{(Error)}_0\right )_{g_x}\right | & \leq C\|u\|_{C^2}\tilde h_n^{6},\\
    \left | \left (\nabla^2u(x)\textnormal{(Error)}_0,V_{n,k}(x)\right )_{g_x}\right | & \leq C\|u\|_{C^2}\tilde h_n^4,\\
    \left | \left (\nabla^2u(x)\textnormal{(Error)}_0,\Gamma_xV_{n,l}(x_k)\right )_{g_x}\right | & \leq C\|u\|_{C^2}\tilde h_n^4.	
  \end{align*}
  Since $\beta \geq 2$, each of the above terms is bounded by $C\|u\|_{C^\beta}\tilde h_n^\beta$. Then, absorbing the terms involving $\textnormal{(Error)}_0$ into $\textnormal{(Error)}$ we obtain 
 \begin{align*}
    & u(\exp_{\exp_x(V_{n,k}(x))}(V_{n,l}(x_k)))\\
    & = u(x)+ (\nabla u(x),V_{n,k}(x))_{g_x}+(\nabla u(x),\Gamma_xV_{n,l}(x_k))_{g_x}\\
    & \;\;\;\; + \tfrac{1}{2} \left (\nabla^2u(x) V_{n,k}(x),V_{n,k}(x)\right )_{g_x}+\tfrac{1}{2} \left (\nabla^2u(x)\Gamma_xV_{n,l}(x_k),\Gamma_xV_{n,l}(x_k)\right )_{g_x}\\
    & \;\;\;\; + \left (\nabla^2u(x) V_{n,k}(x),\Gamma_xV_{n,l}(x_k)\right )_{g_x}+ \textnormal{(Error)}.	
 \end{align*}
  As for the other two terms, we have
  \begin{align*}  	
    u(\exp_x(V_{n,k}(x))) & = u(x)+(\nabla u(x),V_{n,k}(x))_{g_x}\\
	& \;\;\;\;+\tfrac{1}{2} \left( \nabla^2u(x)V_{n,k}(x),V_{n,k}(x) \right )_{g_x}+\textnormal{(Error)},
  \end{align*}
  and
  \begin{align*}	
    u(\exp_{x}(\Gamma_{x}V_{n,l}(x_k))) & = u(x)+(\nabla u(x_0),\Gamma_x V_{n,l}(x_k))_{g_x}\\
	& \;\;\;\;+\tfrac{1}{2} \left( \nabla^2u(x)\Gamma_x V_{n,l}(x_k),\Gamma_x V_{n,l}(x_k) \right )_{g_x}+\textnormal{(Error)}.
  \end{align*}
  In each case, $|\textnormal{(Error)}|$ is no larger than $C\|u\|_{C^\beta}\tilde h_n^\beta$, thanks to Proposition \ref{prop:Appendix Cbeta implies bound on Taylor remainder}. 
  
  Combining the last three formulas, it follows that
  \begin{align*}
    \delta u_{x}(V_{n,k}(x),V_{n,l}(x_k))
  \end{align*}
  is equal to
  \begin{align*}
    & u(x)+(\nabla u(x),V_{n,k}(x))_{g_x}+(\nabla u(x),\Gamma_x(V_{n,l}(x_k))_{g_x}+\tfrac{1}{2}((\nabla^2 u(x))V_{n,k}(x),V_{n,k}(x))_{g_x}\\
    & +((\nabla^2u(x))V_{n,k}(x),\Gamma_xV_{n,l}(x_k))_{g_x}+\tfrac{1}{2}((\nabla^2u(x))\Gamma_xV_{n,l}(x_k),\Gamma_xV_{n,l}(x_k))_{g_x}\\
    & -u(x)-(\nabla u(x),V_{n,k}(x))_{g_x}-\tfrac{1}{2}((\nabla^2u(x))V_{n,k}(x),V_{n,k}(x))_{g_x}-u(x)-(\nabla u(x),\Gamma_x(V_{n,l}(x_k))_{g_x}\\
    &-\tfrac{1}{2}((\nabla^2 u(x))\Gamma_x(V_{n,l}(x_k),\Gamma_x(V_{n,l}(x_k))_{g_x}+u(x)+\textnormal{(Error)}.
  \end{align*}
  From the above, it is clear all but one of the terms in the first two lines above is cancelled out with a term in the last two lines. We then arrive at the formula
  \begin{align*}
    \delta u_{x}(V_{n,k}(x),V_{n,l}(x_k)) =  ((\nabla^2u(x))V_{n,k}(x),\Gamma_xV_{n,l}(x_k))_{g_x}+\textnormal{(Error)},
  \end{align*}
  where --thanks to Proposition \ref{prop:Appendix Cbeta implies bound on Taylor remainder}, as pointed out earlier-- we have
  \begin{align*}
    |\textnormal{(Error)}| \leq C\|u\|_{C^\beta}\tilde h_n^\beta. 
  \end{align*}  
  Then, solving the linear problem corresponding to $(\nabla_n)^2u(x)$ and $\nabla^2u(x)$ it follows that
  \begin{align*}
    |(\nabla_n)^2u(x)-\nabla^2 u(x)|\leq C\|u\|_{C^\beta}\tilde h_n^{\beta-2}.
  \end{align*}
  Finally, if $C^\beta_b = C^2_b$, the convergence of $(\nabla_n)^2u(x)$ to $\nabla^2u(x)$ follows analogously to the convergence of $\nabla_n u(x)$ to $\nabla u(x)$ for $C^\beta_b = C^1_b$,  we omit the details.
\end{proof}

Given the proof of Lemma \ref{lem:InterpolatedDerivativesAreCloseToRealDerivatives} it should be clear that the $L^\infty(\tilde G_n)$ norm of $(\nabla_n)^iu$ ($i=1,2$) is controlled by the appropriate $C^\beta$ norm of $u$ in a manner which is independent of $n$. This fact is the content of the next proposition.
\begin{prop}\label{prop: discrete derivatives are controlled by continuous derivatives}
  Let $x \in \tilde G_n$, then we have the estimates
  \begin{align*}
    |(\nabla_n)^1u(x)|_{g_x} & \leq C\|u\|_{C^1},\\
    |(\nabla_n)^2u(x)|_{g_x} & \leq C\|u\|_{C^2}.
  \end{align*}
  
\end{prop}

\begin{proof}[Proof of Proposition \ref{prop: discrete derivatives are controlled by continuous derivatives}]
   This is an immediate consequence of the previous proposition. Indeed, fix $u \in C^\beta_b(M)$ and $x\in \tilde G_n$. Then, we have
  \begin{align*}
    |(\nabla_n)^1u(x)|_{g_x} & \leq |(\nabla_n)^1u(x)-\nabla u(x)|_{g_x} +|\nabla u(x)|_{g_x},\;\; \beta \geq 1.\\
    |(\nabla_n)^2u(x)|_{g_x} & \leq |(\nabla_n)^2u(x)-\nabla^2 u(x)|_{g_x}+|\nabla^2 u(x)|_{g_x},\;\;\beta \geq 2.
  \end{align*}  
  Then, using the two estimates in Proposition \ref{prop:Appendix Cbeta implies bound on Taylor remainder}, we have
  \begin{align*}
    |(\nabla_n)^1u(x)|_{g_x} & \leq C\tilde h_n^{\min\{\beta-1,1\} } \|u\|_{C^\beta}+\|u\|_{C^1} \leq C\|u\|_{C^\beta},\;\;\beta \geq 1.\\
    |(\nabla_n)^2u(x)|_{g_x} & \leq C\tilde h_n^{\min\{\beta-2,1\}}\|u\|_{C^\beta}+\|u\|_{C^2} \leq C\|u\|_{C^\beta},\;\;\beta \geq 2.
  \end{align*}

\end{proof}

The next proposition yields a quantitative control on the ``continuity'' of $(\nabla_n)^{i}u$ in terms of the regularity of the original function $u$. As one may expect, if $\nabla u(x)$ and $\nabla^2u(x)$ are H\"older continuous in $M$, then $(\nabla_n)^1u$ and $(\nabla_n)^2u$ enjoy a respective modulus of ``continuity'' on $\tilde G_n$, this being uniform in $n$. 
\begin{prop}\label{prop:InterpolatedDerivativesAreRegular}
  Consider points $x,y \in M \setminus \tilde G_n$ and $\hat y$, $\hat x$ the corresponding points in $\tilde G_n$ with $d(x,\tilde G_n)=d(x,\hat x)$, $d(y,\tilde G_n)=d(y,\hat y)$, we have the following estimates with a universal $C$.

  \begin{enumerate}
  \item For $1\leq \beta\leq 2$,
  \begin{align*}
    & |(\nabla_n)_a^1 u(\hat x)-(\nabla_n)_a^1 u(\hat y)| \leq C\|u\|_{C^\beta}d(\hat x,\hat y)^{\beta-1}.
  \end{align*}
  \item For $2\leq \beta\leq3$,
  \begin{align*}
    & | (\nabla_n)_{ab}^2 u(\hat x)-(\nabla_n)_{ab}^2u(\hat y)| \leq C\|u\|_{C^\beta}d(\hat x,\hat y)^{\beta-2}.
  \end{align*}
  \end{enumerate}
  
\end{prop}

\begin{proof}
  If $\hat x=\hat y$ both inequalities are trivial and there is nothing to prove, so let us assume $\hat x,\hat y$ are two different points in $\tilde G_n$. In this case, and thanks to \eqref{eqFinDimApp: lambda def}, we have
  \begin{align}	
    d(\hat x,\hat y) \geq \lambda \tilde h_n.\label{eqAppendix:Different points are far apart}	
  \end{align}
  \emph{First estimate.} The triangle inequality yields,
  \begin{align*}  	
    & |(\nabla_n)_a^1 u(\hat x)-(\nabla_n)_a^1 u(\hat y)| \\
    & \leq |(\nabla_n)_a^1 u(\hat x)-\nabla_a u(\hat x)|+ |\nabla_a u(\hat x)-\nabla_a u(\hat y)|+ |\nabla_a u(\hat y)-(\nabla_n)_a^1 u(\hat y)|.	
  \end{align*}
  Let us estimate each of the three terms on the right. The middle term is straightforward,
  \begin{align*}	
    |\nabla_a u(\hat x)-\nabla_a u(\hat y)|\leq C\|u\|_{C^\beta}d(\hat x,\hat y)^{\beta-1}.
  \end{align*}	
  For the first and third term, we use the first part of Lemma \ref{lem:InterpolatedDerivativesAreCloseToRealDerivatives}, which says that
  \begin{align*}	
    & |(\nabla_n)_a^1 u(\hat x)-\nabla_a u(\hat x)| \leq C\|u\|_{C^\beta}\tilde h_n^{\beta-1},\\
    & |(\nabla_n)_a^1 u(\hat y)-\nabla_a u(\hat y)| \leq C\|u\|_{C^\beta}\tilde h_n^{\beta-1}.
  \end{align*}	  
  Using \eqref{eqAppendix:Different points are far apart} it follows that
  \begin{align*}	
    & |(\nabla_n)_a^1 u(\hat x)-\nabla_a u(\hat x)| \leq C\|u\|_{C^\beta}d(\hat x,\hat y)^{\beta-1},\\
    & |(\nabla_n)_a^1 u(\hat y)-\nabla_a u(\hat y)| \leq C\|u\|_{C^\beta}d(\hat x,\hat y)^{\beta-1}.
  \end{align*}	
  Combining the bounds for the three terms the first estimate follows.\\
  
  \noindent \emph{Second estimate.} As before, we start by breaking the difference in three parts, so
  \begin{align*}
    & |(\nabla_n)_{ab}^2 u(\hat x)-(\nabla_n)_{ab}^2 u(\hat y)|\\
    & \leq |(\nabla_n)_{ab}^2 u(\hat x)-\nabla_{ab}^2 u(\hat x)|+|\nabla_{ab}^2 u(\hat x)-\nabla_{ab}^2 u(\hat y)|+|\nabla_{ab}^2 u(\hat y)-(\nabla_n)_{ab}^2 u(\hat y)|.
  \end{align*}
  The middle term is bounded by
  \begin{align*}
    |\nabla_{ab}^2 u(\hat x)-\nabla_{ab}^2 u(\hat y)| \leq C\|u\|_{C^\beta}d(\hat x,\hat y)^{\beta-2}.  	  
  \end{align*}	  
  Next, thanks to the second part of Lemma \ref{lem:InterpolatedDerivativesAreCloseToRealDerivatives},
  \begin{align*}	
    & |(\nabla_n)_{ab}^2 u(\hat x)-\nabla_{ab}^2 u(\hat x)| \leq C\|u\|_{C^\beta}\tilde h_n^{\beta-2},\\
    & |(\nabla_n)_{ab}^2 u(\hat y)-\nabla_{ab}^2 u(\hat y)| \leq C\|u\|_{C^\beta}\tilde h_n^{\beta-2}.
  \end{align*}	  
  Using \eqref{eqAppendix:Different points are far apart} again, we conclude that
  \begin{align*}	
    & |(\nabla_n)_{ab}^2 u(\hat x)-\nabla_{ab}^2 u(\hat x)| \leq C\|u\|_{C^\beta}d(\hat x,\hat y)^{\beta-2},\\
    & |(\nabla_n)_{ab}^2 u(\hat y)-\nabla_{ab}^2 u(\hat y)| \leq C\|u\|_{C^\beta}d(\hat x,\hat y)^{\beta-2}.
  \end{align*}	  
  As in the previous case, the combined bounds for the three terms yields the estimate.
\end{proof}


\section{The Proof of Proposition \ref{prop:RegularityOfEnTnDistToGn}}\label{sec:ProofOfRegAtDistToBoundary}

This section is dedicated to proving Proposition \ref{prop:RegularityOfEnTnDistToGn}, which we re-record right here for the reader's convenience. 

\begin{prop*}
  Let $x \in M\setminus \tilde G_n$ and $u\in C^\beta$. There is a universal constant $C$ such that the following bounds hold. First, if $0\leq\beta<1$, 
  \begin{align*}	
    |\nabla (E_n^\beta\circ T_n)u(x)|\leq C\|u\|_{C^\beta}d(x,\tilde G_n)^{\beta-1}.
  \end{align*}	
  If $1\leq\beta<2$, we have
  \begin{align*}	
    |\nabla^2(E_n^\beta\circ T_n)u(x)|\leq C\|u\|_{C^\beta}d(x,\tilde G_n)^{\beta-2}.
  \end{align*}	
  Finally, if $2\leq\beta<3$, we have
  \begin{align*}	
    |\nabla^3(E_n^\beta\circ T_n)u(x)|\leq C\|u\|_{C^\beta}d(x,\tilde G_n)^{\beta-3}.
  \end{align*}	
 
\end{prop*}

\begin{proof}
  As done throughout Section \ref{sec: finite dimensional approximations}, for the sake of brevity we shall write $f=\pi_n^\beta u$.
  
  \emph{The case $\beta \in [0,1)$.} Since the sum defining $f$ is locally finite, we may differentiate term by term, which yields
  \begin{align*}
    \nabla f(x) & = \sum \limits_{k} u(\hat y_{n,k})\nabla \phi_{n,k}(x).
  \end{align*}
  Using \eqref{eqFinDim:Partition of Unity Differentiation Identity} with $i=1$ we may rewrite the above as
  \begin{align*}
    \nabla f(x) & = \sum \limits_{k}(u(\hat y_{n,k})-u(\hat x))\nabla \phi_{n,k}(x),\;\;\forall\;x\in M \setminus \tilde G_n.
  \end{align*}
  Then, since the only non-zero terms are those with $k \in K_x$ ($K_x$ was introduced in Lemma \ref{lem:CubesOnM}), 
  \begin{align*}	  
    |\nabla f(x)|_{g_x} & \leq \sum \limits_{k}|u(\hat y_{n,k})-u(\hat x)| |\nabla \phi_{n,k}(x)|_{g_x}\\
	& \leq N \sup\limits_{k \in K_x}|u(\hat y_{n,k})-u(\hat x)||\nabla \phi_{n,k}(x)|_{g_x}.
  \end{align*}
  For $k\in K_x$, using Remark \ref{rem: Whitney local diameter of balls}, and the H\"older regularity of $u$ one can check that
  \begin{align*}
    |u(\hat y_{n,k})-u(\hat x)| |\nabla \phi_{n,k}(x)|_{g_x} & \leq  C\|u\|_{C^\beta}d(x,\tilde G_n)^\beta d(x,\tilde G_n)^{-1}.
  \end{align*}
  From here, it follows that
  \begin{align*}
    |\nabla f(x)| & \leq C\|u\|_{C^\beta}d(x,\tilde G_n)^{\beta-1}.
  \end{align*}
  \emph{The case $\beta\in [1,2)$.} This time, we shall compute the Hessian $\nabla^2f$ using a local system of coordinates $\{x^{1},\ldots,x^{d}\}$. Then, for any pair of indices $a,b$ we have
  \begin{align*}
    \nabla^2_{ab}\phi = \partial_{x_ax_b}^2\phi- \sum \limits_{k=1}^d\Gamma_{ab}^k\partial_{x_k}\phi.
  \end{align*}
  Then
  \begin{align*}
    \nabla^2_{ab} f(x) & = \sum \limits_{k} \nabla^2_{ab}\left ( (u(\hat y_{n,k})+l( \nabla^1_n u (\hat y_{n,k}),\hat y_{n,k};x)) \phi_{n,k}(x) \right ).
  \end{align*}		
  We expand each term using the Leibniz rule, and conclude $\nabla^2_{ab}f(x)$ is equal to  
  \begin{align*}
    \textnormal{I}(x)+\textnormal{II}(x)+\textnormal{III}(x),
  \end{align*}
  where, for the sake of brevity, we have written
  \begin{align*}	
        \textnormal{I}(x) & = \sum \limits_{k} (u(\hat y_{n,k})+l( \nabla^1_n u (\hat y_{n,k}),\hat y_{n,k};x) )\nabla^2_{ab}\phi_{n,k}(x),\\
        \textnormal{II}(x) & = \sum \limits_{k} \nabla_a l( \nabla^1_n u (\hat y_{n,k}),\hat y_{n,k};x) \nabla_b \phi_{n,k}(x)+\sum \limits_{k} \nabla_b \phi_{n,k}(x) \nabla_a l( \nabla^1_n u (\hat y_{n,k}),\hat y_{n,k};x),\\					 	  
        \textnormal{III}(x) & =\sum \limits_{k} \nabla^2_{ab}l( \nabla^1_n u (\hat y_{n,k}),\hat y_{n,k};x) \phi_{n,k}(x).
  \end{align*}
  Since $x\in M \setminus \tilde G_n$, we can use \eqref{eqFinDim:Partition of Unity Differentiation Identity} with $i=1,2$ to obtain
  \begin{align*}
    & \sum \limits_{k} (u(\hat y_{n,k})+l( \nabla^1_n u (\hat y_{n,k}),\hat y_{n,k};x) )\nabla_{ab}^2\phi_{n,k}(x) \\
    & = \sum \limits_{k} (u(\hat y_{n,k})+l( \nabla^1_n u (\hat y_{n,k}),\hat y_{n,k};x) - u(\hat x) )\nabla_{ab}^2\phi_{n,k}(x),
  \end{align*}	
  and
  \begin{align*}
    & \sum \limits_{k} \nabla_a l( \nabla^1_n u (\hat y_{n,k}),\hat y_{n,k};x) \nabla_b \phi_{n,k}(x) \\
    & = \sum \limits_{k} \left (  \nabla_a l( \nabla^1_n u (\hat y_{n,k}),\hat y_{n,k};x) -\nabla_a l( \nabla^1_n u (\hat x),\hat x;x) \right ) \nabla_b \phi_{n,k}(x).
  \end{align*}  
  Let us bound each of these. The triangle inequality says
  \begin{align*}
    & |u(\hat y_{n,k})+l( (\nabla_n)^1u (\hat y_{n,k}),\hat y_{n,k};x)-u(x)|\\
    & \leq |u(\hat y_{n,k})+l( \nabla u (\hat y_{n,k}),\hat y_{n,k};x)-u(x)|+|l((\nabla_n)^1u (\hat y_{n,k}),\hat y_{n,k};x)-l( \nabla u (\hat y_{n,k}),\hat y_{n,k};x)|.
  \end{align*}	
  By Proposition \ref{prop:Appendix Cbeta implies bound on Taylor remainder} the first term on the right is no larger than $C\|u\|_{C^\beta}d(x,\hat y_{n,k})^\beta$. On the other hand, from the definition of $l(\cdot,\cdot;\cdot)$, it is immediate that the second term is no larger than
  \begin{align*}	
    & |(\nabla_n)^1u (\hat y_{n,k})- \nabla u (\hat y_{n,k})|_{g_{\hat y_{n,k}}}d(x,\hat y_{n,k}).	 
  \end{align*}
  Now, Lemma \ref{lem:InterpolatedDerivativesAreCloseToRealDerivatives} says that $|(\nabla_n)^1u (\hat y_{n,k})- \nabla u (\hat y_{n,k})|_{g_{\hat y_{n,k}}} \leq C\|u\|_{C^\beta}\tilde h_n^{\beta-1}$. Noting that $d(x,\hat y_{n,k})$ is no larger than $C\tilde h_n$ for $x\in P_{n,k}^*$, we obtain the estimate
  \begin{align*}
    |(\nabla_n)^1u (\hat y_{n,k})- \nabla u (\hat y_{n,k})|_{g_{\hat y_{n,k}}} \leq C\|u\|_{C^\beta}d(x,\hat y_{n,k})^{\beta-1}.	  
  \end{align*}	  
  Combining the last three estimates, we conclude that  
  \begin{align*}
    & |u(\hat y_{n,k})+l( (\nabla_n)^1u (\hat y_{n,k}),\hat y_{n,k};x)-u(x)|\leq C\|u\|_{C^\beta}d(x, \tilde G_n)^\beta,\;\;\forall\;x\in P_{n,k}^*.
  \end{align*}
  Using the estimates for the size of $\nabla^2\phi_{n,k}$, the above implies that
  \begin{align*}
    & |u(\hat y_{n,k})+l( (\nabla_n)^1u (\hat y_{n,k}),\hat y_{n,k};x)-u(x)||\nabla^2\phi_{n,k}(x)| \leq C\|u\|_{C^\beta}d(x,\tilde G_n)^{\beta-2},\;\;\forall\;x\in P_{n,k}^*.
  \end{align*}
  Finally, let us recall that the only nonzero terms appearing in the sum $\textnormal{I}(x)$ are those with $k\in K_x$ (i.e. $x \in P_{n,k}^*$), and that there at most $N$ of these terms. Then, we conclude that
  \begin{align*}
    \textnormal{I}(x) \leq C\|u\|_{C^\beta}d(x,\tilde G_n)^{\beta-2}.
  \end{align*}
  Let us now bound $\textnormal{II}(x)$, observe that	
  \begin{align*}
    \left | \nabla_a l( (\nabla_n)u (\hat y_{n,k}),\hat y_{n,k};x)-\nabla_a l( \nabla^1_n u (\hat x),\hat x;x)\right | \leq C\|u\|_{C^\beta}d(\hat x,\hat y_{n,k})^{\beta-1},\;\;\forall\;x\in P_{n,k}^*.
  \end{align*}
  Therefore
  \begin{align*}
    \sup \limits_{x\in P_{n,k}^*}\left | \nabla_a l( \nabla^1_n u (\hat y_{n,k}),\hat y_{n,k};x)-\nabla_a l( \nabla^1_n u (\hat x),\hat x;x)\right | & \leq C\|u\|_{C^\beta}d(x,\tilde G_n)^{\beta-1}.
  \end{align*}	 
  As before, the only nonzero terms adding up to $\textnormal{II}(x)$ are those with $x \in P_{n,k}^*$, therefore, the above bound implies that
  \begin{align*}
    \left | \sum \limits_{k} \nabla_a l( \nabla^1_n u (\hat y_{n,k}),\hat y_{n,k};x) \nabla_b \phi_{n,k}(x) \right | & \leq \sum \limits_{k\in K_x}\|u\|_{C^\beta}d(x,\tilde G_n)^{\beta-1}Cd(x,\tilde G_n)^{-1}\\
    & \leq CN \|u\|_{C^{\beta}}d(x,\tilde G_n)^{\beta-2}.	
  \end{align*}
  Therefore,
  \begin{align*}
    \textnormal{II}(x) \leq C \|u\|_{C^{\beta}}d(x,\tilde G_n)^{\beta-2}.	
  \end{align*}
  It remains to bound $\textnormal{III}(x)$. According to Proposition \ref{prop: local interpolation operators} and Proposition \ref{prop: discrete derivatives are controlled by continuous derivatives},
  \begin{align*}
    |\nabla_{ab}^2l( \nabla^1_n u (\hat y_{n,k}),\hat y_{n,k};x)| \leq C\|u\|_{C^\beta(M)}.
  \end{align*}
  Therefore, using \eqref{eqFinDim:Size Set Of Covering Cubes} (from Lemma \ref{lem:CubesOnM}) it follows that
  \begin{align*}
    \left | \sum \limits_{k} \nabla_{ab}^2l( \nabla^1_n u (\hat y_{n,k}),\hat y_{n,k};x) \phi_{n,k}(x) \right | & \leq C\sum \limits_{k\in K_x}\|u\|_{C^\beta}\phi_{n,k}(x)\\
	& \leq CN \|u\|_{C^\beta}.
  \end{align*}
  Gathering the last three estimates, we conclude that
  \begin{align*} 
    |\nabla^2_{ab}f(x)|\leq C\|u\|_{C^\beta}(d(x,\tilde G_n)^{\beta-2}+1).
  \end{align*}
  Moreover, since the indices $a,b$ were arbitrary, and since $d(x,\tilde G_n)$ is bounded from above for $x \in M \setminus \tilde G_n$ by a constant $C$, we conclude that
  \begin{align*} 
    |\nabla^2f(x)|\leq C\|u\|_{C^\beta}d(x,\tilde G_n)^{\beta-2}.
  \end{align*}
  \emph{The case $\beta \in [2,3)$.} The proof is entirely analogous to the previous case, and we only highlight the overall steps of the proof: as before, we pick a local system of coordinates $\{x_1,\ldots,x_d\}$ and use the identity
  \begin{align*}
    \nabla^3_{abc}f(x) & =  \sum \limits_{k} \nabla_{abc}^3\left [ \left ( u(\hat y_{n,k})+l( \nabla^1_n u (\hat y_{n,k}),\hat y_{n,k};x) + q(\nabla^2_nu (\hat y_{n,k} ),\hat y_{n,k};x) \right )\phi_{n,k}(x) \right ],
  \end{align*}
  which  holds for any three indices $a,b,$ and $c$. The expression on the right may be expanded via Leibniz rule, resulting in terms mixing various derivatives of $\phi_{n,k}$, $l( \nabla^1_n u (\hat y_{n,k}),\hat y_{n,k};\cdot )$, and $q(\nabla^2_n u (\hat y_{n,k}),\hat y_{n,k};\cdot )$. 

  It can then be checked that $\nabla^{3}_{abc}f(x)$ is given by a sum in $k$ of terms involving $\phi_{n,k}$ and values of $u$ on $\tilde G_n$ --in a manner analogue to the case $\beta \in [1,2)$. Now, to bound each of the resulting terms we will use \eqref{eqFinDim:Partition of Unity Differentiation Identity} with $i=1,2$ as before, but this time also with $i=3$. The bounds will follow by applying at difference instances Propositions \ref{prop:Appendix Cbeta implies bound on Taylor remainder} and \ref{prop: discrete derivatives are controlled by continuous derivatives}, as well as Lemma \ref{lem:InterpolatedDerivativesAreCloseToRealDerivatives}. All throughout, we will make us of the fact that the only non-zero terms appearing in the sums are those with $k\in P_{n,k}^*$. At the end, we arrive at the bound,
  \begin{align*} 
    |\nabla^3_{abc}f(x)|\leq C\|u\|_{C^\beta}(d(x,\tilde G_n)^{\beta-3}+1),
  \end{align*}	 
  which holds for any choice of the indices $a,b$ and $c$. This means that
  \begin{align*} 
    |\nabla^3 f(x)|\leq C\|u\|_{C^\beta}d(x,\tilde G_n)^{\beta-3},
  \end{align*}	 
  where we have used again that $d(x,\tilde G_n)$ is bounded from above for $x \in M \setminus \tilde G_n$.
\end{proof}

 
\bibliography{../refs}

\def\polhk#1{\setbox0=\hbox{#1}{\ooalign{\hidewidth
  \lower1.5ex\hbox{`}\hidewidth\crcr\unhbox0}}} \def\cprime{$'$}
  \def\cprime{$'$} \def\polhk#1{\setbox0=\hbox{#1}{\ooalign{\hidewidth
  \lower1.5ex\hbox{`}\hidewidth\crcr\unhbox0}}} \def\cprime{$'$}
  \def\cprime{$'$}
\begin{thebibliography}{10}

\bibitem{AlvarezLionsGuichardMorel-1993AxiomsImageProARMA}
Luis Alvarez, Fr{\'e}d{\'e}ric Guichard, Pierre-Louis Lions, and Jean-Michel
  Morel.
\newblock Axioms and fundamental equations of image processing.
\newblock {\em Arch. Rational Mech. Anal.}, 123(3):199--257, 1993.

\bibitem{BaIm-07}
Guy Barles and Cyril Imbert.
\newblock Second-order elliptic integro-differential equations: viscosity
  solutions' theory revisited.
\newblock {\em Ann. Inst. H. Poincar\'e Anal. Non Lin\'eaire}, 25(3):567--585,
  2008.

\bibitem{BarlesSouganidis-1998NewApproachFrontsARMA}
Guy Barles and Panagiotis~E. Souganidis.
\newblock A new approach to front propagation problems: theory and
  applications.
\newblock {\em Arch. Rational Mech. Anal.}, 141(3):237--296, 1998.

\bibitem{BaLe-2002TransitionProb}
Richard~F. Bass and David~A. Levin.
\newblock Transition probabilities for symmetric jump processes.
\newblock {\em Trans. Amer. Math. Soc.}, 354(7):2933--2953 (electronic), 2002.

\bibitem{Bell-1957}
Richard Bellman.
\newblock {\em Dynamic programming}.
\newblock Princeton University Press, Princeton, N. J., 1957.

\bibitem{BonyCourregePriouret-1966FellerSemiGroupOnDifferentialVariety}
Jean-Michel Bony, Philippe Courr{\`e}ge, and Pierre Priouret.
\newblock Sur la forme int\'egro-diff\'erentielle du g\'en\'erateur
  infinit\'esimal d'un semi-groupe de {F}eller sur une vari\'et\'e
  diff\'erentiable.
\newblock {\em C. R. Acad. Sci. Paris S\'er. A-B}, 263:A207--A210, 1966.

\bibitem{CaSi-09RegularityIntegroDiff}
Luis Caffarelli and Luis Silvestre.
\newblock Regularity theory for fully nonlinear integro-differential equations.
\newblock {\em Comm. Pure Appl. Math.}, 62(5):597--638, 2009.

\bibitem{Caff-1989InteriorEstimates}
Luis~A. Caffarelli.
\newblock Interior a priori estimates for solutions of fully nonlinear
  equations.
\newblock {\em Ann. of Math. (2)}, 130(1):189--213, 1989.

\bibitem{CaCa-95}
Luis~A. Caffarelli and Xavier Cabr{\'e}.
\newblock {\em Fully nonlinear elliptic equations}, volume~43 of {\em American
  Mathematical Society Colloquium Publications}.
\newblock American Mathematical Society, Providence, RI, 1995.

\bibitem{CaseChang2016fractionalGJMS}
Jeffrey~S Case and Sun-Yung Alice~Chang.
\newblock On fractional gjms operators.
\newblock {\em Communications on Pure and Applied Mathematics},
  69(6):1017--1061, 2016.

\bibitem{ChGo-2011FracLaplaceGeometry}
Sun-Yung~Alice Chang and Maria del Mar~Gonzalez.
\newblock Fractional laplacian in conformal geometry.
\newblock {\em Advances in Mathematics}, 226(2):1410 -- 1432, 2011.

\bibitem{ChDa-2012RegNonlocalParabolicCalcVar}
H{\'e}ctor Chang~Lara and Gonzalo D{\'a}vila.
\newblock Regularity for solutions of non local parabolic equations.
\newblock {\em Calc. Var. PDE}, 2012.
\newblock published online.

\bibitem{ChDa-2012NonsymKernels}
H{\'e}ctor Chang~Lara and Gonzalo D{\'a}vila.
\newblock Regularity for solutions of nonlocal, nonsymmetric equations.
\newblock {\em Ann. Inst. H. Poincar\'e Anal. Non Lin\'eaire}, 29(6):833--859,
  2012.

\bibitem{Cla1990optimization}
Frank~H Clarke.
\newblock {\em Optimization and nonsmooth analysis}, volume~5.
\newblock Siam, 1990.

\bibitem{Courrege-1965formePrincipeMaximum}
Philippe Courrege.
\newblock Sur la forme int{\'e}gro-diff{\'e}rentielle des op{\'e}rateurs de $
  c^{\infty}_k$ dans $c $ satisfaisant au principe du maximum.
\newblock {\em S{\'e}minaire Brelot-Choquet-Deny. Th{\'e}orie du Potentiel},
  10(1):1--38, 1965.

\bibitem{ElliottKalton-1972Book}
Robert~J. Elliott and Nigel~J. Kalton.
\newblock {\em The existence of value in differential games}.
\newblock American Mathematical Society, Providence, R.I., 1972.
\newblock Memoirs of the American Mathematical Society, No. 126.

\bibitem{EvansIshii-1984DiffGamesNonlinearPDE-ManMath}
L.~C. Evans and H.~Ishii.
\newblock Differential games and nonlinear first order {PDE} on bounded
  domains.
\newblock {\em Manuscripta Math.}, 49(2):109--139, 1984.

\bibitem{Evans-1980OnSolvingPDEAccretive}
Lawrence~C. Evans.
\newblock On solving certain nonlinear partial differential equations by
  accretive operator methods.
\newblock {\em Israel J. Math.}, 36(3-4):225--247, 1980.

\bibitem{Evans-1984MinMaxRepresentations}
Lawrence~C. Evans.
\newblock Some min-max methods for the {H}amilton-{J}acobi equation.
\newblock {\em Indiana Univ. Math. J.}, 33(1):31--50, 1984.

\bibitem{EvGa-92}
Lawrence~C. Evans and Ronald~F. Gariepy.
\newblock {\em Measure theory and fine properties of functions}.
\newblock Studies in Advanced Mathematics. CRC Press, Boca Raton, FL, 1992.

\bibitem{EvansSoug-84DiffGameRepresentation}
L.C. Evans and Panagiotis~E. Souganidis.
\newblock Differential games and representation formulas for solutions of
  {H}amilton-{J}acobi-{Isaacs} equations.
\newblock {\em Indiana Univ. Math. J.}, 33(5):773--797, 1984.

\bibitem{FeffermanIsraelLuli-2016}
Charles Fefferman, Arie Israel, and Garving~K Luli.
\newblock Interpolation of data by smooth non-negative functions.
\newblock {\em arXiv preprint arXiv:1603.02330}, 2016.

\bibitem{Fleming-1969CauchyProbNonlinFirstOrder-UsesMinMax}
Wendell~H. Fleming.
\newblock The {C}auchy problem for a nonlinear first order partial differential
  equation.
\newblock {\em J. Differential Equations}, 5:515--530, 1969.

\bibitem{FlemingRishel--1975Book}
Wendell~H. Fleming and Raymond~W. Rishel.
\newblock {\em Deterministic and stochastic optimal control}.
\newblock Springer-Verlag, Berlin-New York, 1975.
\newblock Applications of Mathematics, No. 1.

\bibitem{Friedman-1974TheCauchyProbFirstOrderEquations-IUMJ}
Avner Friedman.
\newblock The {C}auchy problem for first order partial differential equations.
\newblock {\em Indiana Univ. Math. J.}, 23:27--40, 1974.

\bibitem{GiOs-2008NLImagePro}
Guy Gilboa and Stanley Osher.
\newblock Nonlocal operators with applications to image processing.
\newblock {\em Multiscale Model. Simul.}, 7(3):1005--1028, 2008.

\bibitem{Heb-2000NonlinearAnalysisManifoldsBook}
Emmanuel Hebey.
\newblock {\em Nonlinear analysis on manifolds: Sobolev spaces and
  inequalities}, volume~5.
\newblock American Mathematical Soc., 2000.

\bibitem{Hsu-1986ExcursionsReflectingBM}
Pei Hsu.
\newblock On excursions of reflecting {B}rownian motion.
\newblock {\em Trans. Amer. Math. Soc.}, 296(1):239--264, 1986.

\bibitem{Isaacs-1965Book}
Rufus Isaacs.
\newblock {\em Differential games. {A} mathematical theory with applications to
  warfare and pursuit, control and optimization}.
\newblock John Wiley \& Sons, Inc., New York-London-Sydney, 1965.

\bibitem{JakobsenKarlsen-2006maxpple}
Espen~R. Jakobsen and Kenneth~H. Karlsen.
\newblock A ``maximum principle for semicontinuous functions'' applicable to
  integro-partial differential equations.
\newblock {\em NoDEA Nonlinear Differential Equations Appl.}, 13(2):137--165,
  2006.

\bibitem{KassRangSchwa-2013RegularityDirectionalINDIANA}
Moritz Kassmann, Marcus Rang, and Russell~W. Schwab.
\newblock H\"older regularity for integro-differential equations with nonlinear
  directional dependence.
\newblock {\em Indiana Univ. Math. J.}, To Appear, 2014.

\bibitem{Katsou-1995RepresentationDegParEq}
Markos~A. Katsoulakis.
\newblock A representation formula and regularizing properties for viscosity
  solutions of second-order fully nonlinear degenerate parabolic equations.
\newblock {\em Nonlinear Anal.}, 24(2):147--158, 1995.

\bibitem{KohnSerfaty-2006DeterministicGameMCM-CPAM}
Robert~V. Kohn and Sylvia Serfaty.
\newblock A deterministic-control-based approach to motion by curvature.
\newblock {\em Comm. Pure Appl. Math.}, 59(3):344--407, 2006.

\bibitem{KohnSerfaty-2010DeterministicGamesFullyNonlinearPDE-CPAM}
Robert~V. Kohn and Sylvia Serfaty.
\newblock A deterministic-control-based approach to fully nonlinear parabolic
  and elliptic equations.
\newblock {\em Comm. Pure Appl. Math.}, 63(10):1298--1350, 2010.

\bibitem{KoikeSwiech-2013RepFormulaIntegroPDE-IUMJ}
Shigeaki Koike and Andrzej {\'S}wi{\polhk{e}}ch.
\newblock Representation formulas for solutions of {I}saacs integro-{PDE}.
\newblock {\em Indiana Univ. Math. J.}, 62(5):1473--1502, 2013.

\bibitem{Krylov-2015RatesFinDiffIsaacs}
N.~V. Krylov.
\newblock On the rate of convergence of finite-difference approximations for
  elliptic {I}saacs equations in smooth domains.
\newblock {\em Comm. Partial Differential Equations}, 40(8):1393--1407, 2015.

\bibitem{KuoTru-2007NewMaxPrinciplesIUMJ}
Hung-Ju Kuo and Neil~S. Trudinger.
\newblock New maximum principles for linear elliptic equations.
\newblock {\em Indiana Univ. Math. J.}, 56(5):2439--2452, 2007.

\bibitem{Lee-1997RiemannianManifoldsBook}
John~M. Lee.
\newblock {\em Riemannian manifolds}, volume 176 of {\em Graduate Texts in
  Mathematics}.
\newblock Springer-Verlag, New York, 1997.
\newblock An introduction to curvature.

\bibitem{LiLi2006FullyYamabeBoundary}
YanYan Li and Aobing Li.
\newblock A fully nonlinear version of the yamabe problem on manifolds with
  boundary.
\newblock {\em Journal of the European Mathematical Society}, 8(2):295--316,
  2006.

\bibitem{LiPaVa-88unpublished}
P~L Lions, G~Papanicolaou, and S~R~S Varadhan.
\newblock Homogenization of {H}amilton-{J}acobi equations.
\newblock {\em unpublished}, circa 1988.

\bibitem{LionsSouganidis-1985DiffGamesDerivOfSol}
P.-L. Lions and P.~E. Souganidis.
\newblock Differential games, optimal control and directional derivatives of
  viscosity solutions of {B}ellman's and {I}saacs' equations.
\newblock {\em SIAM J. Control Optim.}, 23(4):566--583, 1985.

\bibitem{Pucci-1965SuLeEquazioniEstremanti}
Carlo Pucci.
\newblock Su le equazioni ellittiche estremanti.
\newblock {\em Rend. Sem. Mat. Fis. Milano}, 35:12--20, 1965.

\bibitem{Schw-10Per}
Russell~W. Schwab.
\newblock Periodic homogenization for nonlinear integro-differential equations.
\newblock {\em SIAM J. Math. Anal.}, 42(6):2652--2680, 2010.

\bibitem{Schw-12StochCPDE}
Russell~W. Schwab.
\newblock Stochastic homogenization for some nonlinear integro-differential
  equations.
\newblock {\em Communications in Partial Differential Equations},
  38(2):171--198, 2012.

\bibitem{SchwabSilvestre-2014RegularityIntDiffVeryIrregKernelsAPDE}
Russell~W. Schwab and Luis Silvestre.
\newblock Regularity for parabolic integro-differential equations with very
  irregular kernels.
\newblock {\em Anal. PDE}, 9(3):727--772, 2016.

\bibitem{Silv-2011DifferentiabilityCriticalHJ}
Luis Silvestre.
\newblock On the differentiability of the solution to the {H}amilton-{J}acobi
  equation with critical fractional diffusion.
\newblock {\em Adv. Math.}, 226(2):2020--2039, 2011.

\bibitem{Souganidis-1985MaxMinRep}
Panagiotis~E. Souganidis.
\newblock Max-min representations and product formulas for the viscosity
  solutions of {H}amilton-{J}acobi equations with applications to differential
  games.
\newblock {\em Nonlinear Anal.}, 9(3):217--257, 1985.

\bibitem{Stei-71}
E.~M. Stein.
\newblock {\em Singular Integrals and Differentiability Properties of
  Functions}.
\newblock Princeton University Press, Princeton, 1971.

\end{thebibliography}
\bibliographystyle{plain}
\end{document}